\documentclass[a4paper, 11pt, reqno,final]{amsart}
\usepackage{fullpage}
\usepackage[utf8]{inputenc}
\usepackage[english]{babel}
\usepackage{amsmath}
\usepackage{amssymb}
\usepackage{amsthm}
\usepackage{xifthen}
\usepackage{holtpolt}
\usepackage{xspace}
\usepackage[usenames]{xcolor}
\usepackage[titletoc,page]{appendix}
\usepackage[authoryear]{natbib}
\usepackage{enumitem}

\newtheorem{thm}{Theorem}[section]
\newtheorem{lemma}[thm]{Lemma}
\newtheorem{cor}[thm]{Corollary}
\newtheorem{prop}[thm]{Proposition}

\theoremstyle{definition}
\newtheorem{example}[thm]{Example}
\newtheorem{defn}[thm]{Definition}
\newtheorem{rem}[thm]{Remark}
\newtheorem{example*}[thm]{Example}
\usepackage{hyperref}
\hypersetup{
    final,
    colorlinks=true,
    raiselinks=true, % Zeilenumbrüche in links
    linkcolor=[rgb]{.0, .0, .8}, % Interne Links
    citecolor=[rgb]{.0, .0, .8}, % Quellenverweise
    filecolor=black, % Dateilinks
    urlcolor=black, % Adressen im web
}

\usepackage{mathtools} \mathtoolsset{showonlyrefs}

\usepackage{chngcntr}
\counterwithin{equation}{section}

\usepackage[normalem]{ulem}

\newcommand{\inn}{\operatorname{Int}}
\newcommand{\relint}{\operatorname{relint}}
\newcommand{\br}{\bigskip}
\DeclareMathOperator*{\essinf}{ess\:inf}
\DeclareMathOperator*{\argmin}{arg\,min}
\newcommand{\e}[1]{\mathbb{E} \left[#1\right]}
\newcommand{\p}[1]{\mathbb{P} \left(#1\right)}
\newcommand{\labelonce}[1]{\stepcounter{equation}\tag{\theequation}\label{#1}}
\newcommand{\R}{\mathbb{R}}
\renewcommand{\l}{\lambda}
\renewcommand{\d}{\delta}

\renewcommand{\a}{\alpha}
\renewcommand{\e}{\varepsilon}

\renewcommand{\b}{\beta}

\newcommand{\lv}{\lvert}
\newcommand{\rv}{\rvert}

\newcommand{\lb}{\left(}
\newcommand{\rb}{\right)}
\newcommand{\lr}{\left[}
\newcommand{\rr}{\right]}
\renewcommand{\O}{\mathcal{O}}
\newcommand{\supp}{\operatorname{supp}}

\renewcommand{\Re}{\operatorname{Re}}
\renewcommand{\Im}{\operatorname{Im}}

\newcommand{\N}{\mathbb{N}}
\renewcommand{\P}{\mathbb{P}}
\newcommand{\mc}{\mathcal}
\newcommand{\m}{\mathbf m}
\newcommand{\mb}{\mathbf}
\newcommand{\Hess}{\operatorname{Hess}}
\newcommand{\C}{\mc C}

\usepackage{tikz}
\usepackage{pgfplots}
\usetikzlibrary{intersections}
\usetikzlibrary{patterns}
\pgfplotsset{compat=1.13}
\usepgfplotslibrary{fillbetween}

\usepackage{subfig}

\newcommand{\plotrestricted}[2][]{
	\subfloat[$m_3=$ #2#1]{
		\begin{tikzpicture}
		\begin{axis}[
		xmin=0,
		xmax=1,
		ymin=0,
		ymax=1,
		samples=160,
		xtick={0, 1},
		ytick={1},
		x=\textwidth*0.42,
		y=\textwidth*0.42
		]
		\pgfmathsetmacro{\low}{#2}
		\pgfmathsetmacro{\high}{exp(ln(#2)/3)}
		\addplot [black][domain=\low:\high] {sqrt(x*#2)};
		\addplot [black][domain=\low:\high] {(x+1)/2-sqrt(-3*x*x/4+x*(1/2+#2)+1/4-#2)};
		\addplot [gray][domain=0:1] {x};
		\addplot [gray][domain=0:1] {x^2};
		\end{axis}
		\end{tikzpicture}
	}
}

\begin{document}
    \title{Random Moment Problems under Constraints}
    
        \author{Holger Dette$^*$}
        \address{$^*$Department of Mathematics, Ruhr University Bochum, 44780 Bochum, Germany}
        \email{holger.dette@ruhr-uni-bochum.de}
 
        \author{Dominik Tomecki$^\dag$}
        \address{$^\dag$Department of Mathematics, Ruhr University Bochum, 44780 Bochum, Germany}
        \email{dominik.tomecki@ruhr-uni-bochum.de}
 
        \author{Martin Venker$^\ddag$}
        \address{$^\dag$Department of Mathematics, Ruhr University Bochum, 44780 Bochum, Germany}
        \email{martin.venker@ruhr-uni-bochum.de}


\begin{abstract}

We investigate moment sequences of probability measures on $E\subset\R$ under constraints of certain moments being fixed. This corresponds to studying sections of $n$-th moment spaces, i.e.~the spaces of moment sequences of order $n$. By equipping these sections with the uniform or more general probability distributions, we manage to give for large $n$ precise results on the (probabilistic) barycenters of moment space sections and the fluctuations of random moments around these barycenters. The measures associated to the barycenters belong to the Bernstein-Szeg\H{o} class and show strong universal behavior. We prove Gaussian fluctuations and moderate and large deviations principles. Furthermore, we demonstrate how fixing moments by a constraint leads to breaking the connection between random moments and random matrices.

\end{abstract}
    \maketitle

    % Section
\section{Introduction}
Classical moment problems on the real line pose the question whether a given sequence of real numbers is the moment sequence of a positive Borel measure with support in a prescribed set $E\subset\R$ and whether such a measure, if it exists, is unique. Most notable are the Hamburger, Stieltjes and Hausdorff moment problems, which correspond to the sets $E=\R$, $E=\R_+:=[0,\infty)$ and $E$ being a compact interval, respectively. Solutions to these moment problems have been known for a long time.

In the classical moment problems one is thus interested in \textit{all} possible moment sequences. In contrast, the random moment problem asks how a \textit{typical} moment sequence looks like. To make this precise, let us denote by $\mathcal{P}(E)$ the set of all Borel probability measures on a Borel set $E \subset \mathbb{R}$ possessing moments of all order, and by $m_j(\mu) := \int x^j \, d\mu(x)$ the $j$-th moment of a measure $\mu \in \mathcal{P}(E)$. The set
    \begin{align*}
        \mathcal{M}_n(E) := \{(m_1(\mu), \ldots, m_n(\mu)) \mid \mu \in \mathcal{P}(E)\}
    \end{align*}
    is called \emph{$n$-th moment space}. It is a convex set in $\R^n$ with positive Lebesgue measure. Beginning with \cite{karsha1953}, \cite{karlin1966} and \cite{krenud1977}, geometric aspects of $\mathcal M_n(E)$ have been investigated in many works. A probabilistic investigation was initiated by \cite{chakemstu1993}, who equipped $\mc M_n([0, 1])$ with the uniform distribution and studied the behavior of a fixed number of the now random moments as the dimension $n$ converges to infinity. They observed that in high dimension such a random moment sequence concentrates near the moment sequence of the arcsine distribution
    \begin{align}
        \mu^0(dx) := \frac{1}{\pi \sqrt{x(1-x)}}\, dx. \label{arcsine}
    \end{align}
		Thus the moment sequence of the arcsine distribution may be seen as a probabilistic barycenter of the moment space $\mc M_n([0,1])$.
		More precisely, let $(m_1^{(n)},\dots,m_n^{(n)})$ be drawn from the uniform distribution on $\mc M_n([0,1])$ and $l\in\N$ be fixed. Then, as $n\to\infty$,
		\begin{align}
		(m_1^{(n)},\dots,m_l^{(n)})\to (m_1(\mu^0),\dots,m_l(\mu^0))\label{LLNold}
		\end{align}
in probability.
    Moreover, \cite{chakemstu1993} proved that 
\begin{align}
\sqrt n((m_1^{(n)},\dots,m_l^{(n)})- (m_1(\mu^0),\dots,m_l(\mu^0)))\label{CLTold}
\end{align}
converge in distribution to a multivariate normal distribution as $n\to\infty$.
Later, \cite{gamloz2004} showed fluctuations on the scales of moderate and  large deviations principles also for $E=[0,1]$. \cite{detnag2012} studied special distributions on the unbounded moment spaces $\mc M_n(\R_+)$ and $\mc M_{2n+1}(\R)$ and also proved central limit theorems. They found moment sequences of a certain Marchenko-Pastur distribution and Wigner's semicircle distribution, respectively, replacing the one of the arcsine measure. 

However, to speak of a typical moment sequence it would be desirable to have a certain universality in the sense that the limiting moment sequences should not strongly depend on the probability distribution which has been put on the moment space $\mc M_n(E)$. The question of universality is even more prominent in the case of unbounded $E$ as then $\mc M_n(E)$ is unbounded itself and thus can not carry the ``natural'' uniform distribution. Therefore recently \cite{DTV} gave a unifying view on the random moment problem by identifying classes of distributions on $\mc M_n(E)$ in all three cases $E=[a,b],\R_+,\R$ that admit universal behavior: On $\R_+$ and $\R$, the moment sequences in these classes always converge in the large $n$ limit to the moment sequences of the Marchenko-Pastur distributions and those of the semicircle distributions, respectively. For $E=[a,b]$, the arcsine measure was found in \cite{DTV} to be rather a special member of the universal family of Kesten-McKay (or free binomial) distributions than being universal itself.

The occurrence of the three families of Kesten-McKay, Marchenko-Pastur and semicircle distributions is somewhat curious and suggests a connection to random matrix theory, where these distributions appear as limits of empirical spectral distributions for the so-called Jacobi, Laguerre and Wigner ensembles, respectively. They are characterized as equilibrium measures to certain external fields on $\R$. 
 We will illuminate the connection of random moments to equilibrium measures, orthogonal polynomials and random matrix theory in the course of this paper.
    
		From a geometric point of view, moment spaces are interesting convex sets that admit special parametrizations and therefore allow for a detailed analysis. For instance, the moment space $\mc M_n([0,1])$ is a convex body contained in $[0,1]^n$. It is very far from other convex bodies like balls, hypercubes or cross-polytopes regarding the strength of the dependence between coordinates. For example, to obtain Gaussian fluctuations of random points in the three mentioned classical convex bodies, one needs to involve a growing number of coordinates due to the rather mild dependence structure. In striking contrast to that, \eqref{CLTold} with $l=1$ shows that even the first coordinate $m_1^{(n)}$ shows for $n\to\infty$ Gaussian fluctuations under the uniform distribution on $\mc M_n([0,1])$, indicating a very strong dependence between all moments. The analysis of moment spaces generally uses special independent coordinates that unravel the dependence structure of moments.
		
    In the present work we investigate the behavior of random moment sequences when certain moments are fixed by a constraint. This corresponds to the question of a typical moment sequence when some moments are a priori known. From a more geometric perspective, we investigate slices of the moment space $\mc M_n(E)$, thereby providing a better insight in its shape. In particular, constraining moment sequences enables to study different regions in moment spaces, see e.g.~Figure \ref{figure1} below, where we show the moment space $\mathcal{M}_2([0, 1])$ and constrained moment spaces obtained by fixing the third moment $m_3$. Of special interest are existence and structure of (probabilistic) barycenters of these constrained moment spaces as well as volume and more refined probabilistic questions like fluctuation laws of the random constrained moment sequences.

Let us now make things precise. Under a constraint $\mc C$ we understand a finite collection of integer indices $1 \le i_1 < \ldots < i_k, k\in\N$ and corresponding values $c_{i_1}, \ldots, c_{i_k}$ which we denote as $\mc C:=\{m_{i_1}=c_{i_1},\dots,m_{i_k}=c_{i_k}\}$. For instance, the constraint $\mc C=\{m_1=c_1\}$ means fixing the expectation, whereas $ \mc C=\{m_1=c_1,m_2=c_2\}$ also fixes the variance. We allow for $k=0$, corresponding to unconstrained moment spaces.  We now want to examine the moments of probability measures whose $i_j$-th moment is given by $c_{i_j},\,j=1,\dots,k$. 
    
    \begin{defn}[Admissible Constraint]\label{admissible}
        For a constraint $\mc C$ we denote by
        \begin{align*}
            \mathcal{P}^{\mc C}(E) := \{\mu \in \mathcal{P}(E) \mid m_{i_j}(\mu) = c_{i_j} \, \forall \, 1 \le j \le k\}
        \end{align*}
        the space of probability measures on $E$ fulfilling constraint $\mc C$ and  by
        \begin{align*}
            \mathcal{M}^{\mc C}_n(E) := \{(m_1(\mu), \ldots, m_n(\mu)) \mid \mu \in \mathcal{P}^{\mc C}(E)\}
        \end{align*}
        the constrained $n$-th moment space. A constraint $\mc C$ is called \emph{admissible} for $E$, if the intersection $\mathcal{M}^{\mc C}_{i_k}(E) \cap \inn \mathcal{M}_{i_k}(E)$ is nonempty, where here and lateron, $\text{Int}$ denotes the interior.
    \end{defn}
    
    For ease of notation,  we will assume throughout the article that $\mc C$ is an admissible constraint for $E$ with indices $i_1, \ldots, i_k$ and corresponding values $c_{i_1}, \ldots, c_{i_k}$. Note that the notion of admissibility depends on the set $E$.
    
    Let us illustrate how a constraint restricts the moment space and allows to study different regions of the space. Figure~\ref{figure1} shows $\mc M_2^{\mc C}([0,1])$ for $\mc C=\{m_3=c_3\}$ (encircled in black) inside of $\mc M_2([0,1])$ (encircled in grey) for two values of $c_3$. In the first plot $c_3=0.3125$ which is the third moment of the arcsine measure. The second plot is for $c_3=0.1$.
    
        \begin{figure}[ht]
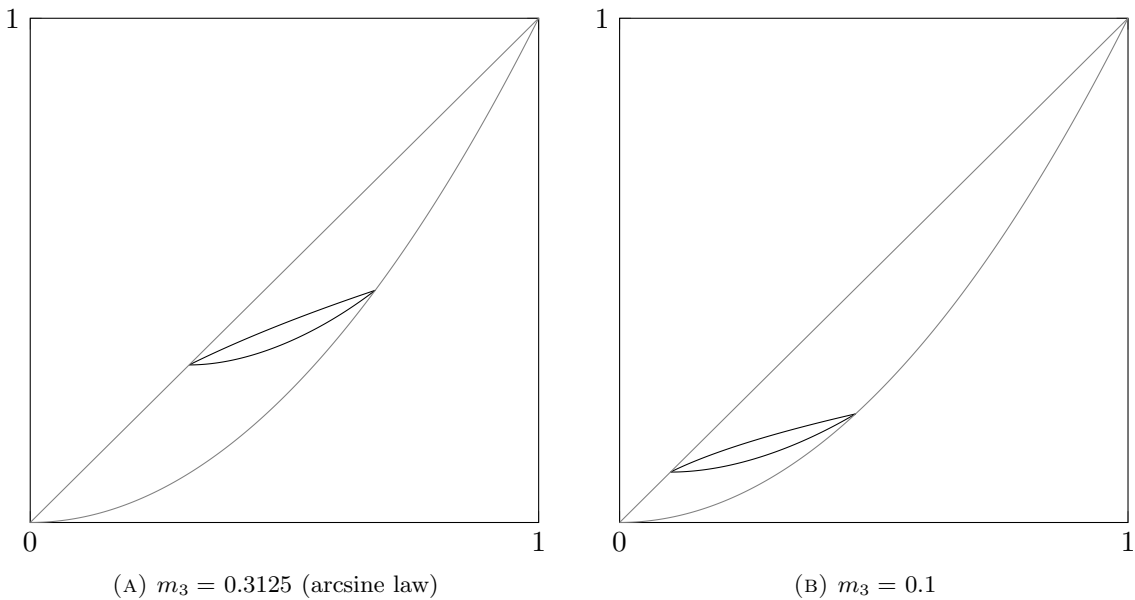

    	\centering
    	 	
    	\plotrestricted[ (arcsine law)]{0.3125}
     	\plotrestricted{0.1}
    	
    	\caption{Visualizations of $\mc M_2^{\mc C}([0,1])$ under constraints on $m_3$.}\label{figure1}
    \end{figure}

    This article is structured as follows. In the next section, we will first consider the uniform distribution on $\mc M_n^{\mc C}(E)$ for a bounded interval $E$. We will see that a uniformly distributed random moment sequence converges for $n\to\infty$ in the sense of \eqref{LLNold}, where the limiting measure has a density w.r.t.~the arcsine measure \eqref{arcsine}. However, it is in general not the equilibrium measure under the constraint $\mc C$ to the uniform external field as might be expected from the discussion above. While an equilibrium measure under a constraint $\mc C$ should be obtained by minimizing Voiculescu's free entropy (from free probability theory, see e.g.~\cite[Chapter 22 and references therein]{Handbook}) to the arcsine measure over the set of measures compatible with the constraint, we rather find the Kullback-Leibler divergence or relative entropy from classical probability playing a major role. This parting of the ways of the random moment problems and random matrix theory will be explained in detail in Section \ref{sec_universality} and has its roots in the fact that introducing a constraint breaks the asymptotic equivalence of the spectral measure encountered in the random moment problem and the empirical spectral measure encountered in random matrix theory. Nevertheless, the limiting measures we find in this paper belong to a famous class of measures as well, the so-called Bernstein-Szeg\H{o} class. Section \ref{sec_uniform} also provides a computation of the volume of the constrained moment spaces, a central limit theorem as well as moderate and large deviations principles. Section~\ref{sec_general_dist} takes a broader approach and defines more general classes of distributions on the constrained moment spaces, in particular on the unbounded moment spaces $\mc M_n^{\mc C}(\R_+)$ and $\mc M_n^{\mc C}(\R)$. For generic densities on these spaces, we identify the universal structures of the limiting moment sequences and give results on the fluctuations around these limits on several scales. We provide a detailed and extensive analysis, even in those cases where the random moment sequences do not have a single limit but rather concentrate around a finite set of limit points. Sections \ref{sec_limit}, \ref{sec_LDP}  and \ref{sec_CLT} are devoted to the proofs of the results in Section \ref{sec_general_dist}, whereas the results in Section \ref{sec_uniform} are proved in Section \ref{sec_application_uniform}.

    % Section
\section{Uniformly distributed random moment sequences}\label{sec_uniform}
 In this section we study the case of random moment sequences uniformly distributed in $\mc M_n^{\mc C}(E)$ for $E$ being a compact interval. Without loss of generality we will choose $E=[0,1]$, since by the linear transformation $t \mapsto a+t(b-a)$ the results for $[0,1]$ can be transferred to any interval $[a,b]$. Note that although Section \ref{sec_general_dist} covers more general distributions on the moment spaces, the results in the present section are more explicit and not easily deduced from the ones of Section \ref{sec_general_dist}.

  Recall our convention that $\mc C$ is an admissible constraint (see Definition \ref{admissible}) of the form $m_{i_1}=c_{i_1},\dots,m_{i_k}=c_{i_k}$. The constrained moment space $\mathcal{M}_n^{\mc C}([0, 1])$ can be identified canonically with the set of the $(n-k)$-dimensional vectors of the unconstrained moments $(m_j,1\leq j\leq n,j\not=i_1,\dots,i_k)$. The set of unconstrained moments is a convex and compact subset of the $(n - k)$-dimensional unit cube and has due to admissibility of $\mc C$ non-zero Lebesgue measure. 
	Pushing forward the $n-k$-dimensional Lebesgue measure from the unconstrained moments to $\mathcal{M}_n^{\mc C}([0, 1])$, we can equip $\mathcal{M}_n^{\mc C}([0, 1])$ with the uniform distribution, which allows us to investigate the behaviour of a ``typical'' moment sequence on the constrained moment spaces.\\
	
	Throughout this section let	$(m_1^{(n)},\dots,m_n^{(n)})$ be drawn from the uniform distribution on $\mathcal{M}_n^{\mc C}([0, 1])$. Here and lateron, we will tacitly assume that all random variables are defined on a common probability space such that we can speak of almost sure convergence. \\
	
	Our first result is a law of large numbers that identifies the limiting moment sequence to which the random moment sequence converges.
    
    \begin{thm}\label{uniform_lln}
        For any $l\in\N$ we have as $n\to\infty$
        \begin{align*}
            (m_1^{(n)}, \ldots, m_l^{(n)}) \to \big(m_1(\mu^{\mc C}), \ldots, m_l(\mu^{\mc C})\big)\quad \text{ a.s.,}
        \end{align*}
        where $\mu^{\mc C}$ is a probability measure on $[0,1]$ of the form
        \begin{align}
            \mu^{\mc C}(dx)=\frac{1}{S_{i_k}(x) \sqrt{x(1-x)}}dx.\label{representation_mu^C}
        \end{align}
        Here $S_{i_k}$ is a polynomial of degree at most $i_k$ that is strictly positive on the interval $[0,1]$. Furthermore, $\mu^{\mc C}$ is the unique probability measure that minimizes the Kullback-Leibler divergence
        \begin{align}
            \mathcal{K}(\mu^0 | \mu) := 
                \begin{cases}
                    \int \log \frac{d\mu^0}{d\mu} \, d\mu^0 &, \mu^0 \ll \mu \\
                    \infty &, \text{else}
                \end{cases} \label{kullback}
        \end{align}
        among all probability measures $\mu\in\mathcal{P}^{\mc C}([0, 1])$, where $\mu^0$ is the arcsine distribution defined in \eqref{arcsine}.
    \end{thm}
    \begin{rem}\strut\label{rem_not_eq_measure}
		\begin{enumerate}
		\item The measure $\mu^\C$ belongs to the so-called Bernstein-Szeg\H{o} class on $[0,1]$ which consists of measures of the form
            \begin{align*}
                \mu(dx)=\frac{(x(1-x))^{\pm \frac12}}{S(x)}dx,
            \end{align*}
            where $S$ is a polynomial strictly positive on the interval $[0,1]$. They play a key role in the theory of orthogonal polynomials and possess many useful properties, e.g.~explicit formulae for their orthogonal polynomials. For details and references we refer to \cite[\S 2.6]{Szego}. The connection between moments and orthogonal polynomials is well-known and at the roots of both theories. Theorem~\ref{uniform_lln} shows that these important measures of orthogonal polynomials theory are also central for moment spaces in the sense that their moments provide the probabilistic barycenters of sections of the moment space. More general members of the Bernstein-Szeg\H o class will be found in Section~\ref{sec_general_dist} when discussing random moment sequences over the unbounded spaces $\R_+$ and $\R$.
			\item For the unconstrained random moment problem, the probabilistic barycenter is given by the moment sequence of the arcsine measure $\mu^0$. It is the equilibrium measure on $[0,1]$ to the external field $Q_{ex}=0$, where we recall that the equilibrium measure to an external field $Q_{ex}:E\to\R$ is the unique Borel probability measure $\mu$ on $E$ such that
\begin{align}
\int_E \int_E \log\lv t-s\rv^{-1}d\mu(t)d\mu(s)+\int_E Q_{ex}(t)d\mu(t)\label{eq_minimization}
\end{align}
is minimal, see e.g.~\cite{SaffTotik}. The Marchenko-Pastur distributions are (if there is no atom at 0) the equilibrium measures on $E=\R_+$ to $Q_{ex}(t)=\frac{t}{z_2}-\frac{z_1-z_2}{z_2}\log t$ for some constants $0<z_2\leq z_1$. Likewise, the semicircle distributions are obtained as equilibrium measures on $E=\R$ to external fields that are quadratic polynomials with positive leading coefficient. One might thus expect $\mu^\C$ from Theorem \ref{uniform_lln} to be the minimizer of a constrained equilibrium problem. However, we can deduce from representation \eqref{representation_mu^C} that for constraints $\mc C$ with $\mu^0\notin\mathcal{P}^{\mc C}([0, 1])$, $\mu^{\mc C}$ is not the solution of the constrained equilibrium problem
			\begin{align}
\inf_{\mu\in\mathcal{P}^{\mc C}([0, 1])}\int_{[0,1]} \int_{[0,1]} \log\lv t-s\rv^{-1}d\mu(t)d\mu(s).\label{eq_minimization_constrained}
\end{align}
As variational calculus shows, a solution $\mu\in\mathcal{P}^{\mc C}([0, 1])$ of \eqref{eq_minimization_constrained} has to fulfill the Euler-Lagrange equations
\begin{align}
2 \int\log\lv t-s\rv^{-1}d\mu(s)+\sum_{j=1}^k\l_j t^{i_j}\begin{cases}	
																											=c,\quad t\in\supp(\mu),\\
																											\geq c,\quad t\notin\supp(\mu),
																											\end{cases}
\end{align} 
for some $c$ and Lagrange multipliers $\l_1,\dots,\l_k$. Thus the moment constraint $\mc C$ leads to the appearance of polynomial external fields. Equilibrium measures to such fields are well-understood, see e.g.~\cite[Theorem 1.38]{Deiftetal99}.  The support of the equilibrium measure $\mu$ consists of finitely many intervals in $[0,1]$ with non-empty interior and it has a density of the form
\begin{align*}
\mu(dx)=T(x)\prod_{a\in \textup{HE}}\frac{1}{\sqrt{\lv x-a\rv}}\prod_{b\in\textup{SE}}\sqrt{\lv x-b\rv}1_{\supp(\mu)}(x)dx.
\end{align*}
Here $\text{HE}\subset\{0,1\}$ is the (possibly empty) set of the so-called hard edges and SE the set of so-called soft edges, and $T$ is a polynomial which is strictly positive on $\supp(\mu)$.  As an example, the minimizer of~\eqref{eq_minimization_constrained} under the constraint $\mc C=\{m_1=c_1,m_2=c_2\}$ is for $c_1\in(0,1)$ and $c_2>c_1^2$ small enough the semicircle distribution $\mu(dx)=c \sqrt{4\beta-(x-c_1)^2}1_{[a,b]}(x)dx$ for some $a,b,\beta,c>0$. This is very different from the measure $\mu^\C$, which we will give in~\eqref{measure_fixed_mean_and_variance} below. This implies that already the supports of the equilibrium measure under a constraint and of the measure $\mu^{\mc C}$ can be very different, as the former can be arbitrarily small while the latter is always $[0,1]$. We will explain this phenomenon in Section \ref{sec_universality}.
		\item  The Kullback-Leibler divergence $\mathcal{K}(\mu^0 | \mu)$ is also called relative entropy. It is always non-negative and can be understood as a distance measure for probability distributions. Indeed, although not being a metric itself, one has Pinsker's inequality $\|\mu-\nu\|_{TV}\leq\sqrt{\mc K(\nu|\mu)/2}$ (see \cite{Csiszar}) and thus convergence of the Kullback-Leibler divergence to 0 implies convergence in the total variation distance $\|\cdot\|_{TV}$. It also appears as the rate function of a large deviations principle in Sanov's theorem. Note however, that in typical encounters of $\mc K$, the minimization is in the first argument. The connection of $\mc K$ to random moment problems was first observed in \cite{gamloz2004}. There the authors call it \textit{reversed} Kullback-Leibler divergence because of the minimization in the second argument.
		\end{enumerate}
		\end{rem}
    
\begin{example}\label{example_uniform}
The proof of Theorem \ref{uniform_lln} is constructive in the sense that the polynomial $S_{i_k}$ can be computed from the constraint $\mc C$. Here we provide some examples, for the full details we refer to Section \ref{sec_limit}. For $\C_1:=\{m_1=c_1\}$ (fixed mean) with $c_1\in(0,1)$, the limiting measure is
    \begin{align*}
       \mu^{\C_1}(dx)= \frac{c_1(1 - c_1)}{\pi \sqrt{x(1-x)}((1 - 2c_1)x + c_1^2)}dx.
    \end{align*}
For a constraint $\C_2:=\{m_1=c_1,m_2=c_2\}$ (fixed mean and fixed variance) the admissibility condition reads $c_1^2 < c_2 < c_1$ and the limiting measure is
\begin{align}
\mu^{\C_2}(dx)=\frac{c_1(1 - c_1)(c_2 - c_1^2)(c_1 - c_2)}{\pi\sqrt{x(1-x)}\big((c_1 - c_2)^2(x - c_1)^2 + (c_2 - c_1^2)^2 x(1 - x)\big)}dx.\label{measure_fixed_mean_and_variance}
\end{align}
Finally, if only the second moment is fixed, i.e.~$\C_3:=\{m_2=c_2\}$, then $\mu^{\C_3}$ is given by \eqref{measure_fixed_mean_and_variance} with $c_1$ being the unique maximizer of the function 
\begin{align*}
c_1\mapsto \frac{(c_2 - c_1^2)(c_1 - c_2)}{c_1(1 - c_1)}
\end{align*}
on $[c_2,\sqrt{c_2}]$.
\end{example}

The limiting measure $\mu^{\mc C}$ allows for an effective description of the volume of $\mc M_n^{\mc C}([0,1])$. \cite[Theorem IV.6.2]{karlin1966} gave an expression of the volume of the unconstrained space $\mc M_n([0, 1])$  in terms of gamma functions, which can by a direct application of Stirling's formula be written as
\begin{align}
{\textup{vol}_n(\mathcal{M}_n([0, 1]))}=\prod_{m=1}^{n-1}\frac{\Gamma(m)^2}{\Gamma(2m)}=2^{-n^2} \Big(\frac{\pi e}{n}\Big)^{n/2} n^{-1/8} (1+\O(1)),	\label{volume_momspace}
\end{align}
vol$_n$ denoting the $n$-dimensional Lebesgue measure. We also need to introduce some notation. Let $m_1,\dots,m_l\in\mc M_l([0,1])$. Then the possible range of $m_{l+1}$ such that $(m_1,\dots,m_{l+1})\in\mc M_{l+1}([0,1])$ holds, is an interval, say $[m_{l+1}^-,m_{l+1}^+]$. In the next theorem, $m_{i_k+1}^\pm(\mu^{\mc C})$ will denote the numbers $m_{i_k+1}^\pm$ to the first $i_k$ moments of the measure $\mu^{\mc C}$ introduced in Theorem \ref{uniform_lln}.
    \begin{thm}\label{volume}
    	 As $n\to\infty$
        \begin{align*}
            \frac{\textup{vol}_{n - k}(\mathcal{M}_n^{\mc C}([0, 1]))}{\textup{vol}_n(\mathcal{M}_n([0, 1]))} &= \frac{\big((m_{i_k + 1}^+ - m_{i_k + 1}^-)(\mu^{\mc C})\big)^{n - i_k} n^{k/2} (2\pi)^{k/2}/\sqrt{d^{\mc C}}}{4^{-i_k(i_k-3)/2} 4^{-i_k n} (\sqrt{\pi/4})^{i_k}}(1 + o(1)),\\
            &=\begin{cases}
            \O\lb e^{-cn}\rb\quad \text{for some }c>0&,\quad\text{ if }\mu^0\notin \mc P^{\mc C}(E),\\
            \O\lb n^{k/2}\rb\quad&,\quad\text{ if }\mu^0\in \mc P^{\mc C}(E),
            \end{cases}
        \end{align*}
        where $d^{\mc C}$ is a constant independent of $n$.

    \end{thm}
		
		Theorem \ref{volume} can be interpreted as follows: If the constraint $\mc C$ is chosen such that the arcsine distribution does not lie in $\mathcal{P}^{\mc C}([0, 1])$, then the relative volume of the constrained moment space goes to zero exponentially fast. If the arcsine distribution lies in $\mathcal{P}^{\mc C}([0, 1])$, then the relative volume grows polynomially. The growth of the relative volume in the latter case reflects the fact that the volume of the (unconstrained) moment space decreases with increasing dimension.

    Another way to interpret this is that almost all intersections of the moment space with hyperplanes orthogonal to the standard basis vectors are ``small''. The only large intersections are those corresponding to a constraint in which $m_{i_j} = m_{i_j}(\mu^0)$ holds for all $1 \le j \le k$. \\

    We now return to the probabilistic setting and turn to fluctuations of the random moment sequence around the limiting measure $\mu^{\mc C}$ on several scales. For fluctuations of order $1/\sqrt{n}$ we observe Gaussian laws.
    
    \begin{thm}[Central Limit Theorem] \label{uniform_clt}
        Let $l \ge i_k$. We have as $n\to\infty$
        \begin{align}
            \sqrt{n}((m_1^{(n)}, \ldots, m_l^{(n)}) - (m_1(\mu^{\mc C}), \ldots, m_l(\mu^{\mc C})) \to \mathcal{N}(0, \Sigma_l)\label{CLT}
        \end{align}
				in distribution, where $\mathcal{N}(0, \Sigma_l)$ denotes the multivariate normal distribution with mean 0 and $\Sigma_l$ is an $l\times l$ matrix of rank $l - k$ that consists of zeros in the columns and rows $i_1, \ldots, i_k$. 
						\end{thm}
    We remark in passing that the condition $l \ge i_k$ is only assumed to state the result, especially the form of the matrix $\Sigma_l$, in a convenient way. Of course, the vector $(m_1, \ldots, m_l)$ also exhibits Gaussian fluctuations for any $l < i_k$. For a precise definition of the matrix $\Sigma_l$ we first need certain preliminaries like a parametrization of the moment space. For the reader's convenience, the expression of $\Sigma_l$ is given in \eqref{Sigma_l}.
    
   For our results on fluctuations on larger scales, we first recall the notion of an LDP.
    Let $\mathcal{X}$ be a topological space and $\mathbb{P}_n$ a sequence of probability measures on $\mathcal{X}$, equipped with the Borel $\sigma$-field. $\mathbb{P}_n$ is said to satisfy a \emph{large deviations principle} (LDP) with speed $a_n \to \infty$ and rate function $I$ if $I: \mathcal{X} \to [0, \infty]$ is lower semicontinuous and for all measurable $\Gamma \subset \mathcal{X}$ the inequalities
    \begin{align}
        -\inf \limits_{x \in \inn \Gamma} I(x) \le \liminf \limits_{n \to \infty} \frac{1}{a_n} \log \mathbb{P}_n(\Gamma) \le \limsup  \limits_{n \to \infty} \frac{1}{a_n} \log \mathbb{P}_n(\Gamma) \le -\inf \limits_{x \in \overline{\Gamma}} I(x)\label{def_LDP}
    \end{align}
    hold. $I$ is called a \emph{good rate function}, if its level sets $I^{-1}((-\infty, \alpha])$ are compact for all $\alpha \ge 0$. In the following,  $\mathcal{X} = \mathbb{R}^l$ will always be equipped with the Euclidean topology. 
    
    The next result shows that for fluctuations between the order $1/\sqrt n$ and order $1$, the Gaussian distribution from Theorem \ref{uniform_clt} still provides the leading order asymptotics. Therefore such deviations are called moderate. Throughout this paper $A^t$ denotes the transpose of the matrix $A$.
    
    \begin{thm}[Moderate Deviations Principle] \label{uniform_mdp}
        Let $l \ge i_k$ and $(a_n)_n$ be a sequence with $a_n\to \infty$ and $a_n = o(n^{1/2})$ as $n\to\infty$. Then the sequence
        \begin{align*}
            a_n((m_1^{(n)}, \ldots, m_l^{(n)}) - (m_1(\mu^{\mc C}), \ldots, m_l(\mu^{\mc C})))
        \end{align*}
        satisfies an LDP with speed $\frac{n}{a_n^2}$ and good rate function
        \begin{align*}
            I(x) :=
            \begin{cases}
                \frac{1}{2} x^t \Sigma_l x &, x_{i_1} = \ldots = x_{i_k} = 0 \\
                \infty &,\text{else}
            \end{cases},
        \end{align*}
        where $\Sigma_l$ is the covariance matrix from Theorem \ref{uniform_clt}.
    \end{thm}
   
    Note that for $l < i_k$ the moderate deviations principle (MDP) can be obtained from the previous result by an application of the contraction principle. This remark extends to all LDPs in this article.
    
    Finally, for fluctuations of order $1$ we can see that the random moment sequences satisfy a large deviations principle with a rate function that is universal for all $\mc C$ up to an additive constant. To state it, recall that given $(m_1,\dots,m_l)\in \mc M_l([0,1])$, the possible range of $m_{l+1}$ such that $(m_1,\dots,m_{l+1})\in\mc M_{l+1}([0,1])$ holds, is an interval $[m_{l+1}^-,m_{l+1}^+]$.
    \begin{thm}[Large Deviations Principle]\label{uniform_ldp}
        Let $l \ge i_k$ and denote
        \begin{align}
            I(m_1,\dots,m_l) :=
            \begin{cases}
                -\log(m_{l + 1}^+ - m_{l + 1}^-) &, (m_1,\dots,m_l) \in \mc M_l^{\mc C}([0,1]) \\
                \infty &, \text{else}
            \end{cases}. \label{idef}
        \end{align}
        Then $(m_1^{(n)},\ldots, m_l^{(n)})$ satisfies an LDP with speed $n$ and good rate function $I(\,\cdot\,) - I(m_1(\mu^{\mc C}),\dots,m_l(\mu^{\mc C}))$.
    \end{thm}
    
    The previous theorem shows an LDP for the first $l$ random moments. By a projective limit argument it is possible to prove an LDP for the infinite random moment sequence. Since the Hausdorff moment problem is determinate, i.e.~each probability measure on a compact interval is determined uniquely by its moment sequence, we can equivalently obtain an LDP for probability measures. 
    
    \begin{thm}[Functional Large Deviations Principle]\label{uniform_ldp_functional}
        Let $(\mu^{(n)})_n$  be a sequence of random probability measures in $\mc P([0,1])$ such that the vector of corresponding moments $(m_1(\mu^{(n)}), \ldots, m_n(\mu^{(n)}))$ is uniformly distributed in $\mc M_n^{\mc C}([0,1])$ for each $n$. Then $(\mu^{(n)})_n$ satisfies an LDP with speed $n$ and good rate function
        \begin{align*}
            I(\mu) :=
            \begin{cases}
                \mathcal{K}(\mu^0 \mid \mu) - \mathcal{K}(\mu^0 \mid \mu^{\mc C}) &, \mu^0 \ll \mu \text{ and } \mu \in \mathcal{P}^{\mc C}([0, 1]) \\
                \infty &, \text{else}
            \end{cases}
        \end{align*}
        on the space $\mathcal{P}([0, 1])$ equipped with the weak topology, where the Kullback-Leibler divergence $\mathcal{K}$ is defined in \eqref{kullback}.
    \end{thm}
    
    \begin{rem}
        We have left the particular choice of the distribution of $\mu^{(n)}$ open and only demand that the distribution of its first $n$ moments is uniform on the $n$-th restricted space. This leaves many possible choices for the distribution of $\mu_n$. For example, $\mu^{(n)}$ could be chosen as the upper or lower principal representation of a random moment vector $(m_1^{(n)}, \ldots, m_n^{(n)})$, see e.g.~\cite{Skibinsky86}. There is a nice constructive approach to $\mu^{(n)}$ using spectral measures of Jacobi matrices which we will discuss in Section \ref{sec_universality}.
    \end{rem}

        % Section
\section{General distributions on constrained moment spaces and universality} \label{sec_general_dist}
While Section \ref{sec_uniform} dealt exclusively with the uniform distribution on $\mc M_n^{\mc C}([0,1])$, we will now turn to more general distributions on $\mc M_n^{\mc C}(E)$ for $E=[0,1]$ as well as the unbounded $E=\R_+$ and $E=\R$. Particular emphasis will be on universal behavior of the random moment sequences.
    
The first important observation towards general distributions on $\mc M_n^{\mc C}(E)$ consists in the fact that the ordinary moments $m_1,\dots,m_n$ are not good coordinates to define probability measures on the moment spaces $\mc M_n(E)$. Indeed, the ordinary moments are not independent but strongly dependent and moreover the possible range for the $l+1$-th moment given the first $l$ moments decreases exponentially in $l$ (cf.~\cite{karsha1953}). For these reasons, it is desirable to have a new system of independent coordinates that scale with the available moment range. For the bounded moment space $\mc M_n([0,1])$, such coordinates have been introduced by Skibinsky in a series of papers \cite{skibinsky1967,skibinsky1968,skibinsky1969}.  To ease notation, let us denote a vector in bold with a subscript indicating dimension, e.g.
\begin{align}
\m_n:=(m_1,\dots,m_n).
\end{align}
    For $m_j^+ \ne m_j^-$, the $j$-th canonical moment is defined as
    \begin{align}
       p_j := \frac{m_j -  m_j^-}{m_j^+ - m_j^-}.\label{def_p_j}
    \end{align}
		They are left undefined if $m_j^+ = m_j^-$ in which case the corresponding moment sequence lies on the boundary of $\mc M_n([0,1])$. The canonical moment simply is the relative position of the ordinary moment in the available section.
  In fact, this construction of the canonical moments induces a smooth bijection
    \begin{align*}
        \varphi_n^{[0, 1]}:\left\{
        \begin{array}{ccc}
            (0, 1)^n  & \to       &   \inn \mathcal{M}_n([0, 1])        \\
            (p_1, \ldots, p_n) & \mapsto   & (m_1, \ldots, m_n) 
        \end{array}
        \right.
    \end{align*}
    between the open unit cube $(0, 1)^n$ and the interior of the $n$-th moment space.

    On the moment space $\mathcal{M}_n([0, \infty))$ the range of the $(n + 1)$-th moment is a half-open interval $[m_{n + 1}^-, \infty)$ and a good system of coordinates is given by
    \begin{align}
        z_j := \frac{m_j - m_j^-}{m_{j - 1} - m_{j - 1}^-} \label{def_z_j}
    \end{align}
    with $m_0 := 1$, $m_0^- := 0$. These parameters are well-defined for all moment sequences in the interior of the moment space. Indeed, they yield a smooth bijection
    \begin{align*}
        \varphi_n^{[0, \infty)}:\left\{
        \begin{array}{ccc}
            (0, \infty)^n  & \to       &    \inn \mathcal{M}_n([0, \infty))       \\
             (z_1, \ldots, z_n) & \mapsto   & (m_1, \ldots, m_n).
        \end{array}
        \right.
    \end{align*}
    
    Finally, the moment space $\mathcal{M}_n(\mathbb{R})$ can be parametrized by the recurrence coefficients of orthogonal polynomials. Let for a vector $\m_n\in\mc M_n(\R)$, $\mu\in \mc P(\R)$ be a measure with the first $n$ moments given by $\m_n$. It is well-known that to each $\mu\in\mc P(\R)$ there is a unique sequence of monic polynomials $(P_j)_j$ with $\operatorname{deg}P_j=j$ that are orthogonal in $L^2(\R,\mu)$. If $\mu$ has a finite support then the sequence of orthogonal polynomials is finite. $P_j$ is determined by the first $2j-1$ moments of $\mu$ which shows that different measures representing a moment sequence $\m_{n}$ have the same (first) orthogonal polynomials. These polynomials satisfy a three-term recursion
    \begin{align}
        P_j(x) = (x - \alpha_j) P_{j- 1}(x) - \beta_{j - 1} P_{j - 2}(x), \label{pol_recursion}
    \end{align}
    where $\alpha_j \in \mathbb{R}$ and $\beta_{j - 1} > 0$. Moreover, by Favard's theorem (cf. Theorem~ I.4.4 in \cite{chihara1978}) each sequence of parameters $\alpha_j\in \R$, $\beta_j>0$ yield a sequence of monic orthogonal polynomials $P_j$. This induces smooth bijections

    \begin{align*}
        \varphi_{2n}^{\mathbb{R}}&:\left\{
        \begin{array}{ccc}
           (\mathbb{R} \times (0, \infty))^n  & \to       &   \inn \mathcal{M}_{2n}(\mathbb{R})        \\
             (\alpha_1, \beta_1, \alpha_2, \ldots, \alpha_n, \beta_n) & \mapsto   & (m_1, \ldots, m_{2n})
        \end{array}
        \right.\\
        \varphi_{2n + 1}^{\mathbb{R}}&:\left\{
        \begin{array}{ccc}
           (\mathbb{R} \times (0, \infty))^n \times \mathbb{R}   & \to       &  \inn \mathcal{M}_{2n + 1}(\mathbb{R})        \\
             (\alpha_1, \beta_1, \alpha_2, \ldots, \alpha_n, \beta_n, \alpha_{n + 1})  & \mapsto   &(m_1, \ldots, m_{2n + 1})
        \end{array}
        \right. .
    \end{align*}
    
    In order to unify notation in the three different cases $E=[0,1],\R_+,\R$, we set
        \begin{align}\label{def_D_j}
            y_j := 
            \begin{cases}
                p_j & ,E = [0, 1] \\
                z_j & ,E = [0, \infty) \\
                \alpha_{(j + 1)/2} & ,E = \mathbb{R} \text{ and $j$ odd} \\
                \beta_{j /2} & ,E = \mathbb{R} \text{ and $j$ even}
            \end{cases}
            \qquad \qquad 
            D_j := 
            \begin{cases}
                (0, 1) & ,E = [0, 1] \\
                (0, \infty) & ,E = [0, \infty) \\
                \mathbb{R} & ,E = \mathbb{R} \text{ and $j$ odd} \\
                (0, \infty) & ,E = \mathbb{R} \text{ and $j$ even}
            \end{cases},
        \end{align}
        so that in all three cases
        \begin{align}
            \varphi_{n}^E&:\left\{
            \begin{array}{ccc}
              D_1 \times \cdots \times D_n    & \to       &    \inn \mathcal{M}_{n}(E)      \\
                (y_1, \ldots, y_n)  & \mapsto   & (m_1, \ldots, m_{n})
            \end{array}
            \right.\label{def_can_coord}
        \end{align}
        yields a smooth bijection. Throughout this paper we will call the $y_j$ \emph{canonical coordinates}. 
				
We will give the Jacobians of parameterizations of the constrained moment spaces in Lemma \ref{lemma_Jacobian} below. At this stage it suffices to note the remarkable result
\begin{align}
\left\lv \det\left[\frac{\partial \varphi_{n}^{[0,1]}(p_1,\dots,p_n)}{\partial (p_1,\dots,p_n)}\right]\right\rv=\prod \limits_{j = 1}^n (p_j(1 - p_j))^{n - j}\,,\label{Skibinsky}
\end{align}
which is due to \cite[p.~1f]{skibinsky1967}. Consequently, we see that for the uniform distribution on $\mc M_n([0,1])$, the random canonical coordinates $p_1^{(n)},\dots,p_n^{(n)}$ have a density on $[0,1]^n$ proportional to the r.h.s.~of \eqref{Skibinsky}. In other words, the canonical coordinates are independent and $p_j^{(n)}$ has a beta$(n-j+1,n-j+1)$ distribution. As we are interested mostly in the first $l$ moments as $n\to\infty$, we have $j\ll n$ and thus the canonical moments are independent and nearly identically distributed. The class of  distributions introduced in \cite{DTV} and adapted here to the constrained spaces, generalizes from the uniform distribution on $\mc M_n([0,1])$ in three ways. Firstly, we include distributions on the unbounded spaces $\mc M_n(\R_+)$ and $\mc M_n(\R)$. Secondly, we generalize from the beta distribution to rather arbitrary densities while keeping the two key properties of independence and (nearly) identical distribution. And lastly, we allow for different densities for even and odd coordinates. This originates in the different roles played by even and odd moments. As even moments are always positive and contain some information about the size of the support of the measure, odd moments contain information about symmetries.\\

    We can now introduce our general class of distributions on the constrained moment spaces. Recall that $\mc C$ is always assumed to be an admissible constraint for $\mc P(E)$.\\

Let $V_1: \overline{D_1} \to \mathbb{R},\,\dots,\,V_{i_k+2}: \overline{D_{i_k+2}} \to \mathbb{R}$ be continuous functions satisfying in the cases $E=\R_+,\R$ the integrability conditions 
				\begin{align}\label{int_conditions}
				V_j(y_j)\geq \begin{cases}
				(2+\varepsilon)\log\lv y_j\rv,\quad j=1,2,&\quad E=\R_+,\\
				(1+\varepsilon)\log\lv y_j\rv,\quad j=1,&\quad E=\R,\\
				(2+\varepsilon)\log\lv y_j\rv,\quad j=2,&\quad E=\R,
				\end{cases}
						\end{align}
						for $\lv y_i\rv$ large enough and some $\varepsilon>0$.
Denote by 
\begin{align}
\m_n^{\mc C}:=(m_j,1\leq j\leq n,j\not=i_1,\dots,i_k)\label{m_n^C}
\end{align}
the vector of unconstrained moments.
Finally we set  for notational convenience $$V_{i_k + 2j-1} := V_{i_k+1}\text{  and  }V_{2i_k + 2j} := V_{i_k+2},\quad j\geq2$$ and define the Borel probability measure $\mathbb{P}^{\C}_{n, E, V}$ on $M_n^{\mc C}(E)$ by $\mathbb{P}^{\mc C}_{n, E, V}(\partial \mc M_n^{\mc C}(E))=0$ and on the  interior of $\mc M_n^{\mc C}(E)$ by  the density
        \begin{align}
            P^{\mc C}_{n,E, V}(\m_n)  :=\frac{1}{Z^{\mc C}_{n,E,V}} \exp\Big(-n \sum \limits_{j = 1}^n V_j(y_j(\m_n))\Big)\label{def_P}
        \end{align}
        w.r.t.~the $(n-k)$-dimensional Lebesgue measure on $\mc M_n^{\mc C}(E)$ as defined at the beginning of Section \ref{sec_uniform}. Here $y_j=y_j(\m_n),\, j=1,\dots,n,$ are the canonical coordinates associated to the moment sequence $\m_n=(m_1, \ldots, m_n)$ and 
				\begin{align}
				Z^{\mc C}_{n,E,V}:=\int_{\mc M_n^{\mc C}(E)}\exp\Big(-n \sum \limits_{j = 1}^n V_j(y_j(\m_n))\Big)d\m_n^{\mc C}\label{normalizing}
				\end{align}
				is the normalizing constant. $\P_{n,E,V}^{\mc C}$ is for the empty constraint $\C=\emptyset$ determined by $V_1$ and $V_2$ and is precisely the class of distributions found in \cite{DTV} showing universal behavior and classical limiting measures from random matrix theory or free probability theory. In the presence of a constraint influencing the first $i_k$ moments, the results of this section will show that $\P_{n,E,V}^{\mc C}$ is an appropriate class of distributions to study $\mc M_n^\C(E)$. More precisely, we will find universal behavior within this class for generic functions $V_1,\dots,V_{i_k+2}$, given by special  Bernstein-Szeg\H{o} measures. We remark in passing that the product form of the density of the first $i_k$ moments is not necessary for observing the universal limits and could be extended to some function of the form $\exp(-nV(m_1,\dots,m_{i_k}))$. However, the product form is convenient for notation and the class we consider is exhaustive in the sense that any universal limit law of the extended class can be observed in the smaller class. 
				
As it is not immediate that $0<Z^{\mc C}_{n,E,V}<\infty$, we formulate the following lemma which is proved at the beginning of Section \ref{sec_limit}.

\begin{lemma} \label{lemma_density}
For $V_1,\dots,V_{i_k+2}$ satisfying \eqref{int_conditions}, $\mathbb P^{\mc C}_{n,E,V}$ is a probability measure on $\mc M_n^{\mc C}(E)$.
\end{lemma}

For the rest of this section, we will assume that $(m_1^{(n)},\dots,m_n^{(n)})$ has distribution $\mathbb P^{\mc C}_{n,E,V}$. Our first result in this setting is a large deviations principle which holds without any further assumptions. 
Define the functions
$W_j: D_j \to \mathbb{R}, j=1,\dots,n$ by 
        \begin{align}\label{def_W}
            W_j(y_j) := 
            \begin{cases}
                V_j(y_j) - \log(y_j(1 - y_j))    &, E = [0, 1], \\
                V_j(y_j) - \log(y_j)             &, E = [0, \infty), \\
                V_j(y_j)                        &, E = \mathbb{R}, j \text{ odd}, \\
                V_j(y_j) - \log(y_j)           &, E = \mathbb{R}, j \text{ even}.
            \end{cases}
        \end{align}

    \begin{thm}[Large Deviations Principle]\label{general_dist_ldp}
        The vector $(m_1^{(n)},\dots,m_l^{(n)})$ satisfies a large deviations principle with speed $n$ and good rate function $I_l^{\mathcal{C}, E}( \, \cdot \,) - \inf \limits_{\m_l\in\mc M_l^{\mc C}(E)} I_l^{\mathcal{C}, E}(\m_l)$, where
        \begin{align}\label{def_I}
            I_l^{\mathcal{C}, E}(m_1, \ldots, m_l) :=
            \begin{cases}
                \sum \limits_{j = 1}^l W_j(y_j) &, (m_1, \ldots, m_l) \in \textup{Int} \mathcal{M}_l^{\mc C}(E) \\
                \infty &, \text{else}
            \end{cases}.
        \end{align}
    \end{thm}
    
Usually a large deviations principle implies a law of large numbers, provided that the rate function has a unique minimizer. In the unconstrained case, all canonical coordinates are independent and uniqueness of the minimizer of the rate function reduces to uniqueness of the minimizers of $W
_1$ and $W_2$. Hence uniqueness of a minimizing moment sequence is actually a univariate problem. In contrast to that, as a consequence of the constraint $\mc C$, in general the first $i_k$ canonical coordinates are strongly dependent and thus uniqueness of a minimizer is in general a truly multivariate problem, involving simultaneously $W_1,\dots, W_{i_k+2}$ and $\mc C$. Let us illustrate this with an example for $\mc M_n^{\mc C}([0,1])$.
\begin{example}
    The constraint $\mc C=\{m_1=c\}$ for some $c\in(0,1)$ does not induce any dependencies, since $m_1=p_1$. The simplest possible, yet instructive, constraint is $\mc C=\{m_2=c\}$ for $c\in(0,1)$. From \eqref{def_p_j} we deduce $p_1=m_1$ and $p_2=\frac{c-p_1^2}{p_1(1-p_1)}$. Thus, changing coordinates to $p_1,p_3,p_4,\dots,p_n$, the density $P_{n,[0,1],V}^{\mc C}$ can using a computation similar to \eqref{Skibinsky} be expressed as
    \begin{align}
    \frac{1}{Z^{\mc C}_{n,E,V}} &\exp\lr{-n \lb W_1(p_1)+W_2\lb\frac{c-p_1^2}{p_1(1-p_1)}\rb\rb}\rr1_{[c,\sqrt c]}(p_1)\label{example_eq1}\\
    &\times\exp\lr -2\log (p_1(1-p_1))-2\log(cp_1+cp_1^2-c^2-p_1^3)+4\log(p_1(1-p_1))\rr\label{example_eq2}\\
    &\times\exp\lb{-n \sum \limits_{j = 3}^n W_j(p_j)-j\log((p_j)(1-p_j))}\rb.\label{example_eq3}
    \end{align}
    The indicator function in \eqref{example_eq1} stems from the fact that $m_2=c$ implies $m_1=p_1\in[c,\sqrt{c}]$. We will see a general version of \eqref{example_eq1}-\eqref{example_eq3} in Corollary \ref{corollary_density} below.

All terms in \eqref{example_eq2} are sub-leading and do not contribute much to the uniqueness problem. \eqref{example_eq3} factorizes and thus uniqueness of minimization over $p_3,p_4,\dots$ reduces to $W_3$ and $W_4$ separately. However, equation \eqref{example_eq1} shows that the uniqueness problem for $p_1$ involves $W_1$, $W_2$ and $\mc C$. Note that in this simple example minimization over $p_1$ is still one-dimensional, albeit non-trivial. The constraint $\mc C=\{m_3=c\}$ would yield dependent $p_1,p_2$, and $p_3$ could be written as a rational function of $p_1$ and $p_2$. Note also that specifying to $V_1=\dots=V_4=0$, i.e.~to the uniform distribution considered in Section \ref{sec_uniform}, does not simplify things much.
\end{example}
From the previous example it is far from obvious why the uniform distribution on $\mc M_n^{\mc C}([0,1])$ concentrates on a unique moment sequence as $n\to\infty$. This uniqueness will for general $V_1,\dots,V_{i_k+2}$ no longer be true. Fortunately, we can prove equally strong results without requiring uniqueness of the multivariate minimization problem. The assumptions on $V_1,\dots,V_{i_k+2}$ used lateron are formulated in the definition below. As a preparation, let
\begin{align}
\mb y_n^{\mc C}:=(y_j,1\leq j\leq n,j\not=i_1,\dots,i_k).\label{parametrization0}
\end{align}
be the vector of canonical coordinates corresponding to $\m_n^{\mc C}$. It is shown in Lemma~\ref{lemma_Jacobian} of Section \ref{sec_limit} that $\mb y_n^{\mc C}$ allows for a parametrization of $\mc M_n^{\mc C}(E)$. At this stage, it suffices to note that knowing $\mb y_n^{\mc C}$ determines $\m_n$ because the ``missing coordinates'' $y_{i_j},j=1,\dots,k$ are determined by $\mb y_n^{\mc C}$ and the constraint $\mc C$.

\begin{defn}[Generic $V_1,\dots,V_{i_k+2}$]\label{def_generic}
Let $V_j\in C^2(D_j)$, $j=1,\dots,i_k+2$ satisfy \eqref{int_conditions}, if $E=\R_+$ or $E=\R$. We call $V_1,\dots,V_{i_k+2}$ \textit{generic}, if the following three conditions are satisfied:
\begin{itemize}
	\item(unique univariate minimizers) $W_{i_k+1}$ and $W_{i_k+2}$ have unique minimizers $y_{i_k+1}^*\in D_{i_k+1}$ and $y_{i_k+2}^*\in D_{i_k+2}$, respectively,
	\item(finitely many multivariate minimizers) $\mb y_{i_k}^{\mc C}\mapsto I_{i_k}^{\C,E}(\m_{i_k}(\mb y_{i_k}^{\mc C}))$ has finitely many minimizers
        \begin{align*}
            \mathbf{y}^{*,1}, \ldots, \mathbf{y}^{*,p}\in\prod\limits_{\substack{j=1\\j\not=i_1,\dots,i_k}}^{i_k}D_j\,,
        \end{align*}
        where the set $D_j$ has been defined in \eqref{def_D_j} and $I_{i_k}^{\C,E}$ in \eqref{def_I}.
    
	\item(non-degeneracy of minimizers) $W_j''(y_j^*)\not=0$, $j=i_k+1,i_k+2$ and 
        \begin{align*}
            \det\text{Hess}^{\mc C}\lb\sum_{j=1}^{i_k}W_j\rb(\mb y^{*,q})\not=0,\ q=1,\dots,p.
        \end{align*}
        Here, $\operatorname{Hess}^{\mc C}$ is the $(i_k-k)\times (i_k-k)$ dimensional Hessian matrix with respect to the variables in $\mathbf{y}_{i_k}^{\mc C}$.
\end{itemize}
\end{defn}

\begin{rem}
The set of $\mb y_{i_k}^\C$ such that $I_{i_k}^{\C,E}(\m_{i_k}(\mb y_{i_k}^{\mc C}))\not=\infty$, i.e.~the set of potential minimizers, is an open set on which $y_{i_1},\dots,y_{i_k}$ are smooth functions of $\mb y_{i_k}^{\mc C}$. On this set we have $I_{i_k}^{\C,E}(\m_{i_k}(\mb y_{i_k}^{\mc C}))=\sum_{j=1}^{i_k}W_j(y_j)$.
It is known from Morse theory that almost all $C^2$ functions have only non-degenerate critical points, meaning that the Hessian determinant does not vanish at these points. Furthermore, there can be only finitely many such points in any compact set. From a more practical point of view, if $V_1,\dots,V_{i_k+2}$ are not generic, then the perturbation $\tilde{V}_1,\dots,\tilde{V}_{i_k+2}$ with $\tilde{V}_j(t):=V_j(t)+a_jt, j=1,\dots,i_{k+2}$ is generic for almost all values $a_1,\dots,a_{i_k+2}\in\R$ \cite[Theorem 2.20]{Matsumoto}. In the case of unbounded $E$ and thus unbounded $D_j$, the search for minimizers can effectively be restricted to a certain compact set thanks to the integrability conditions \eqref{int_conditions}. This justifies calling $V_1,\dots,V_{i_k+2}$ with the properties of Definition \ref{def_generic} generic.

Of course, almost all $C^2$ functions will also have a unique multivariate minimizer. Nevertheless we find it useful to consider the more general case of several multivariate minimizers as it features a particular aspect of the universality phenomenon (see Theorem \ref{general_lln} and Remark \ref{remark_theorem_lln} below).
\end{rem}
We can now formulate the main results of this section. The first one is an analog of Theorem \ref{uniform_lln} in the general setting. It is instructive to briefly review one of the main results of the unconstrained case: \cite[Theorems 2.1,~2.5 and 2.7]{DTV} show that the first $l$ random moments from $\mathbb{P}_{n, E, V}^{\C}$ with $\C = \emptyset$ and generic $V_1,\,V_2$ with minimizers $y_1^*,\,y_2^*$ converge to the first $l$ non-random moments of the probability measure $\mu_{y_1^*,y_2^*}$, where
\begin{align}\label{mu_E}
\mu_{y_1^*,y_2^*}(dx):=\begin{cases}\medskip
\lb 1-\frac{p_1^*}{p_2^*}\rb_+\d_0+\lb\frac{p_1^*+p_2^*-1}{p_2^*}\rb_+\d_1+\frac{\sqrt{(x-a)(b-x)}}{2\pi p_2^*x(1-x)}1_{[a,b]}(x)dx&,\quad \text{ if }E=[0,1],\\\medskip
\lb 1-\frac{z_1^*}{z_2^*}\rb_+\d_0+\frac{1}{2\pi z_2^*}\frac{\sqrt{(x-a)(b-x)}}{x}1_{[a,b]}(x)dx&,\quad \text{ if }E=\R_+,\\
\frac{1}{2\pi\b^*}\sqrt{(x-a)(b-x)}1_{[a,b]}(x)dx&,\quad\text{ if }E=\R.
\end{cases}
\end{align}
Here, $(\,\cdot\,)_+:=\max(\,\cdot\,,0)$ and $a$ and $b$ are constants depending on $y_1^*, y_2^*$ and $E$ only. To ease notation, we will with a slight abuse of notation not index the measure with $E$ and instead trust that the reader will distinguish the three cases $E=[0,1],\,\R_+,\,\R$ from the appearance of $p$'s, $z$'s or $\alpha,\,\beta$, respectively. The measure $\mu_{p_1^*, p_2^*}$ is called Kesten-McKay measure or free binomial distribution as it is a free convolution power of the Bernoulli distribution. We refer to \cite[p.~4f]{DTV} for details. The measures $\mu_{z_1^*, z_2^*}$ and $\mu_{\alpha^*, \beta^*}$ are the Marchenko-Pastur and semicircle distributions, respectively, which are well-known in random matrix theory and free probability. The universality of the three measures is apparent as very different functions $V_1,\,V_2$ lead to the same family of measures.\\

Let us now formulate an analogous result in the presence of a constraint $\mc C$ that identifies the limit measures as members of the Bernstein-Szeg\H o class plus some discrete measure. To state it, define
\begin{align}\label{def_y_12}
y_1^*:=\begin{cases} y^*_{i_k+1},&\quad\text{ if }i_k\text{ is even},\\y^*_{i_k+2},&\quad\text{ if }i_k\text{ is odd},\end{cases}\qquad y_2^*:=\begin{cases} y^*_{i_k+2},&\quad\text{ if }i_k\text{ is even},\\y^*_{i_k+1},&\quad\text{ if }i_k\text{ is odd}.\end{cases}
\end{align} 

    \begin{thm}[Law of Large Numbers]\label{general_lln}
       Let $V_1,\dots,V_{i_k+2}$ be generic. Then we have as $n\to\infty$
        \begin{align*}
            (m_1^{(n)}, \ldots, m_l^{(n)}) \to  \sum \limits_{q= 1}^p w_q \delta_{\mathbf{m}_{l}^{*, q}}
        \end{align*}
				in distribution,
        where $w_1,\dots,w_p>0,\,\sum_{q=1}^p w_q=1$ are weights and  $\mathbf{m}^{*, q}_{l} :=(m_{1}^{*, q},\dots,m_{l}^{*, q})$ is the vector of the first $l$ moments of a probability measure
        \begin{align*}
            \mu_q(dx)=\mu_q^{ac}(x)dx+\mu_q^d(dx)\,.
        \end{align*}
        The density of the absolutely continuous part is given by 
        \begin{align*}
        	\mu_q^{ac}(x)=\frac{1}{D_q(x)}\mu_{y_1^*,y_2^*}^{ac}(x)\,,
        \end{align*}
        	where $\mu_{y_1^*,y_2^*}^{ac}$ is the density of the absolutely continuous part of the measure $\mu_{y_1^*,y_2^*}$ defined in~\eqref{mu_E} and $D_q$ is a polynomial of degree at most $i_k$, strictly positive on the support of $\mu_{y_1^*,y_2^*}^{ac}$. The measure $\mu_q^d$ is discrete having atoms at the positions of atoms of $\mu_{y_1^*,y_2^*}$ and at most $i_k$ extra atoms.
    \end{thm}
		\begin{rem}\label{remark_theorem_lln}
		Theorem \ref{general_lln} shows how strong the universality of \eqref{mu_E} is: The limiting measure needs to fulfil the constraint $\mc C$ and, apart from atoms, the optimal measure from $\mathcal P^{\mc C}(E)$ is absolutely continuous w.r.t.~$\mu_{y_1^*,y_2^*}$. In particular, the support of the absolutely continuous part of $\mu_q$, $q=1,\dots,p$ is always the one of $\mu_{y_1^*,y_2^*}$, regardless of the constraint $\mc C$! Moreover, as nicely featured in the case of several minimizers, this universality is not even broken if there are several limiting measures.
    \end{rem}
    \begin{example}\label{example_general}
		The proof of Theorem \ref{general_lln} is constructive: A recipe to determine the weights $w_q$, polynomials $D_q$ and measures $\mu_q^d$ is given in Proposition~\ref{conv_discrete} and Poposition~\ref{measure_representation}. As an example we consider the case $E = \mathbb{R}$ with the constraint $\C := \{m_1 = 0\}$, i.e.~we only consider measures with mean zero. We take the functions $V_1(\alpha) := (\alpha - 1)^2$ and $V_2(\beta) := 8\beta^2$. Then $W_1(\alpha) = V_1(\alpha)$ has a unique minimizer in $\alpha^* = 1$ and $W_2(\beta) = V_2(\beta) - \log(\beta)$ in $\beta^*=\frac{1}{4}$. The limiting measure is 
  \begin{align}
            \mu_{\a^*,\b^*}(dx) = \frac{\sqrt{x(2-x)}}{2\pi\Big(x + \frac{1}{4}\Big)}1_{[0, 2]}(x) \, dx + \frac{3}{4} \delta_{-\frac{1}{4}}\,.\label{example_measure}
  \end{align}
  The computation of this measure will be performed in Example \ref{example_continued} following Proposition~\ref{measure_representation}, where the necessary preliminaries are given.     
    \end{example}

Our next result concerns Gaussian fluctuations. A remarkable fact is that in the case of several minimizers, the standardization is itself random.
    
    \begin{thm}[Central Limit Theorem]\label{general_dist_clt}
Let $V_1,\dots,V_{i_k+2}$ be generic.  Define
        \begin{align*}
            \mathbf{m}_l^* := \argmin \limits_{\mb m^* \in \{\mathbf{m}^{*, 1}_{l}, \ldots, \mathbf{m}^{*, p}_{l}\}}\|(m_1^{(n)}, \ldots, m_l^{(n)}) - \mb m^*\|,
        \end{align*}
        where $\mb m^{*, q}_{l},\,q=1,\dots,p$ have been introduced in Theorem \ref{general_lln} and $\|\cdot\|$ denotes the Euclidean norm.
        Then, as $n\to\infty$
        \begin{align}
            Z_n:=\sqrt{n}\Sigma_l(\mathbf{m}_l^*)((m_1^{(n)}, \ldots, m_l^{(n)}) - \mathbf{m}_l^*) \to \mathcal{N}(0, L)\label{zn}
        \end{align}
        in distribution, where $L\in\R^{l\times l}$ is the matrix with $L_{uv} =1$ if $u=v$ and $v \notin \{i_1, \ldots, i_k\}$ and $L_{uv} =0$ else. The matrix $\Sigma_l \in \mathbb{R}^{l \times l}$ is given by

        \begin{align*}
            \Sigma_l(\mb m_l) &:= T^{\mc C}\bigg[\Big(\operatorname{Hess}^{\mc C} \sum_{j=1}^{l}W_j\Big)^{1/2}(\mathbf{y}^{\mc C}_l)\left(\frac{\partial \mb y_l^{\mc C}}{\partial \mb m_l^{\mc C}}(\m_l)\right)^t\bigg]\,, 
        \end{align*}
        where $T^{\mc C}:\R^{(l-k)\times(l-k)}\to\R^{l\times l}$ denotes insertion of rows and columns of zeros at the positions $i_1,\dots,i_k$, $\operatorname{Hess}^{\mc C}$ denotes the Hessian with respect to the coordinates in $\mathbf{y}_l^{\mc C}$ and the variables $\mathbf{m}_l^{\mc C}$, $\mathbf{y}_l^{\mc C}$ are defined in~\eqref{m_n^C} and~\eqref{parametrization0}, respectively.
		
    \end{thm}
    
   We finish this section with a moderate deviations principle.

    \begin{thm}[Moderate Deviations Principle]\label{general_dist_mdp}
        Let $V_1,\dots,V_{i_k+2}$ be generic and $a_n \to \infty$ be a sequence with $a_n = o(\sqrt{n})$. Then the sequence $(Z_n)_n$ in~\eqref{zn} satisfies a large deviations principle with speed $\frac{n}{a_n^2}$ and good rate function
        \begin{align*}
            I(\mb m_l) := 
            \begin{cases}
                \frac{1}{2} \|\mb m_l\|_2^2 &, m_{i_1} = \ldots = m_{i_k} = 0 \\
                \infty &\text{,else}
            \end{cases}.
        \end{align*}
    \end{thm}

\subsection{Universality and Connection to Random Matrix Theory}\label{sec_universality}
There are several universal aspects of the results in this section (and Section \ref{sec_uniform}). The first one is the occurrence of Gaussian fluctuations, which is generic for uniform distributions on convex bodies, see \cite{Klartag}.	 As already briefly mentioned in the introduction, it is however surprising to find Gaussian fluctuations for vectors containing only finitely many coordinates (moments). This is a sign of the strong dependence between moments.

Even more interesting is the universality of the shape of the limiting measures $\mu_q, q=1,\dots,p$ of Theorem \ref{general_lln}. Their absolutely continuous parts are all of the form reciprocal of a polynomial times the universal $\mu_{y_1^*,y_2^*}^{ac}$ and thus depend on the constraint $\mc C$ and few values of the functions $V_1,\dots,V_{i_k+2}$ only.\\

As shown in Remark \ref{rem_not_eq_measure}, the limiting measures for constrained moment spaces are generally not equilibrium measures, in contrast to the ones obtained in unconstrained spaces. Let us now shed some more light on this phenomenon. Consider a vector $\m_n\in\mc M_{2n+1}(\R)$ with the corresponding canonical coordinates, i.e.~recurrence coefficients, and form the tridiagonal Jacobi matrix 
\begin{align*}
J:=\begin{pmatrix}
\a_1 & \sqrt{\b_1}\\
\sqrt{\b_1}&\a_2&\sqrt{\b_2}\\
&\sqrt{\b_2}&\a_3&\sqrt{\b_3}\\
&&\ddots &\ddots&\ddots\\
&&&\sqrt{\b_{2n-1}}&\a_{2n}&\sqrt{\b_{2n}}\\
&&&&\sqrt{\b_{2n}}&\a_{2n+1}
\end{pmatrix}.
\end{align*}
Note that $J$ is symmetric and thus diagonalizable with real eigenvalues $x_1,\dots,x_n$ and eigenvectors $\mb v_1,\dots,\mb v_n$. The probability measure
 \begin{align}
\mu^{(n)}:=\sum_{j=1}^n \langle \mb v_j,\mb e\rangle\d_{x_j}
\end{align}
is called spectral measure to $J$ and the first standard basis vector $\mb e$, where $\langle\cdot,\cdot\rangle$ denotes the Euclidean scalar product. It has the property that its $l$-th moment is $m_l$  which is by the spectral theorem simply the $(1,1)$-entry of $J^l$, $l=1,\dots,n$ (see \cite[Chapters 1.2, 1.3]{simon2011}). In particular, if $(m_1^{(n)},\dots,m_n^{(n)})$ is uniformly distributed in $\mc M_n^{\mc C}([0,1])$, then the associated random spectral measure $\mu^{(n)}$ fulfills the assumptions of Theorem \ref{uniform_ldp_functional}. Our laws of large numbers for random moments imply that for $(m_1^{(n)},\dots,m_n^{(n)})$ with distribution $\P_{n,E,V}^{\mc C}$, the associated random spectral measure $\mu^{(n)}$ converges weakly in distribution towards the limiting measure $c^{-1}\sum_{q=1}^p\l_q\mu_q$ from Theorem \ref{general_lln}.

To make the connection to equilibrium measures and random matrices, we remark that in random matrix theory equilibrium measures typically occur as limits of \textit{empirical} spectral measures
\begin{align*}
\nu_n:=\frac{1}n\sum_{j=1}^n \d_{x_j},
\end{align*}
where $x_1,\dots,x_n$ are the eigenvalues of some random $n\times n$ matrix. An important class of random matrix ensembles, the so-called $\b$-ensembles, has eigenvalue densities proportional to
\begin{align}\label{beta-ensemble}
e^{\frac\b2\sum_{i\ne j}\log\lv x_i-x_j\rv-n\sum_{j=1}^n Q_{ex}(x_j)}=e^{- n^2\lb\frac\b2\int_{\ne}\int \log\lv t-s\rv^{-1} d\nu_n(t)d\nu_n(s)+\int Q_{ex}(t)d\nu_n(t)\rb},
\end{align}
where $\int_{\ne}\int$ means that the diagonal is excluded in the integral, $\b>0$ and $Q_{ex}$ is some function called external field. In view of \eqref{beta-ensemble}, it is not surprising that for $n\to\infty$, the ensemble realizes eigenvalue configurations that approach the equilibrium measure which is a solution to minimizing \eqref{eq_minimization}, see e.g.~\cite[Theorem 2.1]{Johansson98}. In fact, a large deviations principle can be obtained for this approximation of the equilibrium measure, cf.~\cite[Theorem 2.6.1]{AGZ}. For our purposes, it is important to note that the speed of this large deviations principle is $n^2$, a fact that is readily read off \eqref{beta-ensemble}. Now, apparently the difference between the spectral measure connected to the random moment problem and the empirical spectral measure from random matrix theory, is the weighting of the atoms. In contrast to $\nu_n$, the weights $w_j:=\langle \mb v_j,\mb e_1\rangle$ of the spectral measure $\mu^{(n)}$ depend on the eigenvectors and are of course random. In certain cases, more can be said: For example, for $\m_{2n-1}^{(n)}$ uniform on $\mc M_{2n-1}([0,1])$, the eigenvalues of the associated random matrix $J$ are distributed according to \eqref{beta-ensemble} with $\b=4$ and $Q_{ex}(t)=0$ for $t\in[0,1]$ and $\infty$ elsewhere, see \cite[Theorem 2.2]{KillipNenciu}. Moreover, the weights are independent from the eigenvalues and have a Dirichlet$(2,\dots,2)$-distribution on the simplex $\{\mb w_n\in[0,1]^n:\sum_{j=1}^nw_j=1\}$. For $n\to\infty$, the weights concentrate around the barycenter $(1/n,\dots,1/n)$. However, the speed is $n$ and thus the concentration of the weights is weaker than that of the eigenvalues. More precisely, it was shown in \cite{GamboaRouault} that the random spectral measure corresponding to the uniform distribution on $\mc M_n([0,1])$ satisfies an LDP with speed $n$ and good rate function given by the reversed Kullback-Leibler divergence $I(\mu)=\mc K(\mu^0|\mu)$. In view of the close connection between spectral measure and the random moment problem, this also explains the occurence of the Kullback-Leibler divergence in Theorem \ref{uniform_lln}. However, the question remains why the unconstrained equilibrium measure $\mu^0$ still appears instead of the measure solving the constrained equilibrium problem. 
To understand this, let us consider what happens if we force the spectral measure to fulfil a constraint $\mc C$. Deviations of the eigenvalue configuration from the energy minimizing equilibrium measure (or rather its discrete analog, the Fekete points) are far more costly than deviations of the weights from $(1/n,\dots,1/n)$. Therefore the weights have to change dramatically from $(1/n,\dots,1/n)$ to some other values in order to bend the ``discretized arcsine measure'' to match the constraint.  This results e.g.~in a limiting measure $\mu^{\mc C}$ that is supported on the full set $[0,1]$ like the arcsine measure, as observed in Remark \ref{rem_not_eq_measure}.

    % Section
\section{Parametrizations and Proof of Law of Large Numbers} \label{sec_limit}
The first step towards proving the results of Sections \ref{sec_uniform} and \ref{sec_general_dist} is a parametrization of the constrained space $\mc M_n^{\mc C}(E)$. Define the set $A_{i_k}\subset \prod_{j=1}^{i_k}D_j$ as
\begin{align}
A_{i_k}:=\lb\varphi_{i_k}^E\rb^{-1}(\textup{relint}\,\mc M_{i_k}^{\mc C}(E)),\label{def_A_i_k}
\end{align}
where $\varphi_{i_k}^E$ is the parametrization \eqref{def_can_coord} of the unconstrained moment space and relint denotes the relative interior, i.e.~$\relint \mathcal{M}_{i_k}^{\mc C}(E):=\mc M_{i_k}^{\mc C}(E)\cap \inn\mc M_{i_k}(E)$. Define $\tilde A_{i_k}$ as the projection of $A_{i_k}$ to the coordinates in $\mb y_{i_k}^{\mc C}$ and $\tilde{\mc M}_n^\C(E)$ as the projection of $\mc M_n^\C(E)$ to the coordinates in $\m_n^\C$.
Then we have the following lemma.

\begin{lemma}\label{lemma_Jacobian}
For any $n \ge i_k$ the mapping \begin{align}
            \varphi_{n}^{E,\C}&:\left\{
            \begin{array}{ccc}
              \tilde A_{i_k}\times D_{i_k+1} \times \cdots \times D_n    & \to       &    \inn \tilde{\mathcal{M}}_{n}^{\mc C}(E)      \\
                \mb y_n^{\mc C}  & \mapsto   & \m_n^{\mc C}
            \end{array}
            \right.
        \end{align}
is a $C^\infty$-diffeomorphism with Jacobian
				\begin{align}\label{Jacobian}
				\left\lv \det\left[\frac{\partial \mb m_n^{\mc C}}{\partial \mb y_n^{\mc C}}\right]\right\rv:=\left\lv \det\left[\frac{\partial \varphi_{n}^{E,\C}(\mb y_n^{\mc C})}{\partial \mb y_n^{\mc C}}\right]\right\rv= \begin{cases}
                \prod \limits_{j = 1}^n (p_j(1 - p_j))^{n - j-d_{j, \C}} &, E = [0, 1], \\
                \prod \limits_{j = 1}^{n - 1} z_j^{n - j-d_{j, \C}} &, E = [0, \infty), \\
                \prod \limits_{j = 1}^{\lfloor n/2 \rfloor} \beta_j^{n - 2j-d_{j, \C}} &, E = \mathbb{R},
            \end{cases}
				\end{align}
where $d_{j, \C} := \#\{l \mid i_l > j\}$ and $\lfloor n/2\rfloor$ is the largest integer smaller or equal $n/2$. 
\end{lemma}
    \begin{proof}
        It is straightforward to see that $\varphi_n^{E, \C}$ is indeed a diffeomorphism and we only need to calculate its Jacobian. Note that the Jacobian matrix is lower triangular and thus its determinant is given by the product of the entries on the diagonal
        \begin{align}
            \det \varphi_n^{E, \C} = \prod \limits_{\substack{1 \le j \le n \\ j \notin \{i_1, \ldots, i_k\}}} \frac{\partial m_j}{\partial y_j}.\label{jacobian_product}
        \end{align}
        In the case $E = [0, 1]$, rearranging \eqref{def_p_j} yields
        \begin{align*}
            m_j = p_j(m_j^+ - m_j^-) - m_j^- \,.
        \end{align*}
        As $m_j^\pm$ only depends on $m_1, \ldots, m_{j - 1}$, we obtain
        \begin{align*}
            \frac{\partial m_j}{\partial p_j} = m_j^+ - m_j^-
        \end{align*}
        and the assertion follows by an application of the formula (cf.~\cite{skibinsky1967})
        \begin{align}
            m_l^+ - m_l^- = \prod \limits_{j = 1}^{l - 1} p_i(1 - p_i) \label{moment_range}
        \end{align}
        and rearranging the order of multiplication in \eqref{jacobian_product}.
        
        The proof in the case $E = \mathbb{R}_+$ is analogous. Note that formula \eqref{def_z_j} yields
        \begin{align*}
            m_j = z_j(m_{j - 1} - m_{j - 1}^-) + m_j^-
        \end{align*}
        and consequently
        \begin{align*}
            \frac{\partial m_j}{\partial z_j} = m_{j - 1} - m_{j - 1}^- = \prod \limits_{l = 1}^{j - 1} \frac{m_l - m_l^-}{m_{l - 1} - m_{l - 1}^-} = z_1 \cdots z_{j - 1}\,.
        \end{align*}
        
        In the case $E = \mathbb{R}$, the calculation of the partial derivatives is slightly more complicated. As in the proof of Lemma~2.6 in \cite{DTV}, the partial derivatives of the moments with respect to the recursion coefficients can be calculated as
        \begin{align*}
            \frac{\partial m_{2j - 1}}{\partial \alpha_{j}} &= \beta_1 \cdots \beta_{j - 1}\,, \\
            \frac{\partial m_{2j}}{\partial \beta_{j}} &= \beta_1 \cdots \beta_{j - 1}\,.
        \end{align*}
        The assertion again follows by plugging in these formulas into \eqref{jacobian_product} and rearranging the order of multiplication.
    \end{proof}
		We are now in the position to prove Lemma \ref{lemma_density}.
\begin{proof}[Proof of Lemma \ref{lemma_density}]
We have to show that \eqref{normalizing} is finite, positivity is trivial. The compact case $E=[0,1]$ is clear. From the unbounded cases we exemplarily consider $E=\R_+$. Changing to canonical coordinates we see that 
\begin{align}\label{normalizing2}
Z_{n,E,V}^{\mc C}=\int_{\R_+^{n-k}} 1_{\tilde A_{i_k}}(\mb z^{\mc C}_{i_k})\exp\lb-n \sum \limits_{j = 1}^n \lb V_j(z_j)-\frac{n-j-d_{j,\C}}{n}\log z_j\rb\rb d\mb z_n^{\mc C}.
\end{align}
Let us drop the indicator $1_{\tilde A_{i_k}}(\mb z^{\mc C}_{i_k})$ from the integral, thereby making it larger, and consider first the $z_{i_j}$, $j=1,\dots,k$.
 Note that by \eqref{int_conditions}, there is a constant $c_1\in\R$ such that we have for $z_{i_j}>1$
\begin{align*}
V_j(z_j)-\frac{n-j-d_{j,\C}}{n}\log z_j\geq c_1.
\end{align*}
 Thus, for those $\mb z_n^{\mc C}$ with $z_{i_j}>1$ we can use the bound
\begin{align*}
\exp\lb- n\lb V_{i_j}(z_{i_j})-\frac{n-i_j-d_{i_j,\C}}{n}\log z_{i_j}\rb\rb\leq e^{-nc_1}.
\end{align*}
By continuity we have for $0\leq z_{i_j}\leq 1$ for some $c_2\in\R$ the bound $V_{i_j}(z_{i_j})\geq c_2$ and hence for those $\mb z_n^{\mc C}$ with $0\leq z_{i_j}\leq1$, we may use
\begin{align}
\exp\lb- n\lb V_{i_j}(z_{i_j})-\frac{n-i_j-d_{i_j,\C}}{n}\log z_{i_j}\rb\rb\leq e^{-nc_2}.
\end{align}
This shows
\begin{align*}
Z_{n,E,V}^{\mc C}\leq e^{-kn\min(c_1,c_2)}\int_{\R_+^{n-k}} \exp\lb-n \sum \limits_{\substack{j = 1\\j\not=i_1,\dots,i_k}}^n \lb V_j(z_j)-\frac{n-j-d_{j,\C}}{n}\log z_j\rb\rb d\mb z_n^{\mc C},
\end{align*}
which factorizes and is integrable by assumption \eqref{int_conditions}.
\end{proof}
As a direct consequence of Lemma \ref{lemma_Jacobian} we obtain the following 
\begin{cor} \label{corollary_density}
$\mathbb P_{n,E,V}^{\mc C}$ induces via $\lb\varphi_n^{E,\C}\rb^{-1}$ a probability measure on $\prod \limits_{\substack{1 \le j \le n \\j \notin \{i_1, \ldots, i_k\}}}D_j$ with density
\begin{align}
 \tilde P^{\mc C}_{n,E, V}(\mb y_n^{\mc C})  &:=\frac{1}{ Z^{\mc C}_{n,E,V}} \exp\Big(-n \sum \limits_{j = 1}^n W_j(y_j)\Big)1_{\tilde A_{i_k}}(\mb y^{\mc C}_{i_k})\label{density_canonical}\\
&\times \exp\lb\sum_{j=1}^n\lb W_j(y_j)-V_j(y_j)\rb\lb j+d_{j, \C}\rb\rb\label{density_remainder}
\end{align}
 w.r.t.~the $(n-k)$-dimensional Lebesgue measure.
\end{cor}
Note that in \eqref{density_canonical} and \eqref{density_remainder} the canonical coordinates $y_{i_1},\dots,y_{i_k}$ corresponding to the constrained moments $m_{i_1},\dots,m_{i_k}$ are functions of $\mb y_n^{\mc C}$. More precisely, $y_{i_j}$ depends on all coordinates $y_d$ from $\mb y_n^{\mc C}$ with $d<i_j$ and is determined by the requirement $m_{i_j}=c_{i_j}$, $j=1,\dots,k$.

We can see from \eqref{density_canonical} and \eqref{density_remainder} that the coordinates in $\mb y_{i_k}^{\mc C}$ are dependent. On the one hand, they are coupled via the indicator function. On the other hand, even apart from the indicator, in general the density does not factorize as $y_{i_1},\dots,y_{i_k}$ are functions of the lower canonical coordinates.

   Corollary \ref{corollary_density} allows for the simple but important observation that the canonical coordinates $\mb y_n^{\mc C}$ have a density of the form
    \begin{align}
        \mathbb{P}_n(dx) = \frac{1}{Z_n} e^{-nW(x)} R(x),\label{Laplace_density}
    \end{align}
    where $W, R: \mathbb{R}^m \to \mathbb{R}\cup\{\infty\}$ are some measurable functions and $Z_n$ is the normalization constant.

    On the level of canonical coordinates, probabilistic statements can be deduced directly from the form of the density \eqref{Laplace_density}. We will use in Section~\ref{sec_application_uniform} a different parametrization and thus obtain functions $W$ and $R$ different from those needed here to prove the results of Section \ref{sec_general_dist}.  To avoid duplication, we therefore state and prove a proposition valid for rather general $W$ and $R$. It might also be of independent interest.
		
		Assume from here on that  $W: \mathbb{R}^m \to \mathbb{R} \cup \{\infty\}$ and $R: \mathbb{R}^m \to \mathbb{R}_+$ are measurable and such that
    	\begin{align}
    	0<\int e^{-n_0W(x)} R(x) \, dx< \infty \quad \text{ for some }n_0 \in \mathbb{N}.\label{Laplace_assumption}
    	\end{align}

    \begin{prop}[Convergence to discrete distribution] \label{conv_discrete}
        Let $W$ and $R$ satisfy the following conditions:
        \begin{enumerate}
            \item $W$ attains its global minimum exactly in the points $\theta_1, \ldots, \theta_p$, i.e. $W(x) > W(\theta_1)$ for $x \in \mathbb{R}^m \setminus \{\theta_1, \ldots, \theta_p\}$ and $W(\theta_1) = \ldots = W(\theta_p)$.
            \item $W$ is twice continuously differentiable in a neighborhood of each $\theta_q$.
            \item For all $\varepsilon > 0$ and $M_\varepsilon := \bigcup \limits_{q=1}^p B_\varepsilon(\theta_q)$ we have $\inf \limits_{x \in M_\varepsilon^{\mc C}} W(x) > W(\theta_q)$, where $B_\e(\theta_q)$ denotes the open $\e$-ball around $\theta_q$.
             \item $\Hess W(\theta_q)$ is nonsingular for each $1\le q \le p$.
            \item $R$ is continuous in all $\theta_q$ and does not vanish in all $\theta_q$ simultaneously.
        \end{enumerate}
        Then $\mathbb{P}_n$ from \eqref{Laplace_density} converges as $n\to\infty$ weakly to the distribution
        \begin{align*}
            \mathbb{P} = \left(\sum \limits_{q=1}^p R(\theta_q) \det \Hess W(\theta_q)^{-1/2}\right)^{-1}\sum \limits_{q=1}^p R(\theta_q) \det \Hess W(\theta_q)^{-1/2} \delta_{\theta_q}.
        \end{align*}
    \end{prop}

    \begin{proof}
 	We may assume without loss of generality that $W(\theta_1) = \ldots = W(\theta_p) = 0$. Let $U \subset \mathbb{R}^m$ be an arbitrary open set and $\varepsilon > 0$ so small, that the following conditions hold:
 	\begin{enumerate}[label=(\roman*)]
 		\item If $\theta_q \in U$ holds, then also $B_\e(\theta_q) \subset U$.
 		\item $W$ is twice differentiable on $B_\e(\theta_q)$.
 		\item $\Hess W(x) \ge M_q$ for all $x \in B_\e(\theta_q)$ and some positive definite matrices $M_q$, where $\ge$ stands for the Löwner (partial) order, i.e.~$M\geq N$ means that $M-N$ is positive semidefinite. \label{cond_pos}
 		\end{enumerate}
 	To see that the third condition is attainable, note that $\Hess W(\theta_q)$ is positive semidefinite, since $\theta_q$ is an absolute minimum of $W$. Since $\Hess W(\theta_q)$ is nonsingular, the matrix is even positive definite. By Weyl's inequality (see Theorem~1 in Section 6.7 of \cite{frank1968}) there is a $\delta > 0$ so that $\Hess W(\theta_q) - \delta I_m$ is positive definite. Since $\Hess W$ is twice continuously differentiable near all $\theta_q$, we may choose $\varepsilon > 0$ such that
 	\begin{align*}
 	\left|\frac{\partial^2 W}{\partial x_i \partial x_j}(x) - \frac{\partial^2 W}{\partial x_i \partial x_j}(\theta_q)\right| < \frac{\delta}{2m}
 	\end{align*}
 	is satisfied for all $x \in B_\e(\theta_q)$ and $1 \le i, j \le m$. By Gerschgorin's theorem (see Theorem~1 in Section 6.8 of \cite{frank1968}) this implies that all eigenvalues of $\Hess W(x) - \Hess W(\theta_q)$ have absolute value less than $\frac{\delta}{2}$. From Weyl's inequality we can conclude that all eigenvalues of $\Hess W(x) - \Hess W(\theta_q) + \delta I_m$ must be at least $\frac{\delta}{2}$. This means that in the Löwner order the inequalities
 	\begin{align*}
 	\Hess W(x) > \Hess W(\theta_q) - \delta I_m > 0
 	\end{align*}
 	hold for all $x \in B_\e(\theta_q)$ and we may choose $M_i = \Hess W(\theta_q) - \delta I_m$ in condition~\ref{cond_pos}. \br

Now, with $M_\e$ from condition (3) of the proposition,
 	\begin{align*}
 	\mathbb{P}_n(U) &\ge \frac{\sum \limits_{q=1}^p 1_U(\theta_q) \int_{B_\e(\theta_q)} e^{-n W(x)}R(x) \, dx}{\int_{(M_\varepsilon)^c}e^{-nW(x)} R(x) \, dx  + \sum \limits_{q=1}^p \int_{B_\e(\theta_q)} e^{-nW(x)} R(x) \, dx}.
 	\end{align*}
 By a standard application of Laplace's method, we have as $n\to\infty$
 \begin{align*}
 \int_{B_\e(\theta_q)} e^{-nW(x)} R(x) \, dx = n^{-m/2} \left(R(\theta_q) \frac{(2\pi)^{m/2}}{\sqrt{\det \Hess W(\theta_q)}} + o(1)\right).
 \end{align*}
 Furthermore, with $K := \inf \limits_{x \in (M_\varepsilon)^c} W(x) > 0$ we obtain
 \begin{align*}
 \int_{(M_\varepsilon)^c} e^{-nW(x)} R(x) \, dx \le e^{-(n - n_0)K} \int_{(M_\varepsilon)^c} e^{-n_0W(x)} R(x) \, dx = o(n^{-m/2}),
 \end{align*}
 implying
 	\begin{align*}
 	&P_n(U)\geq \frac{\left(\frac{2\pi}{n}\right)^{m/2} \left(\sum \limits_{q=1}^p 1_U(\theta_q) R(\theta_q) \det \Hess W(\theta_q)^{-1/2} + o(1) \right)}{\left(\frac{2\pi}{n}\right)^{m/2} \left(\sum \limits_{q=1}^p R(\theta_q) \det \Hess W(\theta_q)^{-1/2} + o(1)\right)} \xrightarrow{n \to \infty} \mathbb{P}(U).
 	\end{align*}
 	By Portmanteau's theorem, this implies $\mathbb{P}_n \xrightarrow{d} \mathbb{P}$.
 \end{proof}
    
    We will see that applied to our random moment problem, Proposition \ref{conv_discrete} shows convergence of the random canonical coordinates $\mb y_{n}^{\C,(n)}$ to some discrete distribution concentrated on the moment sequences of certain limiting measures. The main task now is to identify these measures from the information on their canonical coordinates. To this end, we make use of the following proposition which gives information on probability measures with eventually constant recurrence coefficients. We remark that similar results are known, see e.g.~\cite[Theorem 2.1]{GeronimoIliev} and references therein, yet they are not sufficient for our purposes.
    
    \begin{prop}\label{measure_representation}
    	Let $\mu$ be a probability measure with recursion coefficients $\alpha_1, \ldots, \alpha_j$ and $\beta_1, \ldots, \beta_{j - 1} > 0$, such that $\alpha_l = \alpha$ and $\beta_{l - 1} = \beta > 0$ for all $l > j$. Then the Lebesgue decomposition of $\mu$ consists of an absolutely continuous part $\mu^{ac}$ and a discrete part $\mu^d$. The absolutely continuous part $\mu^{ac}$ has the density
    	\begin{align*}
    	\frac{\sqrt{4\beta - (x - \alpha)^2} \beta_1 \cdots \beta_{j - 1}}{2\pi D(x)}
    	\end{align*}
    	on the interval $I_{\alpha, \beta} := [\alpha - 2\sqrt{\beta}, \alpha + 2\sqrt{\beta}]$, where $D(x) := P_j^2(x) - P_{j + 1}(x) P_{j - 1}(x)$ and $P_i,\,i\geq0$ are the monic orthogonal polynomials corresponding to $\mu$. To leading order (with the convention $\beta_0 = 0$),
    	\begin{align}
    	D(x) = (\alpha - \alpha_j) x^{2j - 1} + \Big[(\beta - \beta_{j - 1}) + \big(2(\alpha_1 + \cdots + \alpha_{j - 1}) + \alpha_j\big)(\alpha_j - \alpha)\Big] x^{2j - 2} + O(x^{2j - 3}). \label{expansion_d}
    	\end{align}
    	Moreover, $D$ is strictly positive on $\inn I_{\alpha, \beta}$ and possible zeros on the boundary of $I_{\a,\b}$ are simple.
    	
    	The discrete part $\mu^d$ is a linear combination of dirac measures
    	\begin{align*}
    	\mu^d = \sum \limits_{i = 1}^{2j - 1} \lambda_i \delta_{x_i},
    	\end{align*}
    	where the $x_i$ are the real roots of $D$ outside of $I_{\alpha, \beta}$. The weights $\lambda_i$ are possibly zero and are given as
    	\begin{align}
    	\lambda_i := \lim \limits_{x \to x_i} \Big|(x - x_i) \frac{f(x)}{D(x)}\Big|,\label{weights}
    	\end{align}
    	where the function $f$ is defined by
    	\begin{align}
    	f(x) := Q_j(x) P_j(x) - Q_{j + 1}(x) P_{j - 1}(x) + \frac{\beta_1 \cdots \beta_{j - 1}}{2} \big((x - \alpha) - z(x) \sqrt{(x - \alpha)^2 - 4\beta}\big) \label{func_f}
    	\end{align}
    	with 
    	\begin{align*}
    	z(x) := 
    	\begin{cases}
    	1 & x > \alpha + 2\sqrt{\beta} \\
    	-1 & x < \alpha - 2\sqrt{\beta}
    	\end{cases},
    	\qquad \qquad Q_i(x) := \int \frac{P_i(x) - P_i(t)}{x - t} \, d\mu(t).
    	\end{align*}
    \end{prop}
    	The polynomials $Q_i,\, i\geq0$ in the proposition are called \emph{secondary polynomials} and satisfy the shifted recurrence relation $Q_0(x) = 0$, $Q_1(x) = 1$ and
\begin{align}
Q_i(x) = (x - \alpha_i) Q_{i - 1}(x) - \beta_{i - 1} Q_{i - 1}(x),\quad i \ge 2.\label{recursion_secondary}
			\end{align}
    \begin{proof}
    	By a classical result due to Markov, the Stieltjes transform $S_\mu(z)$ of $\mu$,
			\begin{align*}
			S_\mu(z):=\int_E\frac{1}{z-t}d\mu(t),\quad z\in\mathbb C_+:=\{z\in\mathbb C\,:\,\Re z>0\},
			\end{align*}
			admits a continued fraction expansion in terms of the recurrence coefficients, 
    	\begin{align}
    	S_\mu(z) = \polter{1}{z - \alpha_1} - \polter{\beta_1}{z - \alpha_2} - \cdots - \polter{\beta_{j - 1}}{z - \alpha_{j}} - \beta S_{\alpha, \beta}(z)\,,\label{cf_expansion}
    	\end{align}
			see e.g.~\cite[Lemma 4.1]{DTV} for a simple proof. Here $S_{\alpha, \beta}(z)$ is the Stieltjes transform of the measure with constant recurrence coefficients $\a,\,\b$. 
				Writing the convergent
				\begin{align*}
				\polter{1}{z - \alpha_1} - \polter{\beta_1}{z - \alpha_2} - \cdots - \polter{\beta_{j - 1}}{z - \alpha_{j}}=\frac{A_j}{B_j}
				\end{align*}
				as a single fraction, it is a standard argument to see that $B_j$, $j\geq 0$, are polynomials of degree $j$ satisfying the same recursion as the monic orthogonal polynomials $P_j$, $j\geq0$ w.r.t.~$\mu$. Comparison of coefficients then gives $B_j=P_j$. Likewise one finds that $A_j,\,j\geq0$ satisfies the recursion \eqref{recursion_secondary} and concludes $A_j=Q_j$. This gives
    	\begin{align}
    	S_\mu(z) = \frac{Q_j(z) - \beta S_{\alpha, \beta}(z) Q_{j - 1}(z) }{P_j(z) - \beta S_{\alpha, \beta}(z) P_{j - 1}(z)}\, \label{stieltjes1}.
    	\end{align}

			 Expanding $S_{\alpha, \beta}(z)$ shows that it satisfies the recursion
			\begin{align*}
			S_{\alpha, \beta}(z)=\frac1{z-\a-\b S_{\alpha, \beta}(z)},
			\end{align*}
			which yields $S_{\alpha, \beta}(z) = \frac{z - \alpha - \sqrt{(z - \alpha)^2 - 4\beta}}{2\beta}$. In this representation, the square root $\sqrt{y}$ of a complex number $y \in \mathbb{C} \setminus \mathbb{R}_+$ is defined to be the unique solution of the equation $x^2 = y$ with a positive imaginary part. With this definition, the square root appearing in $S_{\alpha, \beta}$ can be continuously extended to the real line via
    	\begin{align}
    	\lim \limits_{\substack{z \to x \\ z \in \mathbb{C}^+}} \sqrt{(z - \alpha)^2 - 4\beta} =
    	\begin{cases}
    	- \sqrt{(x - \alpha)^2 - 4\beta} &, \text{if } x < \alpha - 2\sqrt{\beta} \\
    	i\sqrt{4\beta - (x - \alpha)^2} &, \text{if } x  \in [\alpha - 2\sqrt{\beta}, \alpha + 2\sqrt{\beta}] \\
    	\sqrt{(x - \alpha)^2 - 4\beta} &, \text{if } x > \alpha + 2\sqrt{\beta} \\
    	\end{cases}\,.\label{root_extension}
    	\end{align}
			The latter expression then identifies $S_{\alpha, \beta}$ via the Stieltjes inversion formula
			\begin{align}
			\frac{\mu(dx)}{dx}=-\frac1\pi\lim_{y\to0}\Im S_\mu(x+iy)\label{Stieltjes_inversion}
			\end{align}
			as the Stieltjes transform of the semicircle distribution $\mu_{\a,\b}$ of \eqref{mu_E} on $I_{\alpha, \beta} = [\alpha - 2\sqrt{\beta}, \alpha + 2\sqrt{\beta}]$.
    
    	Observing \eqref{stieltjes1} and \eqref{root_extension}, we can see that $S_\mu$ can be extended to the interior of the interval $I_{\alpha, \beta}$ as 
    	\begin{align}
    	S_\mu(x) = \frac{Q_j(x) - \frac{1}{2}(x - \alpha) Q_{j - 1}(x) + \frac{i}{2} \sqrt{4\beta - (x - \alpha)^2} Q_{j - 1}(x) }{P_j(x) - \frac{1}{2}(x - \alpha) P_{j - 1}(x) + \frac{i}{2} \sqrt{4\beta - (x - \alpha)^2} P_{j - 1}(x)} \label{stieltjes-extension}
    	\end{align}
    	for any $x$ in which the denominator does not vanish. The only points in which the imaginary part of the denominator vanishes are the roots of the polynomial $P_{j - 1}$. However, in these points the real part of the denominator is given by $P_j(x)$, which does not vanish as the roots of two consecutive orthogonal polynomials strictly interlace. Consequently, $S_\mu$ can be continuously extended to $\inn I_{\alpha, \beta}$ and thus it admits a density on this interval by \eqref{Stieltjes_inversion}. To compute it, 
we start with the imaginary part of \eqref{stieltjes-extension}, expanding the fraction with the complex conjugate of the denominator and thus obtaining the new denominator
    	\begin{align*}
    	D(x) :={}& (P_j(x) - \beta S_{\alpha, \beta}(x) P_{j - 1}(x))\overline{(P_j(x) - \beta S_{\alpha, \beta}(x) P_{j - 1}(x))} \\
    	={}& \Big(P_j(x) - \frac{1}{2}(x - \alpha) P_{j - 1}(x)\Big)^2 + \frac{1}{4}\big(4\beta - (x - \alpha)^2\big) P_{j - 1}^2(x) \labelonce{denominator}\\
    	={}& P_j^2(x) - (x - \alpha) P_j(x) P_{j - 1}(x) + \beta P_{j - 1}^2(x) \\
    	={}& P_j^2(x) - P_{j + 1}(x) P_{j - 1}(x)\,.
    	\end{align*}
    	Calculating the corresponding numerator we arrive at
    	\begin{align*}
    	\Im S_\mu(x) = -\frac{1}{2D(x)} \sqrt{4\beta - (x - \alpha)^2} \Big(P_{j - 1}(x) Q_j(x) - Q_{j - 1}(x) P_j(x)\Big)
    	\end{align*}
    	for any $x \in (\alpha - 2\sqrt{\beta}, \alpha + 2\sqrt{\beta})$. We can further simplify using the Christoffel-Darboux formula (Theorem 4.5 in \cite{chihara1978})
    	\begin{align}
    	P_{j - 1}(x) Q_j(x) - Q_{j - 1}(x) P_j(x) ={}& \int \frac{P_{j - 1}(t)P_j(x) - P_j(t)P_{j - 1}(x)}{x-t} \, d\mu(t) \nonumber \\
    	={}& h_{j - 1} \int \sum \limits_{i = 0}^{j - 1} \frac{P_i(x)P_i(t)}{h_i} \, d\mu(t) = h_{j - 1} \label{christoffel}
    	\end{align}
    	with $h_i := \int P_i^2(t) \, d\mu(t) = \beta_1 \cdots \beta_i$. Consequently, $\mu$ has a density on the interval $\inn I_{\alpha, \beta}$ that is given by
    	\begin{align*}
    	\frac{\sqrt{4\beta - (x - \alpha)^2} \beta_1 \cdots \beta_{j - 1}}{2\pi D(x)}\,.
    	\end{align*}
    	
    	Next, we turn to a more thorough investigation of $D(x)$. First of all, from the representation \eqref{denominator} we can conclude that the only possible roots of $D$ on $I_{\alpha, \beta}$ must lie on the boundary of the interval. If $x_0$ is such a root, then
    	\begin{align*}
    	P_j(x_0) - \frac{1}{2}(x_0 - \alpha) P_{j - 1}(x_0) = 0
    	\end{align*}
    	must hold. Since $P_j$ and $P_{j - 1}$ have no common roots, this implies $P_{j - 1}(x_0) \ne 0$. But then $x_0$ can only be a simple root of $D$ by \eqref{denominator}. This also shows $\mu(\{\alpha \pm 2\sqrt{\beta}\}) = 0$ by the Stieltjes inversion formula for atoms
			\begin{align}
			\mu(\{x\})=-\lim_{y\to0}y\Im S_\mu(x+iy)\label{Stieltjes_inversion2}.
			\end{align}
    	Furthermore, an induction argument shows
    	\begin{align*}
    	P_j(x) ={}& x^j - (\alpha_1 + \cdots + \alpha_j) x^{j - 1} + \left(\sum \limits_{1 \le i < l \le j} \alpha_i \alpha_l - \sum \limits_{i = 1}^{j - 1} \beta_i \right)x^{j - 2} + O(x^{j - 3})\,,\\
    	P_j^2(x) ={}& x^{2j} - 2(\alpha_1 + \cdots + \alpha_j) x^{2j - 1}  \\
    	&+ \left(2 \sum \limits_{1 \le i < l \le j} \alpha_i \alpha_l - 2\sum \limits_{i = 1}^{j - 1} \beta_i + (\alpha_1 + \cdots + \alpha_j)^2 \right)x^{2j - 2} + O(x^{2j - 3})\,, \\
    	P_j(x) P_{j - 1}(x) ={}& x^{2j - 1} - (2(\alpha_1 + \cdots + \alpha_{j - 1}) + \alpha_j) x^{2j - 2} \\
    	&+ \Bigg(\sum \limits_{1 \le i < l \le j - 1} \alpha_i \alpha_l - \sum \limits_{i = 1}^{j - 2} \beta_i + \sum \limits_{1 \le i < l \le j } \alpha_i \alpha_l - \sum \limits_{i = 1}^{j - 1} \beta_i \\
    	&\qquad+ (\alpha_1 + \cdots + \alpha_j) (\alpha_1 + \cdots + \alpha_{j - 1})\Bigg)x^{2j - 3} + O(x^{2j - 4})\,.
    	\end{align*}
    	Plugging these formulas into \eqref{denominator} proves \eqref{expansion_d}. \br

    	Observing \eqref{root_extension}, we can continuously extend the Stieltjes transform to $\mathbb{R} \setminus I_{\alpha, \beta}$ via
    	\begin{align}
    	\frac{Q_j(x) - \frac{1}{2}(x - \alpha) Q_{j - 1}(x) + \frac{1}{2} z(x)\sqrt{(x - \alpha)^2 - 4\beta} Q_{j - 1}(x)}{P_j(x) - \frac{1}{2}(x - \alpha) P_{j - 1}(x) + \frac{1}{2} z(x)\sqrt{(x - \alpha)^2 - 4\beta} P_{j - 1}(x)} \label{stieltjes-extension2}
    	\end{align}
    	in any point $x$ that is not a root of the denominator. As this expression is real, the absolutely continuous part vanishes outside of $I_{\alpha, \beta}$. Let us now examine whether $\mu$ admits any point masses in the roots of the denominator. If $x_0$ is such a root, then an expansion of the polynomial terms in \eqref{stieltjes1} combined with the Stieltjes inversion formula \eqref{Stieltjes_inversion2} shows
    	\begin{align*}
    	\mu(\{x_0\}) = \lim \limits_{x \to x_0} |(x - x_0) S_\mu(x)|\,,
    	\end{align*}
    	where $S_\mu(x)$ is the extension of the Stieltjes transform to the real line \eqref{stieltjes-extension2}.
    	It only remains to show that the representation \eqref{weights} of the discrete part of $\mu$ in the formulation of the theorem holds. To this end, note that by expanding the fraction in \eqref{stieltjes-extension2} with the term
    	\begin{align*}
    	P_j(x) - \frac{1}{2}(x - \alpha) P_{j - 1}(x) - \frac{1}{2} z(x)\sqrt{(x - \alpha)^2 - 4\beta} P_{j - 1}(x)
    	\end{align*}
    	the denominator becomes $D(x)$. Using the recursion \eqref{pol_recursion} of the orthogonal polynomials and the Christoffel-Darboux formula \eqref{christoffel} shows
    	\begin{align*}
    	&(Q_j(x) - \frac{1}{2}(x - \alpha) Q_{j - 1}(x) + \frac{1}{2} z(x)\sqrt{(x - \alpha)^2 - 4\beta} Q_{j - 1}(x)) \\
    	&\cdot (P_j(x) - \frac{1}{2}(x - \alpha) P_{j - 1}(x) - \frac{1}{2} z(x)\sqrt{(x - \alpha)^2 - 4\beta} P_{j - 1}(x)) \\
    	={}& Q_j(x) P_j(x) - \frac{1}{2}(x - \alpha) (Q_j(x) P_{j - 1}(x) + Q_{j - 1}(x) P_j(x)) \\
    	& -z(x) \frac{1}{2}(Q_j(x) P_{j - 1}(x) - Q_{j - 1}(x) P_j(x)) \sqrt{(x - \alpha)^2 - 4\beta} + \beta Q_{j - 1}(x) P_{j - 1}(x) \\
    	={}& Q_j(x) P_j(x) + \frac{1}{2}(x - \alpha) (Q_j(x) P_{j - 1}(x) - Q_{j - 1}(x) P_j(x)) - (x -  \alpha) Q_j(x) P_{j - 1}(x) \\
    	&-z(x) \frac{h_{j - 1}}{2}\sqrt{(x - \alpha)^2 - 4\beta} + \beta Q_{j - 1}(x) P_{j - 1}(x) \\
    	={}& Q_j(x) P_j(x) - Q_{j + 1}(x) P_{j - 1}(x) + \frac{h_{j - 1}}{2}\Big((x - \alpha)  -z(x) \sqrt{(x - \alpha)^2 - 4\beta}\Big) = f(x)\,. \qedhere
    	\end{align*}
    \end{proof}

		\begin{example*}[Continuation of Example \ref{example_general}]\label{example_continued}
        It remains to compute the measure in \eqref{example_measure} from the information that the sequence of random moments $(m_1^{(n)}, \ldots, m_l^{(n)})$ converges in probability to the sequence of moments $(m_1(\mu), \ldots, m_l(\mu))$, where the measure $\mu$ is uniquely determined by having recursion coefficients $\alpha_1 = 0$, $\alpha_j = 1$ and $\beta_{j - 1} = \frac{1}{4}$ for all $j > 1$. In light of Theorem~\ref{measure_representation}, we calculate the orthogonal polynomials $P_j$ as well as the secondary polynomials $Q_j$ up to order $2$ as
        \begin{align*}
            P_0(x) &= 1\,, & Q_0(x) &= 0\,,\\
            P_1(x) &= x - \alpha_1 = x\,, & Q_1(x) &= 1\,, \\
            P_2(x) &= (x - \alpha_2)P_1(x) - \beta_1 P_0(x) = x^2 - x - \frac{1}{4}\,, & Q_2(x) &= x - \alpha_2 = x - 1 \,.
        \end{align*}
        Consequently, the polynomial $D$ is given by
        \begin{align*}
            D(x) = P_1^2(x) - P_2(x) P_0(x) = x + \frac{1}{4} \,.
        \end{align*}
        Therefore $\mu$ has a possible point mass in $-\frac{1}{4}$. In order to calculate the weight of the point mass, we determine
        \begin{align*}
            f(x) = 1 + \frac{1}{2}(x - 1 - z(x) \sqrt{x^2-2x}) \,.
        \end{align*}
        This yields the weight of the point mass as
        \begin{align*}
            \lim \limits_{x \to -\frac{1}{4}} \left|\Big(x + \frac{1}{4}\Big) \frac{f(x)}{x + \frac{1}{4}}\right| = f\Big(-\frac{1}{4}\Big) = \frac{3}{4}\,.
        \end{align*}
        Hence the measure $\mu$ is given by \eqref{example_measure}.
\end{example*}   
    We are now in the position to prove Theorem~\ref{general_lln}.
    \begin{proof}[Proof of Theorem~\ref{general_lln}]
    	After a change of variables by means of Lemma~\ref{lemma_Jacobian}, an application of Proposition \ref{conv_discrete} shows the convergence on the level of canonical coordinates. The convergence result can then be transferred back to the ordinary moments with the continuous mapping theorem.
    	
    	It remains to show the claimed representation of $\mu_q$. The measure $\mu_q$ is uniquely determined by having $\mathbf{y}^{*,q}$ as first $i_k$ canonical coordinates, as well as
    	\begin{align*}
            y_j =
            \begin{cases}
                y_1^* &, j \text{ odd} \\
                y_2^* &, j \text{ even} 
            \end{cases}
    	\end{align*}
    	for all $j > i_k$, where $y_1^*$ and $y_2^*$ were defined in \eqref{def_y_12}.
			We have to distinguish the three cases of $E$.\\
	
	\medskip
 	
    	\textbf{Case $E = \mathbb{R}$:} In this case the respresentation is a direct consequence of Proposition~\ref{measure_representation}. \\
			
			\medskip
			
			\textbf{Case $E = \mathbb{R_+}$:}
			In the case $E = \mathbb{R}_+$, \cite[p.~18]{DTV} shows that the recursion coefficients of the orthogonal polynomials corresponding to $\mu_q$ are given in terms of the canonical coordinates as
    	\begin{align*}
            \alpha_j &= z_{2j - 2} + z_{2j - 1}\,, \\
            \beta_j &= z_{2j - 1} z_{2j}\,,
    	\end{align*}
		and moreover, we have
		\begin{align}
		z_{2j + 1} = -\frac{P_{j + 1}(0)}{P_j(0)}.\label{ratio_of_ops}
		\end{align}
    	Since the odd and even parameters $z_j$ are constant for all $j > i_k$, the $\alpha_j$ and $\beta_{j - 1}$ are constant for all $j > l := \lfloor \frac{i_k + 3}{2} \rfloor$ and we may apply Proposition~\ref{measure_representation}. Observing the expansion~\eqref{expansion_d} of the polynomial $D(x)$ from Proposition \ref{measure_representation} for $i_k$ even and odd respectively, we can see that the degree of $D$ is at most $i_k + 1$. To prove that the absolutely continuous part of the measure is of the claimed form
			$$\mu_q^{ac}(x)=\frac{1}{D_q(x)}\mu_{z_1^*,z_2^*}^{ac}(x)$$
			for some polynomial $D_q$ of degree at most $i_k$, we have to show $D(0)=0$ in order to factor out the linear polynomial $p(x)=x$.
			By \eqref{ratio_of_ops},
    	\begin{align}
            P_s(0) = (-1)^s\prod \limits_{j = 0}^{s - 1}\Big(- \frac{P_{j + 1}(0)}{P_{j}(0)}\Big) = (-1)^s \prod \limits_{j = 0}^{s - 1} z_{2j + 1}\,,\label{poly_zero}
    	\end{align}
    	which yields
    	\begin{align*}
            D(0) = P_l^2(0) - P_{l - 1}(0) P_{l + 1}(0) = \left(\prod \limits_{j = 0}^{l - 2} z_{2j + 1}^2\right) z_{2l - 1}(z_{2l - 1} - z_{2l + 1}) = 0\,,
    	\end{align*}
			since $z_{2l - 1} = z_{2l + 1}$.
        Factoring out the polynomial $p(x)=x$ yields the desired result for the absolutely continuous part with the polynomial $D_q(x):=D(x)/x$ which is of degree at most $i_k$ and is strictly positive at 0 since the root of $D$ in $0$ is simple by Proposition \ref{measure_representation}.
        
        For the discrete part of the measure, let $x_0$ be a root of $D(x)$, i.e.~a number satisfying $P_l^2(x_0) = P_{l - 1}(x_0) P_{l + 1}(x_0)$. Since two consecutive orthogonal polynomials have no common roots, $x_0$ is neither a root of $P_{l - 1}$, $P_l$ or $P_{l + 1}$. We then have
        \begin{align*}
            &\Big(Q_l(x_0)P_l(x_0) - Q_{l + 1}(x_0) P_{l - 1}(x_0)\Big) P_{l + 1}(x_0) \\
            ={}& Q_l(x_0) P_l(x_0) P_{l + 1}(x_0) - Q_{l + 1}(x_0) P_l^2(x_0) \\
            ={}& -P_l(x_0)\Big(Q_{l + 1}(x_0) P_{l}(x_0) - Q_l(x_0) P_{l + 1}(x)\Big) = -P_l(x_0) h_l,
        \end{align*}
        where we have used the identity \eqref{christoffel} in the last line. Recalling the definition of the function $f$ in \eqref{func_f} we therefore obtain
        \begin{align}
            f(x_0) = -\frac{P_l(x_0) h_l}{P_{l + 1}(x_0)} + \frac{\beta_1 \cdots \beta_{l - 1}}{2} \Big(x - \alpha - z(x)\sqrt{(x - \alpha)^2 - 4\beta}\Big)\,. \label{value_f}
        \end{align}
        Plugging in the value $x_0 = 0$ and using~\eqref{poly_zero} as well as $\alpha = z_1^* + z_2^*$, $\beta = z_1^*z_2^*$ and $h_j = \beta_1 \cdots \beta_j$ we obtain
        \begin{align*}
            f(0) ={}& \frac{h_l}{z_{2l + 1}} + \frac{h_{l- 1}}{2} \Big(-z_1^* - z_2^* + \sqrt{(z_1^* + z_2^*)^2 - 4z_1^*z_2^*}\Big) \\
            ={}& h_{l - 1} \Big(\frac{z_{2l - 1} z_{2l}}{z_{2l + 1}} + \frac{1}{2} \big(-z_1^* - z_2^* + |z_1^* - z_2^*|\big)\Big) \\
            ={}& h_{l - 1} (z_2^* - z_1^*)_+\,,
        \end{align*}
        where we have used that $z_{2l - 1} = z_{2l + 1} = z_1^*$ and $z_{2l} = z_2^*$ (note that $l$ was chosen such that $2l - 1 \ge i_k + 1$).
        
        Recall that the weight of the point mass of $\mu_q$ in zero can be calculated via the formula
        \begin{align*}
            \lambda = \lim \limits_{x \to 0} \left| x\frac{f(x)}{D(x)}\right|\,.
        \end{align*}
        We now have to consider three separate cases:
        
        \begin{enumerate}
            \item If $\mu_{z_1^*, z_2^*}$ has an atom at $0$, then $z_2^* > z_1^*$ must hold. In this case, we have $f(0) > 0$ and consequently $\lambda > 0$. This means that $\mu_q$ has a point mass at zero, as well as up to $i_k$ further point masses at the roots of $D$ outside of zero.
            \item If $\mu_{z_1^*, z_2^*}$ has no atom at $0$ and $0$ is not a simple root of $D$, then $D$ has at most $i_k$ distinct roots and therefore up to $i_k$ possible point masses.
            \item If $\mu_{z_1^*, z_2^*}$ has no atom at $0$ and $0$ is a simple root of $D$, then $z_1^* \le z_2^*$, $f(0) = 0$ and $D'(0) \ne 0$ holds. We therefore obtain
                \begin{align*}
                    \lambda = \lim \limits_{x \to 0} \left| x\frac{f(x)}{D(x)}\right| = \left| \frac{f(0)}{D'(0)} \right| = 0\,.
                \end{align*}
                Consequently, $\mu_q$ has no point mass in zero and up to $i_k$ further point masses in the roots of $D$ outside of zero.
        \end{enumerate}
        Common to all three cases is that $\mu_q$ has at most $i_k$ point masses more than $\mu_{z_1^*, z_2^*}$, which yields the desired result.\\
				\medskip

\textbf{Case $E=[0,1]$:}    	In the case $E = [0, 1]$, arguments as above show that the recursion coefficients $\alpha_j$ and $\beta_{j - 1}$ are constant for all $j > l := \lfloor \frac{i_k + 4}{2}\rfloor$. Similar calculations using (see \cite{detnag2012})
			\begin{align}
            \begin{split}
    	\alpha_j &= q_{2j - 3} p_{2j - 2} + q_{2j - 2} p_{2j - 1},\\
    	\beta_j &= q_{2j - 2} p_{2j - 1} q_{2j - 1} p_{2j},
        \end{split}\label{canonical_recursion}
    	\end{align}
    	where $q_j := 1 - p_j$ and $p_{-1} := p_0 := 0$, and formula~\eqref{poly_zero} then show that $D$ is a polynomial of degree at most $i_k + 2$, satisfying $D(0) = 0$. Furthermore, it is well-known that the monic orthogonal polynomial $P_s$ of order $s$ is given by
        \begin{align*}
            P_s(x) = (\Delta_{s - 1})^{-1}
            \begin{pmatrix}
                m_0 & m_1 & \ldots & m_{s - 1} & 1 \\
                m_1 & m_2 & \cdots & m_s & x \\
                \vdots & \vdots & \ddots & \vdots & \vdots \\
                m_s & m_{s + 1} & \cdots & m_{2s - 1} & x^s
            \end{pmatrix}
        \end{align*}
        with $\Delta_{s - 1} = \det\Big( (m_{a + b})_{a, b = 0}^{s - 1}\Big)$ and $m_0:=1$, see e.g.~\cite[p.~38]{Deift}. This yields
        \begin{align*}
            P_2(1) &= (\Delta_{s - 1})^{-1}
            \begin{pmatrix}
                m_0 & m_1 & \ldots & m_{s - 1} & 1 \\
                m_1 & m_2 & \cdots & m_s & 1 \\
                \vdots & \vdots & \ddots & \vdots & \vdots \\
                m_s & m_{s + 1} & \cdots & m_{2s - 1} & 1
            \end{pmatrix}\
             \\
            &= (\Delta_{s - 1})^{-1}
            \begin{pmatrix}
                m_0 - m_1 & m_1 - m_2 & \ldots & m_{s - 1} - m_s & 0 \\
                m_1 - m_2 & m_2 - m_3 & \cdots & m_s - m_{s + 1} & 0 \\
                \vdots & \vdots & \ddots & \vdots & \vdots \\
                m_{s - 1} - m_2 & m_{s} - m_{s + 1} & \cdots & m_{2s - 2} - m_{2s - 1} & 0 \\
                m_s & m_{s + 1} & \cdots & m_{2s - 1} & 1
            \end{pmatrix}\\
            &= (\Delta_{s - 1})^{-1}\det \Big((m_{a + b} - m_{a + b + 1})_{a, b = 0}^{s - 1}\Big) \,.
    	\end{align*}
    	Combined with  \cite[Theorem~1.4.10 and Corollary~1.4.6]{dettstud1997} this results in
    	\begin{align}
            P_s(1) = \prod \limits_{j = 0}^{s - 1} q_{2j} q_{2j + 1} \label{poly_one}
    	\end{align}
    	and consequently
    	\begin{align*}
            D(1) = \left(\prod \limits_{j = 0}^{l - 2} q_{2j}^2 q_{2j + 1}^2 \right) q_{2l - 2} q_{2l - 1} (q_{2l - 2} q_{2l - 1} - q_{2l}q_{2l + 1}) = 0\,.
    	\end{align*}
    	Factoring out $x(1-x)$ in $D$ proves the representation of the absolutely continuous part.
        
        For the discrete part of $\mu_q$, we observe that by~\eqref{value_f}, \eqref{poly_zero} and \eqref{poly_one} we have
        \begin{align*}
            f(0) &= h_{l - 1} (p_2^* - p_1^*)_+,\\
            f(1) &= h_{l - 1} \Big(- p_1^* p_2^* + \frac{1}{2} \big(1 - (p_1^* + p_2^* - 2p_1^*p_2^*) - |1 - (p_1^* + p_2^*)|\big)\Big) \\
            &= h_{l - 1} (p_1^* + p_2^* - 1)_+.
        \end{align*}
        Arguments similar to the case $E = \mathbb{R}_+$ show that $\mu_q$ has point masses in the point masses of $\mu_{p_1^*, p_2^*}$, as well as up to $i_k$ additional point masses.
    \end{proof}

    \medskip

 \section{Proof of the LDP}\label{sec_LDP}
 In this section, we will prove Theorem \ref{ldp}. As before, we start with a general proposition. We note in passing that the following statement is stronger than a usual LDP in the sense that it establishes the rate function for any measurable set as an actual limit that is described as an essential infimum, instead of just lower and upper bounds as in a usual LDP. This is formulated in the following lemma.

 \begin{lemma}[Large deviation principle]\label{exact_ldp}
 	Let $\mathbb{P}_n$ be a sequence of probability measures on $\mathbb{R}^m$ and $a_n \to \infty$ a sequence such that
 	\begin{align}
 	\lim \limits_{n \to \infty} \frac{1}{a_n} \log \mathbb{P}_n(\Gamma) = - \essinf \limits_{x \in \Gamma} S(x) \label{exact_ldp_eq}
 	\end{align}
 	holds for all measurable sets $\Gamma \subset \mathbb{R}^m$ and some function $S:\R^m\to[0,\infty]$. Then $\mathbb{P}_n$ satisfies an LDP with speed $a_n$ and  rate function
 	\begin{align*}
 	I(x) := \operatorname*{ess\,lim\,inf}\limits_{y \to x} S(y) = \sup\{\essinf \limits_{y \in U} S(y) \mid x \in U, U \text{ open}\}.
 	\end{align*}
 \end{lemma}
 \begin{proof}
 	It suffices to prove three assertions:
 	\begin{enumerate}
 		\item $I$ is lower semicontinuous
 		\item $\inf \limits_{x \in G} I(x) \ge \essinf\limits_{x \in G} S(x)$ holds for all open sets $G$.
 		\item $\inf \limits_{x \in F} I(x) \le \essinf\limits_{x \in F} S(x)$ holds for all closed sets $F$.
 	\end{enumerate}
 	
 	Let $x_0 \in \mathcal{X}$ and $K \le I(x_0)$ be arbitrary with $K < \infty$. Choose $U$ open so that $\essinf \limits_{y \in U} S(y) \ge K - \varepsilon$ holds. This implies for all $x \in U$ the inequality
 	\begin{align*}
 	I(x) = \sup\{\essinf \limits_{y \in V} S(y) \mid x \in V, V \text{ open}\} \ge \essinf \limits_{y \in U} S(y) \ge K - \varepsilon.
 	\end{align*}
 	Now the lower semicontinuity follows, since for $I(x_0)$ finite, we may choose $K = I(x_0)$ and for infinite $I(x_0)$ by taking the limit $K \to \infty$. 
 	
 	Next, let $G$ be an open set with $x \in G$. Then we have
 	\begin{align*}
 	I(x) = \sup\{\essinf \limits_{y \in U} S(y) \mid x \in U, U \text{ open}\} \ge \essinf \limits_{y \in G} S(y)
 	\end{align*}
 	The assertion (2) now follows by taking the infimum over all $x \in G$ on the left-hand side.
 	
 	Finally, let $F$ be closed, fix an arbitrary finite $K \le \inf \limits_{y \in F} I(y)$ and let $\varepsilon > 0$. By the definition of $I$ we can find for each $x \in F$ an open set $U_x\ni x$ such that
 	\begin{align*}
 	\essinf \limits_{y \in U_x} S(y) \ge K - \varepsilon.
 	\end{align*}
 	As $\{U_x\}_{x \in F}$ is an open covering of the closed set $F$ and $\mathbb{R}^m$ is $\sigma$-compact, we can find a countable subcovering $U_{x_1}, U_{x_2}, \ldots$ for $F$. This implies
 	\begin{align*}
 	\essinf \limits_{x \in F} S(x) \ge \inf \limits_{n \in \mathbb{N}} \essinf \limits_{x \in U_{x_n}} S(x) \ge K - \varepsilon\,,
 	\end{align*}
 	and thus we get $\essinf \limits_{x \in F} S(x) \ge K$. The result now follows by choosing $K := \inf \limits_{y \in F} I(y)$ if the infimum or else letting $K \to \infty$.
 \end{proof}

 \begin{prop}[Large deviations] \label{ldp}
 	Let $\mathbb P_n$ be given by \eqref{Laplace_density} satisfying \eqref{Laplace_assumption} and assume that $W: \mathbb{R}^m \to \mathbb{R} \cup \{\infty\}$ and $R: \mathbb{R}^m \to [0, \infty ])$ are measurable functions such that 
 	\begin{enumerate}
 		\item $W$ is essentially lower bounded,
 		\item $R$ is almost surely strictly positive.
 	\end{enumerate}
 	Then $\mathbb{P}_n$ satisfies for all measurable $\Gamma \subset \mathbb{R}^m$
 	\begin{align*}
 	\lim \limits_{n \to \infty} \frac{1}{n} \log \mathbb{P}_n(\Gamma) = -\essinf \limits_{x \in \Gamma} S(x)\,,
 	\end{align*}
 	where
 	\begin{align*}
 	S(x) := W(x) - \essinf\limits_{y \in \mathbb{R}^m} W(y),
 	\end{align*}
 	
 \end{prop}

 \begin{proof}
 	We will first show that
 	\begin{align}
 	\lim \limits_{n \to \infty} \frac{1}{n} \log \int_\Gamma e^{-nW(x)} R(x) \, dx = -\essinf\limits_{x \in \Gamma} W(x) \label{ldp_limit}
 	\end{align}
 	holds for all measurable sets $\Gamma \subset \mathbb{R}^m$. Set $\alpha := \essinf\limits_{x \in \Gamma} W(x)$, and note that the case $\alpha = \infty$ is trivial. If $\alpha < \infty$ we have $\Gamma \subset \{W \ge \alpha\} \cup N$ for some Lebesgue nullset $N$. This implies with $n_0$ from \eqref{Laplace_assumption}
 	\begin{align}
 	\limsup \limits_{n \to \infty} \frac{1}{n} \log \int_\Gamma e^{-nW(x)} R(x) \, dx
 	\le{}& \limsup \limits_{n \to \infty} \frac{1}{n} \log \int\limits_{\{W \ge \alpha\}} e^{-n_0W(x)} R(x) e^{-(n - n_0)\alpha} \, dx \nonumber\\
 	\le{}& \limsup \limits_{n \to \infty} \left(\frac{\log C}{n} - \alpha \frac{n - n_0}{n} \right) = -\alpha. \label{ldp_upper_bound}
 	\end{align}\bigskip
 	
 	Let $\varepsilon > 0$ be arbitrary. By definition of $\alpha$ we have $\text{vol}_m(\Gamma \cap \{W \le \alpha + \varepsilon\}) > 0$. Since $R$ is almost surely strictly positive, this implies
 	\begin{align*}
 	0 < \int\limits_{\Gamma \cap \{W \le \alpha + \varepsilon\}} e^{-n_0W(x)} R(x) \, dx < \infty,
 	\end{align*}
 	from which we can conclude
 	\begin{align*}
 	&\liminf \limits_{n \to \infty} \frac{1}{n} \log \int_\Gamma e^{-nW(x)} R(x) \, dx \\
 	\ge{}& \liminf \limits_{n \to \infty} \frac{1}{n} \log \int\limits_{\Gamma \cap \{W \le \alpha + \varepsilon\}} e^{-n_0W(x)} e^{-(n-n_0)(\alpha + \varepsilon)}R(x) \, dx = -(\alpha + \varepsilon).
 	\end{align*}
 	Now let $\varepsilon \to 0$ and combine this equation with (\ref{ldp_upper_bound}) to conclude that (\ref{ldp_limit}) holds. This yields the desired result, as
 	\begin{align*}
 	\lim \limits_{n \to \infty} \frac{1}{n} \log \mathbb{P}_n(\Gamma) ={}& \lim \limits_{n \to \infty}\left\{ \frac{1}{n} \log \int_\Gamma e^{-nW(x)}R(x) \, dx - \frac{1}{n} \log \int_{\mathbb{R}^m}e^{-nW(x)} R(x) \, dx \right\}\\
 	={}&-\essinf \limits_{x \in \Gamma} W(x) + \essinf\limits_{y \in \mathbb{R}^m} W(y) =  -\essinf \limits_{x \in \Gamma} S(x).
 	\end{align*}
 \end{proof}

 \begin{proof}[Proof of Theorem~\ref{general_dist_ldp}]
 	This theorem is a consequence of the previously obtained results and follows in three simple steps. Firstly, change coordinates from $(m_1^{(n)}, \ldots, m_l^{(n)})$ to the vector of unconstrained canonical coordinates $\mb y_l^{\C,(n)}$ by means of Corollary~\ref{corollary_density}. Secondly, apply Proposition~\ref{ldp} and Lemma \ref{exact_ldp_eq} to obtain an LDP for the vector of canonical coordinates.
 	The last step then is to transfer the LDP from canonical coordinates to ordinary coordinates using the contraction principle.
 	
 	The compactness of level sets of $I$ is clear for $E=[0,1]$ and follows for unbounded $E$ from the integrability conditions \eqref{int_conditions}.
 \end{proof}

\section{Proofs of the CLT and MDP}\label{sec_CLT}
In this section we prove the  central limit theorem and the moderate deviations principle of Section \ref{sec_general_dist}. As before, for later use we first give results for the general density $\mathbb{P}_n$ of the form \eqref{Laplace_density}, satisfying \eqref{Laplace_assumption}.

\begin{prop}[Convergence to normal distribution]\label{conv_normal}
	Let the same conditions and notations as in Proposition~\ref{conv_discrete} hold and let $X_n$ be a random variable having distribution with density $\mathbb{P}_n$ from \eqref{Laplace_density}. Let $	\theta_n^*$ be such that
	\begin{align*}
	\theta_n^* \in \argmin\{\|x - X_n\| \mid x \in \{\theta_1, \ldots, \theta_p\}\}
	\end{align*}
	with the standard Euclidean norm $\|\cdot\|$ and
	\begin{align*}
	Y_n := \Hess W(\theta_n^*)^{1/2}\sqrt{n}(X_n - \theta_n^*).
	\end{align*}
	If $N \sim \mathcal{N}(0, 1)$, then as $n\to\infty$
	\begin{align*}
	d_{TV}(Y_n, N) \to 0,
	\end{align*}
	where $d_{TV}$ denotes the total variation distance 
	\begin{align*}
	d_{TV}(X, Y) := \sup \limits_{A \in \mathcal{B}(\mathbb{R}^m)} |\p{X \in A} - \p{Y \in A}|.
	\end{align*}
\end{prop}

\begin{proof}
	Choose $\varepsilon > 0$ so that the conditions
	\begin{enumerate}
		\item $B_\varepsilon(\theta_q) \cap B_\varepsilon(\theta_{q'}) = \emptyset$ for $q \ne q'$,
		\item for all $x \in B_\e(\theta_q)$ we have $\Hess W(x) \ge M_q$ for some postive definite matrices $M_q$,
	\end{enumerate}
	are satisfied and let $A \subset \mathbb{R}^m$ be a Borel set. Then
	\begin{align*}
	&|\p{Y_n \in A} - \p{N \in A}|\le{} |\p{Y_n \in A, \|X_n - \theta_n^*\| < \varepsilon} - \p{N \in A}| + \p{\|X_n - \theta_n^*\| \ge \varepsilon}.
	\end{align*}
	By Proposition~\ref{conv_discrete} and Portmanteau's theorem, the last summand converges to zero. Set
	\begin{align*}
	D := \sum \limits_{q=1}^p R(\theta_q) \det \Hess W(\theta_q)^{-1/2}  \qquad \text{ and } \qquad d_n := D^{-1} \left(\frac{n}{2\pi}\right)^{m/2} \int e^{-nW(x)} R(x) \, dx.
	\end{align*}
	Then $d_n \to 1$ holds by calculations analogous to the ones used to prove Proposition~\ref{conv_discrete}. To simplify notation, define
	\begin{align*}
	f_n^q(y) &:= \exp(-n W(\theta_q + \Hess W(\theta_q)^{-1/2} n^{-1/2} y)) R(\theta_q + \Hess W(\theta_q)^{-1/2} n^{-1/2} y), \\
	S_n^q &:= \Hess W(\theta_q)^{1/2} n^{1/2} B_\e(0), \\
	K_n^q &:= \Hess W(\theta_q)^{-1/2} n^{-1/2} A
	\end{align*}
	and introduce the decomposition
	\begin{align}
	\begin{split}
	\exp(-y^ty/2)={}& D^{-1}\sum \limits_{q=1}^p R(\theta_q) \det \Hess W(\theta_q)^{-1/2} \exp(-y^ty/2) 1_{S_n^q}(y) \\
	&+ D^{-1} \sum \limits_{q=1}^p R(\theta_q) \det \Hess W(\theta_q)^{-1/2} \exp(-y^ty/2) 1_{\lb S_n^q\rb^{c}}(y).
	\end{split}\label{decomp_gauss}
	\end{align}
	Finally, note that 
	\begin{align*}
	\{Y_n \in A, \|X_n - \theta_n^*\|<\e\} = \bigcup \limits_{q=1}^p \{X_n - \theta_q \in K_n^q \cap B_\e(0)\}.
	\end{align*}
	Combining all this yields
	\begin{align*}
	&|\p{Y_n \in A, \|X_n - \theta_n^*\| < \varepsilon} - \p{N \in A}| \\
	={}& \left| \sum \limits_{q=1}^p \p{X_n - \theta_q \in K_n^q \cap B_\e(0)} - \p{N \in A}\right| \\
	={}&\left| \frac{n^{m/2}}{(2\pi)^{m/2}d_n D} \sum \limits_{q=1}^p \int\limits_{K_n^q \cap B_\e(0)} \exp(-n W(\theta_q + x)) R(\theta_q + x) \, dx - \frac{1}{(2\pi)^{m/2}} \int_A \exp(-x^tx/2) \, dx\right| \\
	={}&(2\pi)^{-m/2} \left|\int\limits_A \sum \limits_{q=1}^p d_n^{-1} D^{-1} f_n^q(y) 1_{S_n^q}(y) \det \Hess W(\theta_q)^{-1/2} - \exp(-y^ty/2) \, dy\right|.
	\end{align*}
	By the decomposition \eqref{decomp_gauss} we therefore obtain
	\begin{align*}
	&|\p{Y_n \in A, \|X_n - \theta_n^*\| < \varepsilon} - \p{N \in A}| \\
	\le{}&(2\pi)^{-m/2} \sum \limits_{q=1}^p D^{-1} \det \Hess W(\theta_q)^{-1/2}\left|\int_A \big(d_n^{-1} f_n^q(y) - \exp(-y^ty/2) R(\theta_q)\big) 1_{S_n^q}(y)\, dy\right| \\
	&+ D^{-1} \sum \limits_{q=1}^p R(\theta_q) \det \Hess W(\theta_q)^{-1/2} P(N \notin S_n^q) \\
	\le{}&(2\pi)^{-m/2} \sum \limits_{q=1}^p D^{-1} \det \Hess W(\theta_q)^{-1/2}\int_{S_n^q} \big|d_n^{-1} f_n^q(y) - \exp(-y^ty/2) R(\theta_q)\big|\, dy \\
	&+ D^{-1} \sum \limits_{q=1}^p R(\theta_q) \det \Hess W(\theta_q)^{-1/2} \p{N \notin S_n^q}.
	\end{align*}
	Note that this bound is uniform in $A$. Clearly $\p{N \notin S_n^q}\to0$ as $n\to\infty$. On $S_n^q$ the pointwise convergence
	\begin{align*}
	d_n^{-1} f_n^q(y) \to \exp(-y^ty/2) R(\theta_q)
	\end{align*}
	holds and the claimed result follows from the dominated convergence theorem, using the dominating functions
	\begin{align*}
	R(\theta_q) \left(\left(\sup \limits_{n} d_n^{-1}\right) \exp(-y^t M_q y/2) + \exp(-y^t y/2)\right),\qquad q=1,\dots,p.
	\end{align*}
\end{proof}

\begin{proof}[Proof of Theorem~\ref{general_dist_clt}]
	After a change of coordinates by Corollary~\ref{corollary_density}, an application of Proposition~\ref{conv_normal} yields a central limit theorem for the canonical coordinates,
    \begin{align*}
        \sqrt{n}(\operatorname{Hess^\C}\sum \limits_{j = 1}^l W_j)^{1/2}(\mathbf{y}_l^*)(\mathbf{y}_l^{\mathcal{C},(n)} - \mathbf{y}_{l}^*) \xrightarrow{d} \mathcal{N}(0, I_{l - k})\,,
    \end{align*}
    where the first $i_k - k$ coordinates of $\mathbf{y}_l^*$ are given by
    \begin{align*}
        \argmin \limits_{y \in \{\mathbf{y}^{*, 1}, \ldots, \mathbf{y}^{*, p}\}} \|\mathbf{y}_l^{\mathcal{C}, (n)} - y\|
    \end{align*}
    and the remaining coordinates for $j>i_k$ are
    \begin{align*}
        (\mathbf{y}_l^*)_j =
        \begin{cases}
            y_1^* &, j \text{ odd} \\
            y_2^* &, j \text{ even} \\
        \end{cases}\,.
    \end{align*}
    To transfer this CLT for $\mb y_l^{\C,(n)}$ to one for $\mb m_l^{\C,(n)}$, we use a variant of the delta-method, which is a common technique in mathematical statistics:

Let 	 $(X_n)_n$, $(Y_n)_n$ be 	two sequences of random variables with $X_n, Y_n\in\R^m$ and $\sqrt{n}(X_n-Y_n)\to N\sim \mathcal{N}(0,\Sigma)$ in distribution as $n\to\infty$. Let  $f:\R^m\to\R^{m}$ be a continuously differentiable function. Assume moreover that $Y_n$ is almost surely contained in an $n$-independent set $B$ on which $y\mapsto Df(y)$ is uniformly continuous and $\|Df(y)\|_{op}\geq c>0$ for some $c$ and $y\in B$, $\|\cdot\|_{op}$ denoting the operator norm. Then $\sqrt nDf(Y_n)^{-1}(f(X_n)-f(Y_n))$ is asymptotically $\mathcal{N}(0,\Sigma)$-distributed. To see this, let $f_j$ denote the $j$-th component of the vector-valued $f$, $j=1,\dots,m$ and note that by the mean value theorem,
\begin{align*}
\sqrt n(f_j(X_n)-f_j(Y_n))=\langle\nabla f_j(\xi_{j,n}),\sqrt n(X_n-Y_n)\rangle
\end{align*}
		for some $\xi_{j,n}$ in the line segment $[X_n,Y_n]$. Since $\sqrt n(X_n-Y_n)$ is asymptotically normal, we have $\xi_{j,n}-Y_n\to0$ in probability. Let $M_n$ be the matrix built from the rows $\nabla f_j(\xi_{j,n}),\,j=1,\dots,m$. By the uniform continuity, $Df(M_n)-Df(Y_n)\to0$ in probability and the asymptotic normality of $\sqrt nDf(Y_n)^{-1}(f(X_n)-f(Y_n))$ follows from $\|Df^{-1}\|_{op}\leq c^{-1}$ on $B$ and Slutsky's theorem.
		
		In our case, $X_n:=\mb y_l^{\C,(n)}$, $Y_n:=\mb y_l^*$ and $f:=\varphi_l^{E,\C}$. The assumptions on $Df$ are easily verified for $B$ being the set of the $p$ possible values of $\mb y_l^*$.  There is a small subtlety to consider: As the minimizers $\mb y_l^*$ and $\m_l^*$ are determined via $\mb y_l^{\C,(n)}$ and $\m_l^{(n)}$ on different spaces, not necessarily $\varphi_l^{E,\C}(\mb y_l^*)=\m_l^*$ holds. However, by the continuity of $\varphi_l^{E, \mathcal{C}}$ the minimizers are identical whenever $\mathbf{m}_l^{(n)}$ is in a sufficiently small neighborhood $U$ of $\{\mathbf{m}_{l}^{*, 1}, \ldots, \mathbf{m}_{l}^{*, p}\}$. Due to Theorem~\ref{general_lln} we have $\mathbb{P}(\mathbf{m}_{l}^{(n)} \in U) \to 1$, which implies the stated result.

\end{proof}
Let us now turn to moderate deviations. The general result is

    \begin{prop}[Moderate Deviations]\label{mdp}
	With the assumptions and notation of Proposition~\ref{conv_normal} for any sequence $a_n \to \infty$ with $a_n = o(\sqrt{n})$, the sequence of random variables
	\begin{align*}
	Z_n := \Hess W(\theta_n^*)^{1/2} a_n (X_n - \theta_n^*)
	\end{align*}
	satisfies with $b_n := \frac{n}{a_n^2}$
	\begin{align*}
	\lim \limits_{n \to \infty} \frac{1}{b_n} \log \mathbb{P}(Z_n \in \Gamma) = -\essinf \limits_{x \in \Gamma} \frac{1}{2} \|x\|_2^2
	\end{align*}
	for any measurable set $\Gamma$.
\end{prop}

\begin{proof} Let $S_\e^q$ denote the ellipsoid
	\begin{align}
	S_\e^q := \{x \in \mathbb{R}^m \mid \|\Hess W(\theta_q)^{1/2}(x - \theta_q)\| < \varepsilon\} = \theta_q + \Hess W(\theta_q)^{-1/2} B_\e(0).\label{def_S_e}
	\end{align}
	As in the proof of Proposition~\ref{conv_normal}, let $0 < \varepsilon < 1$ be so small that the following conditions hold:
	\begin{enumerate}
		\item $S_\varepsilon^q \cap S_\e^{q'} = \emptyset$ for $q \ne q'$,
		\item $\theta_q = \argmin \limits_{y \in S_\e^q} W(y)$ holds for all $1 \le q \le p$,
		\item for all $x \in S_\e^q$ we have $\Hess W(x) \ge M_q$ for some postive definite matrices $M_q$,
		\item there is a constant $K > 0$ such that $R(x) \le K$ holds for all $x \in S_\varepsilon := \bigcup \limits_{q=1}^p S_\e^q$,
		\item $\inf \limits_{y \in S_\e^q} R(y) > 0$ holds for some $1 \le q \le p$.
	\end{enumerate}
	For an arbitrary Borel set $\Gamma$ we use the decomposition
	\begin{align}
	\p{Z_n \in \Gamma} ={}& \sum \limits_{q=1}^p \p{\Hess W(\theta_q)^{1/2} a_n (X_n - \theta_q) \in \Gamma, X_n \in S_\e^q} \notag\\
	&+\p{Z_n \in \Gamma, X_n \notin S_\varepsilon}.\label{prob_decomposition}
	\end{align}
	Observing that
	\begin{align}
	\max \limits_{j = 1}^N \frac{1}{b_n} \log A_j \le \frac{1}{b_n} \log \left(\sum \limits_{j = 1}^N A_j\right) \le \max \limits_{j = 1}^N \frac{1}{b_n} \log A_j + \frac{\log N}{b_n}, \label{log_inequality}
	\end{align}
	holds for any $A_1, \ldots, A_N \ge 0$, we see that only the largest of the probabilities in \eqref{prob_decomposition} matter.
    As will become apparent in the course of the proof, the normalization constant of the density $\exp(-nW(x)) R(x)$ satisfies
    \begin{align}
        \lim \limits_{n \to \infty} \frac{1}{b_n} \log \int a_n^m e^{-nW(x)} R(x) \, dx = 0,\label{mdp_normalization}
    \end{align}
    which means that on the exponential scale of a large deviations principle, the integration constant roughly equals $a_n^{-m}$. In order to deal with $\p{\Hess W(\theta_q)^{1/2} a_n (X_n - \theta_q) \in \Gamma, X_n \in S_\e^q}$, we are therefore interested in approximations for the term
    \begin{align}
        a_n^m \int_{S^q_\e} 1_\Gamma(\Hess W(\theta_q) a_n(x - \theta_q)) \exp(-nW(x)) R(x) \, dx \label{mdp_limit_scaled}\,.
    \end{align}
	Set $\alpha := \essinf \limits_{x \in \Gamma} \|x\|$, then we eventually want to prove $\lim \limits_{n \to \infty} \frac{1}{b_n} \log \p{Z_n \in \Gamma} = -\alpha^2/2$. The case $\alpha = \infty$ corresponds to a nullset and the claimed result obviously holds. In the case $\alpha < \infty$ we start with an upper bound of \eqref{mdp_limit_scaled}. Define
	\begin{align*}
	\lambda_q^- &:= \lambda_q^-(\varepsilon) := \inf \limits_{y \in S_\e^q} \lambda_{\text{min}} \Big(\Hess W(\theta_q)^{-1/2} \Hess W(y) \Hess W(\theta_q)^{-1/2}\Big) \\
	\lambda_q^+ &:= \lambda_q^+(\varepsilon) := \sup \limits_{y \in S_\e^q} \lambda_{\text{max}} \Big(\Hess W(\theta_q)^{-1/2} \Hess W(y) \Hess W(\theta_q)^{-1/2}\Big),
	\end{align*}
	where $\lambda_{\text{min}}(M)$ and $\lambda_{\text{max}}(M)$ denote the smallest and largest eigenvalue of a symmetric square matrix $M$, respectively. By continuity $\lambda_q^-$ and $\lambda_q^+$ both converge to 1 as $\varepsilon \to 0$ and by Taylor's theorem we have for all $x \in S_\e^q$
	\begin{align*}
	&W(x)=\\
	&\frac{1}{2}(\Hess W(\theta_q)^{1/2} x)^t \big( \Hess W(\theta_q)^{-1/2} \Hess W(\theta_q + \eta(x - \theta_q)) \Hess W(\theta_q)^{-1/2}\big) (\Hess W(\theta_q)^{1/2} x)
	\end{align*}
	for some $\eta = \eta(x) \in (0, 1)$, where we used our assumption $W(\theta_q)=0$. Since $x \in S_\e^q$ we have $x = \theta_q + \Hess W(\theta_q)^{-1/2} v$ for some vector $v$ with $\|v\| \le \varepsilon$. Therefore
	\begin{align*}
	\theta_q + \eta (x - \theta_q) = \theta_q + \Hess W(\theta_q) ^{-1/2} (\eta v) \in S_\e^q
	\end{align*}
	and we can conclude 
	\begin{align*}
	\lambda_q^-(\varepsilon) \frac{\|\Hess W(\theta_q)^{1/2}(x-\theta_q)\|^2}{2} \le W(x) \le \lambda_q^+(\varepsilon)\frac{\|\Hess W(\theta_q)^{1/2}(x-\theta_q)\|^2}{2}.
	\end{align*}
	This yields with $K$ from (4)
	\begin{align*}
	&a_n^m \int_{S_\e^q}1_\Gamma(\Hess W(\theta_q)^{1/2} a_n (x - \theta_q)) \exp(-nW(x)) R(x) \, dx \\
	\le{}&a_n^m\int 1_\Gamma(\Hess W(\theta_q)^{1/2} a_n (x - \theta_q)) \exp\big({-}n \lambda_q^{-}\|\Hess W(\theta_q)^{1/2}(x - \theta_q)\|^2/2\big) K \, dx \\
	={}& \frac{a_n^m K}{n^{m/2} \sqrt{\det \Hess W(\theta_q)}} \int 1_\Gamma(a_n t/\sqrt{n}) \exp\big({-}\lambda_q^-\|t\|^2/2\big) \, dt \\
	\le{}& \frac{b_n^{-m/2} K}{\sqrt{\det \Hess W(\theta_q)}} \int 1_{[n\alpha^2 /(2a_n^2),\infty)}(\|t\|^2/2) \exp\Big({-} (1 - \varepsilon)\lambda_q^-n \alpha^2/(2a_n^2) - \varepsilon \lambda_q^-\|t\|^2/2\Big) \, dt \\
	\le{}& \frac{b_n^{-m/2} K}{\sqrt{\det \Hess W(\theta_q)}} \exp\Big({-} (1 - \varepsilon)\lambda_q^-b_n\alpha^2/2\Big) \left(\frac{2\pi}{\varepsilon \lambda_q^-}\right)^{m/2}.
	\end{align*}
	Consequently, it follows
	\begin{align*}
	&\limsup\limits_{n \to \infty} \frac{1}{b_n} \log \left(a_n^m \int_{S_\e^q}1_\Gamma(\Hess W(\theta_q)^{1/2} a_n (x - \theta_q)) e^{-nW(x)} R(x) \, dx\right) \\
	\le{}& -(1-\varepsilon) \lambda_q^- \frac{\alpha^2}{2}. \labelonce{moderate_1}
	\end{align*}
	To find a lower bound of \eqref{mdp_limit_scaled}, let $n$ be so large that $\frac{\alpha + \varepsilon}{a_n} < \varepsilon$ holds. Then
	\begin{align*}
	&a_n^m \int_{S_\e^q}1_\Gamma(\Hess W(\theta_q)^{1/2} a_n (x - \theta_q)) e^{-nW(x)} R(x) \, dx \\
	\ge{}&a_n^m\inf \limits_{y \in S_\e^q} R(y) \int 1_\Gamma(\Hess W(\theta_q)^{1/2} a_n x)1_{[0,\e)}(\|\Hess W(\theta_q)^{1/2}x\|) e^{{-}n\lambda_q^+\|\Hess W(\theta_q)^{1/2}x\|^2/2} \, dx \\
	\ge{}& \inf \limits_{y \in S_\e^q} R(y) \frac{b_n^{-m/2}}{\sqrt{\det \Hess W(\theta_q)}} \int 1_\Gamma(a_n t/\sqrt{n})1_{[0,\a+\e)}(\|a_n t/\sqrt{n}\|) \exp\big({-}\lambda_q^+\|t\|^2/2\big) \, dt \\
	\ge{}&\inf \limits_{y \in S_\e^q} R(y) \frac{b_n^{-m/2}}{\sqrt{\det \Hess W(\theta_q)}} \exp\big({-}\lambda_q^+ n(\alpha + \varepsilon)^2/(2a_n^2)\big)  \cdot \text{vol}_m \Big(\sqrt{b_n} \cdot \big(\Gamma \cap B_{\alpha + \varepsilon}(0)\big)\Big) \\
	={}&\inf \limits_{y \in S_\e^q} R(y) \big(\det \Hess W(\theta_q)\big)^{-1/2} \exp\big({-}\lambda_q^+ b_n (\alpha + \varepsilon)^2/2\big)  \cdot \text{vol}_m \Big(\Gamma \cap B_{\alpha + \varepsilon}(0)\Big).
	\end{align*}
	By definition of $\alpha$, $\Gamma \cap B_{\alpha + \varepsilon}(0)$ has positive volume. Therefore we get 
	\begin{align}
	\liminf\limits_{n \to \infty} \frac{1}{b_n} \log \left(a_n^m \int _{S_\e^q}1_\Gamma(\Hess W(\theta_q)^{1/2} a_n (x - \theta_q)) e^{-nW(x)} R(x) \, dx\right)
	\ge - \lambda_q^+ \frac{(\alpha + \varepsilon)^2}{2}, \label{moderate_2}
	\end{align}
	if $\inf \limits_{y \in S_\e^q} R(y)>0$ and $-\infty$ else. 
	
	For the term $\p{Z_n \in \Gamma, X_n \notin S_\varepsilon}$ in \eqref{prob_decomposition} we obtain
	\begin{align*}
	&\limsup \limits_{n \to \infty} \frac{1}{b_n} \log \left( a_n^m \int_{(S_\varepsilon)^c} e^{-nW(x)} R(x) \, dx\right) \\
	\le{}&\limsup \limits_{n \to \infty} \frac{1}{b_n} \log \left(a_n^m \int \exp(-n_0 W(x)) R(x) \, dx \exp\Big({-}(n-n_0) \inf \limits_{y \notin S_\varepsilon} W(y)\Big)\right),
	\end{align*}
	where $n_0$ is the fixed number from \eqref{Laplace_assumption}. Thus
	\begin{align*}
	&\limsup \limits_{n \to \infty} \frac{1}{b_n} \log \left( a_n^m \int_{(S_\varepsilon)^c} e^{-nW(x)} R(x) \, dx\right) \\
	\le{}& \limsup \limits_{n \to \infty} \frac{a_n^2}{n} \left(\log(a_n^m) - (n - n_0) \inf \limits_{y \notin S_\varepsilon} W(y)\right) \\
	\le{}& \limsup \limits_{n \to \infty} \frac{a_n^2\big(m\log(a_n) - n/2 \inf \limits_{y \notin S_\varepsilon} W(y)\big)}{n} + \limsup \limits_{n \to \infty} -a_n^2 (1 - n_0/n) \inf \limits_{y \notin S_\varepsilon} W(y) = -\infty. \labelonce{moderate_3}
	\end{align*}
	In order to see the last equality, note that the first term is negative for sufficiently large $n$ and the second term diverges to $-\infty$. In view of \eqref{log_inequality} we conclude that $\p{Z_n \in \Gamma, X_n \notin S_\varepsilon}$ is negligible.

It remains to show \eqref{mdp_normalization}.	Using \eqref{log_inequality} in combination with \eqref{moderate_1} and \eqref{moderate_3} for $\Gamma = \mathbb{R}^m$ yields
	\begin{align*}
	&\limsup \limits_{n \to \infty} \frac{1}{b_n} \log \int a_n^m e^{-nW(x)} R(x) \, dx \\
	={}& \max \bigg\{\Big(\max \limits_{q=1}^p \limsup\limits_{n \to \infty} \frac{1}{b_n} \log \int_{S_\e^q} a_n^m e^{-nW(x)} R(x) \, dx\Big), \\
	&\qquad \qquad \limsup \limits_{n \to \infty} \frac{1}{b_n} \log\int_{(S_\varepsilon)^c}  a_n^me^{-nW(x)} R(x) \, dx\bigg\} \le 0.
	\end{align*}
	On the other hand, we obtain from \eqref{moderate_2} and \eqref{log_inequality} that
	\begin{align*}
	&\liminf \limits_{n \to \infty} \frac{1}{b_n} \log \int a_n^m e^{-nW(x)} R(x) \, dx \ge \max \limits_{q=1}^p \liminf \limits_{n \to \infty}  \frac{1}{b_n} \log \int_{S_\e^q} e^{-nV(x)} R(x) \, dx \\
	\ge{}& \max \limits_{q=1}^p \Big\{-\lambda_q^+ \frac{\varepsilon^2}{2}\Big\}.
	\end{align*}
	Letting $\varepsilon \to 0$, \eqref{mdp_normalization} follows.
	Combining \eqref{moderate_1}, \eqref{log_inequality} and \eqref{mdp_normalization} gives
	\begin{align*}
	&\limsup \limits_{n \to \infty} \frac{1}{b_n} \log \p{Z_n \in \Gamma} \le \limsup \limits_{n \to \infty} \frac{1}{b_n} \log \left(\p{X_n \notin S_\varepsilon} + \sum \limits_{q=1}^p \p{Z_n \in \Gamma, X_n \in S_\e^q}\right)\\
	\le{}&\limsup \limits_{n \to \infty}\; \max\bigg\{\Big(\max \limits_{q=1}^p \frac{1}{b_n} \log \int_{S_\e^q} a_n^m 1_\Gamma(\Hess W(\theta_q)^{1/2} a_n (x - \theta_q)) e^{-nW(x)} R(x) \, dx\Big), \\
	& \qquad \qquad  \frac{1}{b_n} \log \int_{(S_\e)^c} a_n^m e^{-nW(x)} R(x) \, dx\bigg\} \\
	&+ \limsup \limits_{n \to \infty} \frac{\log(p + 1)}{b_n} + \limsup \limits_{n \to \infty}\Big\{- \frac{1}{b_n} \log \int a_n^m e^{-nW(x)} R(x) \, dx\Big\} \\
	\le{}& \max \limits_{q=1}^p\bigg\{ -(1 - \varepsilon)\lambda_q^- \frac{\alpha^2}{2}\bigg\}.
	\end{align*}
	Letting $\varepsilon \to 0$ we now get
	\begin{align}
	\limsup \limits_{n \to \infty} \frac{1}{b_n} \log \p{Z_n \in \Gamma} \le -\frac{\alpha^2}{2}. \label{moderate_upper}
	\end{align}
	Similarly, \eqref{moderate_2}, \eqref{log_inequality} and \eqref{mdp_normalization} yield
	\begin{align*}
	&\liminf \limits_{n \to \infty} \frac{1}{b_n} \log \p{Z_n \in \Gamma} \ge \liminf \limits_{n \to \infty} \frac{1}{b_n} \log \left(\sum \limits_{j = 1}^n \p{Z_n \in \Gamma, X_n \in S_\e^q}\right) \\
	\ge{}& \max \limits_{q=1}^p \liminf \limits_{n \to \infty} \frac{1}{b_n} \log \int_{S_\e^q} a_n^m 1_\Gamma(\Hess W(\theta_q)^{1/2} a_n (x - \theta_q)) e^{-nW(x)} R(x) \, dx \\
	\ge{}& \max \limits_{q=1}^p \Big\{-\lambda_q^+ \frac{(\alpha + \varepsilon)^2}{2}\Big\}.
	\end{align*}
	Letting $\varepsilon \to 0$ we thus obtain
	\begin{align}
	\liminf \limits_{n \to \infty} \frac{1}{b_n} \log \p{Z_n \in \Gamma} \ge -\frac{\alpha^2}{2}. \label{moderate_lower}
	\end{align}
	Combining \eqref{moderate_upper} and \eqref{moderate_lower} now finally shows
	\begin{align*}
	\lim \limits_{n \to \infty} \frac{1}{b_n} \log \p{Z_n \in \Gamma} = -\frac{\alpha^2}{2} = -\frac{1}{2} \Big(\essinf \limits_{x \in \Gamma} \|x\|\Big)^2 = -\essinf \limits_{x \in \Gamma} \frac{\|x\|^2}{2}.
	\end{align*}
\end{proof}

\begin{proof}[Proof of Theorem~\ref{general_dist_mdp}]
	After a change of coordinates by Corollary~\ref{corollary_density}, an application of Proposition~\ref{mdp} and Lemma \ref{exact_ldp} yields the MDP for
	\begin{align*}
	Y_n:=a_nH(\mb y^*_l) \cdot (\mb y_l^{\C,(n)}-\mb y^*_l)
	\end{align*} 
	with good rate function $I(x)=-\frac12 \|x\|^2$, speed $b_n = \frac{n}{a_n^2}$ and
    \begin{align*}
        H(\mb y_l) := \Big(\Hess^\C\sum \limits_{j = 1}^l W_j\Big)^{1/2}(\mb y_l)\,.
    \end{align*}
    To transfer this to an MDP for the ordinary moments, we first argue that we have the MDP with same rate function and same speed for $Y_n$ under the conditioned probability measures
\begin{align*}
\P^q:=\P(\cdot|\mb y^*_l=\mb y^*_{l,q}),\quad q=1,\dots,p,
\end{align*}
where $\mb y^*_{l,1},\dots,\mb y^*_{l,p}$ are the $p$ possible values of $\mb y^*_l$ and $\P$ denotes the underlying probability measure. To see this, note that under $\P^q$ we have
\begin{align*}
Y_n=Y_n^q:=a_n H(\mb y^*_{l, q}) (\mb y_l^{\C,(n)}-\mb y^*_{l,q}) = a_n \Big( H(\mb y^*_{l, q}) \mb y_l^{\C,(n)} - H(\mb y^*_{l, q})\mb y^*_{l,q}\Big).
\end{align*}
By Theorem \ref{general_lln} there are $0<c_1<c_2$ such that for $n$ large enough and all $q=1,\dots,p$
\begin{align*}
c_1\leq\P(\mb y^*_l=\mb y^*_{l,q})\leq c_2.
\end{align*}
The upper bound in \eqref{def_LDP} for $Y_n^q$  follows for $n$ large from 
\begin{align*}
\P^q(Y_n^q\in\Gamma)\leq c_1^{-1}\P(Y_n\in\Gamma).
\end{align*}
For the lower bound, for $n$ large enough
\begin{align*}
\P^q(Y_n^q\in\Gamma)\geq c_2^{-1}\,\P(Y_n\in\Gamma,\,Y_n\in S_\e^q)
\end{align*}
with $S_\e^q$ from \eqref{def_S_e} and $\e$ chosen as in the proof of Proposition \ref{mdp}. It was shown in \eqref{prob_decomposition}, \eqref{mdp_normalization} and \eqref{moderate_2} that
\begin{align*}
\liminf_{n\to\infty}\frac1{b_n}\log\P(Y_n\in\Gamma,\,Y_n\in S_\e^q)=-\frac{(\essinf_{x\in\Gamma} \|x\|^2+\e)^2}2
\end{align*}
and the lower bound follows from letting $\e\to0$ and Lemma \ref{exact_ldp}.

Our aim is to transfer the MDPs for $Y_n^q,\,q=1,\dots,p$ to one for $Z_n$ from Theorem \ref{general_dist_mdp}. Choosing
\begin{align*}
    f^q(x): = H(\mb y^*_{l, q})\Big[\frac{\partial \mb y_l^\C}{\partial \mb m_l^\C}(\mb y^*_{l, q})\Big]^t\varphi_l^{E, \C}\Big((H(\mb y^*_{l, q}))^{-1}\cdot x\Big)
\end{align*}
we can apply the delta-method for large deviations \cite[Theorem~3.1]{gaozhao2011} to obtain an LDP for the sequence
\begin{align*}
    Z_n^q:={}&a_n\Big(f^q\big(H(\mb y^*_{l, q}) \mb y_l^{\C,(n)}\big) -  f^q\big(H(\mb y^*_{l, q}) \mb y^*_{l,q}\big)\Big) \\
    ={}& a_n H(\mb y^*_{l, q})\Big[\frac{\partial \mb y_l^\C}{\partial \mb m_l^\C}(\mb y^*_{l, q})\Big]^t \Big(\mathbf{m}_l^{\C, (n)} - \varphi_l^{E, \C}(\mb y_{l, q}^*)\Big)
\end{align*}
under $\P^q$, with good rate function $I$ and speed $b_n$. The next step is to extend the MDP to
\begin{align*}
    Z_n' := a_n H(\mb y^*_{l})\Big[\frac{\partial \mb y_l^\C}{\partial \mb m_l^\C}(\mb y^*_{l})\Big]^t \Big(\mathbf{m}_l^{\C, (n)} - \varphi_l^{E, \C}(\mb y_{l}^*)\Big)
\end{align*}
under $\P$. In order to show this LDP, we estimate for an arbitrary measurable set $\Gamma$
\begin{align*}
c_1\sum \limits_{q = 1}^p\P^q(Z_n^q \in \Gamma)\leq\P\Big(Z_n' \in \Gamma\Big)
\leq c_2\sum \limits_{q = 1}^p\P^q(Z_n^q \in \Gamma)\,.
\end{align*}
Therefore the LDP for $Z_n'$ holds by~\eqref{log_inequality}. The last step in the proof is to augment $Z_n'$ to $Z_n$ by inserting zeros at at positions $i_1,\dots,i_k$. Let $T$ denote this operation. The MDP for the sequence $Z_n$ then follows from observing that by Theorem~\ref{general_dist_ldp} we have 
\begin{align*}
    \limsup \limits_{n \to \infty} \frac{1}{b_n} \log \P(Z_n \ne T(Z_n')) \le \limsup \limits_{n \to \infty} \frac{1}{b_n} \log \P(\varphi_l^{E, \C}(\mb y_l^*) \ne \mb m_l^{*, \C}) = -\infty\,,
\end{align*}
where $\mb m_l^{*, \C}$ denotes the vector of unconstrained moments in $\m_l^*$. This proves the asymptotic equivalence of $T(Z_n')$ and $Z_n$ and concludes the proof.

\end{proof}

    % Section
    
    % Section

    % Section
\section{Proofs of the results in the uniform case}\label{sec_application_uniform}
    
    The uniform distribution on the moment space $\mathcal{M}_n^{\mc C}([0, 1])$ is a special case of the distribution $\mathbb{P}_{n, [0,1], V}^{\mc C}$ defined in Section~\ref{sec_general_dist}. It is however difficult to prove genericity of $V_1=\dots=V_{i_k+2}\equiv0$ for \textit{any} constraint $\C$ instead of just generic constraints. We will therefore consider a special parametrization of $\mathcal{M}_n^{\mc C}([0, 1])$ that allows to exploit concavity properties of the density of the canonical moments.

    \begin{lemma}\label{coordinates_uniform}
        Recall the definition of $\m_{i_k}^\C$ in \eqref{m_n^C} and denote by $\tilde{\mathcal{M}}_l^{\mc C}([0, 1])$ the projection of $\mathcal{M}_{l}^{\mc C}([0, 1])$ to the coordinates in $\mathbf{m}_{l}^{\mc C}$. For any $n \ge i_k$ the mapping
        \begin{align}
            \varphi_n^{\mc C} :\left\{
            \begin{array}{ccc}
                \inn \tilde{\mathcal{M}}_{i_k}^{\mc C} \times (0, 1)^{n - i_k}  & \to       & \inn \tilde{\mathcal{M}}_n^{\mc C}(E)       \\
                \big(\mathbf{m}_{i_k}^{\mc C}, p_{i_k + 1}, \ldots, p_n\big)    & \mapsto   & \mathbf{m}_n^\C
            \end{array} \right. \label{coord_map_uniform}
        \end{align}
        is a $C^\infty$-diffeomorphism with Jacobian
        \begin{align*}
            \det \varphi_n^{\mc C} = \prod \limits_{j = 1}^n (p_j(1 - p_j))^{n - \max(j, i_k)}\,.
        \end{align*}
    \end{lemma}
    \begin{proof}
        The Jacobian matrix of $\varphi_n^{\mc C}$ is lower triangular with determinant
        \begin{align*}
            \det \varphi_n^{\mc C} = \prod\limits_{\substack{l = 1\\ l\not=i_1,\dots,i_k}}^{i_k} \frac{\partial m_l}{\partial m_l} \prod \limits_{l = i_k + 1}^n \frac{\partial m_l}{\partial p_l} = \prod \limits_{l = i_k + 1}^{n} \prod \limits_{j = 1}^{l - 1} p_j(1 - p_j) = \prod \limits_{j = 1}^n (p_j(1 - p_j))^{n - \max(j, i_k)} \,,
        \end{align*}
				where we used the previously applied formula \eqref{moment_range}  again.
    \end{proof}

    In order to identify the limiting measures we need a specialization of Proposition~\ref{measure_representation}.
    \begin{cor} \label{measure_representation_uniform}
        Let $p_1, \ldots, p_r \in (0, 1)$. The probability measure $\mu$ on $[0,1]$ corresponding to the infinite sequence of canonical moments $p_1, \ldots, p_r, \frac{1}{2}, \frac{1}{2}, \ldots$ is absolutely continuous with density
        \begin{align*}
            \frac{\prod \limits_{i = 1}^r p_i q_i}{\pi R_r(x)\sqrt{x - x^2} },
        \end{align*}
        where $R_r(x)$ is a polynomial of degree at most $r$ that is strictly positive on the interval $[0, 1]$. The coefficient of $x^r$ in $R_r$ is $1 - 2p_r$ and all coefficients depend continuously on $p_1, \ldots, p_r$.
        Furthermore, $R_r$ can be expressed as
        \begin{align*}
            R_r(x) = \bigg(\frac{\big(P_j(x) - \frac{1}{2}(x - \frac{1}{2}) P_{j - 1}(x))^2}{x(1-x)} + \frac{1}{4} P_{j - 1}^2(x)\bigg)\cdot
            \begin{cases}
                16 &, \text{ if $r$ is even} \\
                4 &, \text{ if $r$ is odd} \\
            \end{cases},
        \end{align*}
        where $P_1, P_2, \ldots$ are the monic orthogonal polynomials corresponding to $\mu$ and $j:= \Big\lfloor\frac{r + 4}{2}\Big\rfloor$.
    \end{cor}
    \begin{proof}
        By \eqref{canonical_recursion} we have $\alpha_l = \alpha = \frac{1}{2}$ for all $l > \frac{r + 3}{2}$ and $\beta_{l - 1} = \beta = \frac{1}{16}$ for all $l > \frac{r + 4}{2}$. By Proposition~\ref{measure_representation} the measure $\mu$ therefore admits a density on the interval $[0, 1]$ that is given by
        \begin{align*}
            \frac{\sqrt{x(1 - x)} \beta_1 \cdots \beta_{j - 1}}{2 \pi D(x)},
        \end{align*}
        where $j: = \Big\lfloor\frac{r + 4}{2}\Big\rfloor$ and
        \begin{align*}
            D(x) = \Big(P_j(x) - \frac{1}{2}(x - \frac{1}{2}) P_{j - 1}(x)\Big)^2 + \frac{1}{4}x(1 - x)P_{j - 1}^2(x) \,.
        \end{align*}
        Considering separately the cases where $r$ is odd and even, respectively, we obtain from \eqref{canonical_recursion}
        \begin{align*}
            \beta_1 \cdots \beta_{j - 1} = p_{2j - 2} \prod \limits_{i = 1}^{2j - 3} p_i q_i = \prod \limits_{i = 1}^r p_i q_i \cdot
            \begin{cases}
                \frac{1}{8} &, r \text{ even} \\
                \frac{1}{2} &, r \text{ odd} \\
            \end{cases}\,.
        \end{align*}
        With the same arguments as in the proof of Theorem~\ref{general_lln} we find that $D$ has zeros in $0$ and $1$ and the assertion follows from setting
        \begin{align*}
            R_r(x) := \frac{D(x)}{x(1 - x)} \cdot 
            \begin{cases}
                16 &, r \text{ even} \\
                4 &, r \text{ odd}
            \end{cases}.
        \end{align*}
        It remains to see that the degree of $R_r$ is at most $r$ and that the coefficient of $x^r$ is $1 - 2p_r$.        
        If $r$ is odd, then by \eqref{expansion_d} the degree of $R_r(x)$ is at most $2j - 3 = r$ and the coefficient of $x^{r}$ in $R_r(x)$ is given by 
        \begin{align*}
            4\lb\alpha_j - \frac12\rb = 4\left(q_{r} p_{r + 1} + q_{r + 1} p_{r + 2} - \frac{1}{2}\right) = 1 - 2p_r.
        \end{align*}
        If $r$ is even, then by \eqref{expansion_d} the degree of $R_r(x)$ is at most $2j - 3 = r + 1$ and the coefficient of $x^{r + 1}$ in $R_r(x)$ is given by
        \begin{align*}
            16\lb\alpha_j - \frac12\rb= 16\left(q_{r + 1} p_{r + 2} + q_{r + 2} p_{r + 3} - \frac{1}{2}\right) = 0 \,.
        \end{align*}
        Consequently, the degree of $D(x)$ is at most $r$ and the coefficient of $x^{r}$ is given by
        \begin{align*}
            16\lb\beta_{j - 1} - \frac1{16}\rb = 16\left(q_r p_{r + 1} q_{r + 1} p_{r + 2} - \frac{1}{16}\right) = 1 - 2p_r \,.
        \end{align*}
    \end{proof}
    
    A key observation is contained in the following
    \begin{lemma}\label{concave}
        The map
        \begin{align}
            W_n:\left\{
            \begin{array}{ccc}
                \inn \mathcal{M}_n([0, 1])  & \to       & \mathbb{R} \\
                (m_1, \ldots, m_n)  & \mapsto   & \log(4^n(m_{n + 1}^+ - m_{n + 1}^-))
            \end{array}
            \right.\label{moment_range_map}
        \end{align}
        is strictly concave and the Hessian matrix is negative definite in every point.
    \end{lemma}
    \begin{proof}
        Recall that a function $f: \Omega \to \mathbb{R}$ on a convex set $\Omega$ is called \emph{strongly concave with modulus $c > 0$} if
        \begin{align*}
            f(\lambda x + (1 - \lambda) y) \ge \lambda f(x) + (1 - \lambda) f(y) + \frac{c}{2} \lambda (1 - \lambda)\|x - y\|_2^2
        \end{align*}
        holds for all $x, y \in \Omega$ and $\lambda \in [0, 1]$. If $\Omega \subset \mathbb{R}^n$ and $f$ is twice differentiable, then this is equivalent to $\operatorname{Hess}f(x) \le -c I_n$ for all $x \in \Omega$, cf.~\cite[Theorem~4.3.1]{hirbap2001}). We will show that for each point $m \in \inn \mathcal{M}_n([0, 1])$ there is a neighborhood $U$ in which $W_n$ is strongly concave with a positive modulus. Then by the mentioned equivalence, negative definiteness of $\operatorname{Hess}W_n(m)$ will follow.

        Let $m, m' \in \inn \mathcal{M}_n([0, 1])$ be two moment vectors in the interior of the $n$-th moment space. Let $\mu$ and $\mu'$ denote the probability measures with first $n$ moments given by $m$ and $m'$, respectively and whose canonical moments $p_i$ are $\frac{1}{2}$ for $i > n$. By Corollary~\ref{measure_representation_uniform}, $\mu$ and $\mu'$ admit densities $f$ and $g$ w.r.t.~the arcsine measure $\mu^0$. To proceed we use the remarkable identity
        \begin{align}
            \mathcal{K}(\mu^0 \mid \mu) = -\sum \limits_{i = 1}^\infty \log(4 p_i(1 - p_i)) \label{sumrule}
        \end{align}
         between a measure $\mu$ on $[0, 1]$ and its corresponding sequence of canonical moments.  Identity \eqref{sumrule} is a so-called sum rule, given in \cite[p. 523]{GNR1}. It can be deduced from Szeg\H{o}'s theorem (\cite{GNR1}), for an interesting historical survey we refer to \cite[p.29]{simon2011}.
        
By \eqref{moment_range} and     \eqref{sumrule}, $W_n$ can be written in terms of the Kullback-Leibler divergence as
        \begin{align*}
            W_n(m) &= \sum \limits_{i = 1}^n \log(4p_iq_i) = \sum \limits_{i = 1}^\infty \log(4p_iq_i) = -\mathcal{K}(\mu^0 \mid \mu), \\
            W_n(m') &= \sum \limits_{i = 1}^n \log(4p_iq_i) = \sum \limits_{i = 1}^\infty \log(4p_iq_i) = -\mathcal{K}(\mu^0 \mid \mu').
        \end{align*}
        Let $\lambda \in [0, 1]$ be arbitrary and denote by $\nu$ the measure whose first $n$ moments are given by $\lambda m + (1 - \lambda)m'$ and which has canonical moments $p_i = \frac{1}{2}$ for $i > n$. Since $\nu$ and $\lambda \mu + (1 - \lambda) \mu'$ agree on the first $n$ moments, they also have the same first $n$ canonical moments. This yields
        \begin{align*}
            W_n(\lambda m + (1 - \lambda) m') = -\mathcal{K}(\mu^0 | \nu) = \sum \limits_{i = 1}^\infty \log(4p_i(1 - p_i)) \ge -\mathcal{K}(\mu^0 | \lambda \mu + (1 - \lambda) \mu').
        \end{align*}
				To see the last inequality, we observe that the canonical moment $p_i$ of the measure $\lambda \mu + (1 - \lambda) \mu'$ is for $i>n$ in general not $\frac12$ and $\log (4p_i(1-p_i))\leq0$ with equality only for $p_i=\frac12$.
        
        Now note that for any constant $c > 0$ the function $x \mapsto \log(x)$ is strongly concave on the interval $(0, c]$ with modulus $\frac{1}{c^2}$, i.e. for all $x, y \in (0, c]$ we have
        \begin{align*}
            \log(\lambda x + (1 - \lambda) y) \ge \lambda \log(x) + (1 - \lambda) \log(y) + \frac{\lambda (1 - \lambda)}{2c^2} (x - y)^2.
        \end{align*}
        This yields
        \begin{align*}
            &W_n(\lambda m + (1 - \lambda)m') \ge -\mathcal{K}(\mu^0 | \lambda \mu + (1 - \lambda)\mu') = \int_0^1 \frac{\log(\lambda f(x) + (1 - \lambda) g(x))}{\pi\sqrt{x-x^2}} \, dx \\
            \ge{}& \lambda W_n(m) + (1 - \lambda) W_n(m') + \lambda(1 - \lambda) \cdot \frac{1}{2} \int_0^1 \frac{(f(x) - g(x))^2}{\max\{ f^2(x), g^2(x)\}\pi \sqrt{x-x^2}} \, dx\,.
        \end{align*}
        
        By Hölder's inequality we have
        \begin{align*}
            &\left(\int_0^1 \frac{x^m(f(x) - g(x))}{\pi\sqrt{x-x^2}} \, dx\right)^2 \le \left(\int_0^1 \frac{|f(x) - g(x)|}{\max\{f(x), g(x)\} \sqrt{\pi}\sqrt[4]{x-x^2}}  \cdot \frac{\max\{f(x), g(x)\}}{\sqrt{\pi}\sqrt[4]{x-x^2}}dx\right)^2 \\
            \le{}& \int_0^1 \frac{(f(x) - g(x))^2}{\max\{f^2(x), g^2(x)\} \pi \sqrt{x-x^2}} \, dx \int_0^1 \frac{\max\{f^2(x), g^2(x)\}}{\pi \sqrt{x-x^2}} \, dx.
        \end{align*}
        
        If $f_\lambda$ denotes the density of $\mu$ with respect to the Lebesgue-measure, then
        \begin{align*}
            f_\lambda(x) = \frac{d\mu}{d\lambda}(x) = \frac{d\mu}{d\mu^0}(x) \cdot \frac{d\mu^0}{d\lambda}(x) = \frac{f(x)}{\pi \sqrt{x-x^2}}.
        \end{align*}
        This results in
        \begin{align*}
            \int_0^1 \frac{f^2(x)}{\pi \sqrt{x-x^2}} \, dx = \int_0^1 (\pi \sqrt{x-x^2} f_\lambda(x)) \cdot f_\lambda(x) \, dx \le \int_0^1 \frac{1}{R_n(x)} f_\lambda(x) \, dx,
        \end{align*}
				where $R_n$ is the polynomial from Corollary \ref{measure_representation_uniform}.
        By the same corollary, we know that $R_n(x)$ is a strictly positive polynomial with coefficients depending on $p_1, \ldots, p_n$ in a continuous way and therefore locally in $p_1, \ldots, p_n$ we can bound $R_n(x)$ from below by some constant $K_n$ depending on $n$ and $p_1, \ldots, p_n$. Since $f_\lambda$ is a density, we have
        \begin{align*}
            \int_0^1 \frac{1}{R_n(x)} f_\lambda(x) \, dx \le \frac{1}{K_n}.
        \end{align*}
        Combining these inequalities we get
        \begin{align*}
            W_n(\lambda m + (1 - \lambda)m') \ge \lambda W_n(m) + (1 - \lambda) W_n(m') + \lambda(1 - \lambda) \cdot \frac{1}{2K_nn} \|m - m'\|_2^2,
        \end{align*}
        i.e. $W_n$ is (locally) strongly concave.
    \end{proof}
    
    Before giving the proofs of the main results of Section \ref{sec_uniform}, a few words on the logic structure of the proofs are in order. We start with the LDP which will be used in the proof of the law of large numbers (Theorem \ref{uniform_lln}) lateron. The limiting measure $\mu^\C$ was introduced in Section \ref{sec_uniform} in Theorem \ref{uniform_lln} but as we already need it in the proof of the LDP, we will define it as the unique minimizer of the Kullback-Leibler divergence $\mc K(\mu^0|\cdot)$ over $\mc P^\C([0,1])$. In particular, in the following proof we show that this minimization problem indeed has a unique solution. The convergence of the random moments to the moments of the minimizing measure and the representation of the minimizer will then be given in the proof of Theorem \ref{uniform_lln}.
    \begin{proof}[Proof of Theorem~\ref{uniform_ldp}]
        Applying the parametrization of Lemma~\ref{coordinates_uniform} to $(m_1^{(n)}, \ldots, m_n^{(n)})$, we obtain a random vector $(\mathbf{m}_{i_k}^{\mc C,(n)}, p_{i_k + 1}^{(n)}, \ldots, p_n^{(n)})$, where $p_j^{(n)}$ is beta$(n-j+1,n-j+1)$-distributed, the density of $\mathbf{m}_{i_k}^{\mc C,(n)}$ is proportional to
        \begin{align}
            \exp(-(n - i_k) \sum \limits_{i = 1}^{i_k} \log(p_i(1 - p_i))) = \exp(-(n - i_k) \log(m_{i_k + 1}^+ - m_{i_k + 1}^-)),\label{density_first}
        \end{align}
        and we have used formula~\eqref{moment_range} for the last identity. Applying Theorem~\ref{ldp} and Lemma~\ref{exact_ldp}, we obtain an LDP for the coordiates $(\mathbf{m}_{i_k}^{\mc C,(n)}, p_{i_k + 1}^{(n)}, \ldots, p_n^{(n)})$ with good rate function
        \begin{align}\label{equation_K}
            I_1(\mathbf{m}_{i_k}^{\mc C}, p_{i_k + 1}, \ldots, p_l) := -\log(m_{i_k + 1}^+ - m_{i_k + 1}^-) - \sum \limits_{j = i_k + 1}^l \log(p_j(1 - p_j)) + K,
        \end{align}
        where $K$ is a constant that guarantees the infimum of $I_1$ being zero. By the contraction principle and equation~\eqref{moment_range} the vector $(m_1^{(n)}, \ldots, m_l^{(n)})$ therefore satisfies an LDP with good rate function
        \begin{align}
            I_2(m) := I(m) + K\,,\label{def_I_2}
        \end{align}
        where $I$ is defined in \eqref{idef}. From \eqref{equation_K} we see that 
				\begin{align}
				K=\log(m_{i_k + 1}^+(\mu^*) - m_{i_k + 1}^-(\mu^*))-(l-i_k)\log4,\label{equation_K_2}
				\end{align}
				where $\mu^*$ is a measure with first $i_k$ canonical moments $p_1(\mu^*),\dots,p_{i_k}(\mu^*)$ minimizing 
				\begin{align*}
				(p_1,\dots,p_{i_k})\mapsto -\log(m_{i_k + 1}^+(\mu^*) - m_{i_k + 1}^-(\mu^*))
				\end{align*}
				and higher canonical moments $p_j=\frac12$ for all $j>i_k$. It remains to show that there is a unique such $\mu^*$ and that $\mu^*=\mu^\C$, where 
				\begin{align}
				\mu^\C:=\argmin_{\mu\in\mc P^\C([0,1])}\mc K(\mu^0|\mu).\label{def_mu}
				\end{align}
				For the uniqueness, note that since $\m_{i_k}\mapsto-\log(m_{i_k+1}^+-m_{i_k+1}^-)$ is $+\infty$ on the boundary of $\mc M_{i_k}^\C([0,1])$ and strictly convex on the interior by Lemma \ref{moment_range_map}, it has a unique minimizer which is attained on $\inn\mc M_{i_k}^\C([0,1])$. To see $\mu^*=\mu^\C$, recall that whether a measure $\mu\in\mc P([0,1])$ fulfills constraint $\C$ or not, is determined by the first $i_k$ canonical moments $p_1,\dots,p_{i_k}$ only. Thus by \eqref{sumrule}, a measure $\mu\in\mc P^\C([0,1])$ minimizing $\mc K(\mu^0|\cdot)$ has to have canonical moments $p_j=\frac12$ for all $j>i_k$ and hence
					\begin{align}
					\mc K(\mu^0|\mu)=-\sum_{j=1}^{i_k}\log(4p_j(1-p_j))=-\log(m_{i_k+1}^+(\mu)-m_{i_k+1}^-(\mu))-i_k\log4.\label{sumrule2}
					\end{align}
					This implies $\mu^*=\mu^\C$ and finishes the proof. 
    \end{proof}
    \begin{proof}[Proof of Theorem~\ref{uniform_ldp_functional}]
        By \eqref{def_I_2}, \eqref{equation_K_2}, \eqref{sumrule} and \eqref{sumrule2}, for any $l$ the sequence $(m_1(\mu^{(n)}), \ldots, m_l(\mu^{(n)}))$ satisfies an LDP with good rate function
        \begin{align*}
            I(\m_l) = \mathcal{K}(\mu^0 \mid \mu(\m_l)) - \mathcal{K}(\mu^0 \mid \mu^{\mc C})\,,
        \end{align*}
        where $\m_l$ is an $l$-dimensional vector of moments and $\mu(\m_l)$ is the measure that has $\m_l$ as its first $l$ moments and canonical moments $p_j = \frac{1}{2}$ for all $j > l$. By the Dawson-Gärtner theorem \cite[Theorem 4.6.1]{demzeit1998} we then obtain an LDP for the infinite sequence of moments $(m_1(\mu^{(n)}), m_2(\mu^{(n)}), \ldots)$ with good rate function
        \begin{align*}
            I(\m_\infty) = \sup \limits_{l \ge 0} \mathcal{K}(\mu^0 \mid \mu(\m_l)) - \mathcal{K}(\mu^0 \mid \mu^{\mc C})\,,
        \end{align*}
        where $\m_\infty$ is the infinite vector of moments containing the $\m_l$'s. By \eqref{sumrule} the term
        \begin{align*}
            \mathcal{K}(\mu^0 \mid \mu(\m_l)) = -\sum \limits_{j = 1}^l \log(4p_iq_i)
        \end{align*}
        is increasing in $l$ and we therefore obtain
        \begin{align*}
            I(\m_\infty) = \sup \limits_{l \ge 0} \mathcal{K}(\mu^0 \mid \mu(\m_l)) - \mathcal{K}(\mu^0 \mid \mu^{\mathcal{C}}) = \mathcal{K}(\mu^0 \mid \mu(\m_\infty)) - \mathcal{K}(\mu^0 \mid \mu^{\mathcal{C}}),
        \end{align*}
        where $\mu(\m_\infty)$ is the measure with moments $\m_\infty$. The theorem now follows from applying the contraction principle to transfer the LDP from the sequence $(m_1(\mu^{(n)}), m_2(\mu^{(n)}), \ldots)$ to the measure $\mu^{(n)}$. This is possible since $\mu^{(n)}$ is compactly supported and hence uniquely determined by its moments.
    \end{proof}
    \begin{proof}[Proof of Theorem~\ref{uniform_lln}]
		Recall from the proof of Theorem~\ref{uniform_ldp} that $\m_l^{(n)}$ satisfies an LDP with good rate function $I-I(\m_l(\mu^\C))$, where $\mu^\C$ has been defined as the unique minimizer of $\mc K(\mu^0|\cdot)$ over $\mc P^\C([0,1])$. The convergence stated in Theorem~\ref{uniform_lln} now follows from an application of the Borel-Cantelli lemma.
	The claimed representation of $\mu^{\mc C}$ is then provided by Corollary~\ref{measure_representation_uniform}.
    \end{proof}
    Next, we give the computation of the measures in Example~\ref{example_uniform}.
    \begin{example}[Continuation of Example~\ref{example_uniform}]
        We will consider the different constraints $\mathcal{C}_1$, $\mathcal{C}_2$ and $\mathcal{C}_3$ separately. \medskip
        
        \textbf{Case $\mathcal{C}_1 = \{m_1 = c_1\}$:} As seen in the proof of Theorem~\ref{uniform_lln}, the measure $\mu^{\mathcal{C}_1}$ is uniquely determined by having canonical moments $p_1 = c_1$ and $p_j = \frac{1}{2}$ for all $j > 1$. We may thus apply Corollary~\ref{measure_representation_uniform} to obtain the representation of $\mu^{\mathcal{C}_1}$. By formula~\eqref{canonical_recursion} the recursion parameters of $\mu^{\mathcal{C}_1}$ are given by $\alpha_1 = c_1$, $\alpha_2 = \frac{3 - 2c_1}{4}$, $\beta_1 = \frac{c_1(1 - c_1)}{2}$ and $\alpha_j = \frac{1}{2}$, $\beta_{j - 1} = \frac{1}{16}$ for all $j > 2$. The orthogonal polynomials of $\mu^{\mathcal{C}_1}$ up to order 2 can hence be calculated as
        \begin{align*}
            P_0(x) &= 1 \\
            P_1(x) &= x - c_1 \\
            P_2(x) &= (x - \alpha_2) P_1(x) - \beta_1 P_0(x) = x^2 - \frac{3 + 2c_1}{4} x + \frac{c_1}{4}
        \end{align*}
        and we obtain
        \begin{align*}
            R_1(x) = 4 \left(\frac{P_2(x) - \frac{1}{2}(x - \frac{1}{2}) P_1(x))^2}{x(1 - x)} + \frac{1}{4} P_1^2(x)\right) = x(1 - x) + (x - c_1)^2\,.
        \end{align*}
        The stated form of the measure now follows from Corollary~\ref{measure_representation_uniform}. \medskip
        
        \textbf{Case $\mathcal{C}_2 = \{m_1 = c_1, m_2 = c_2\}$:} As in the first case, the measure $\mu^{\mathcal{C}_2}$ is uniquely determined by having canonical moments $p_1 = c_1$, $p_2 = \frac{m_2 - m_2^-}{m_2^+ - m_2^-} = \frac{c_2 - c_1^2}{c_1(1 - c_1)}$ and $p_j = \frac{1}{2}$ for all $j > 2$. Analogously to the first case we calculate
        \begin{align*}
            R_2(x) = (1 - p_2)^2(x - p_1)^2 + p_2^2 x(1 - x) = \frac{(c_1 - c_2)^2(x - c_1)^2 + (c_2 - c_1^2)^2 x(1 - x)}{c_1^2(1 - c_1)^2}
        \end{align*}
        and the representation of the measure follows again from Corollary~\ref{measure_representation_uniform}. Note here that
        \begin{align*}
            p_1q_1 p_2q_2 =c_1(1 - c_1) \frac{c_2 - c_1^2}{c_1(1 - c_1)} \frac{c_1 - c_2}{c_1(1 - c_1)} = \frac{(c_2 - c_1^2)(c_1 - c_2)}{c_1(1 - c_1)}.
        \end{align*}
        \medskip
        
        \textbf{Case $\mathcal{C}_3 = \{m_2 = c_2\}$:} As seen in the proof of Theorem~\ref{uniform_lln}, the moment $m_1$ of $\mu^{\mathcal{C}_3}$ maximizes the moments range $m_3^+ - m_3^-$ under all measures satisfying $\mathcal{C}_3$. Recalling~\eqref{moment_range}, we have for any measure with $m_2 = c_2$
        \begin{align*}
            m_3^+ - m_3^- = p_1q_1p_2q_2 = \frac{(c_2 - m_1^2)(m_1 - c_2)}{m_1(1 - m_1)}\,,
        \end{align*}
        which (as a function of $m_1$) has a unique maximizer $m_1^*$ by Lemma~\ref{concave}. Therefore the canonical moments of $\mu^{\mathcal{C}_3}$ are given by $p_1 = m_1^*$, $p_2 = \frac{c_2 - (m_1^*)^2}{m_1^*(1 - m_1^*)}$ and $p_j = \frac{1}{2}$ for all $j > 2$. The assertion now follows from the calculations in the case $\mathcal{C}_2 = \{m_1 = m_1^*, m_2 = c_2\}$.
    \end{example}
    \begin{proof}[Proof of Theorem~\ref{uniform_clt}]
       Recall that the coordinates $(\mathbf{m}_{i_k}^{\mc C,(n)}, p_{i_k + 1}^{(n)}, \ldots, p_l^{(n)})$ have distributions $p_j^{(n)} \sim $beta$(n - j + 1, n - j + 1)$ and $\mathbf{m}_{i_k}^{\mc C,(n)}$ has a density proportional to
        \begin{align*}
            \exp(n \log(m_{i_k + 1}^+ - m_{i_k + 1}^-)) (m_{i_k + 1}^+ - m_{i_k + 1}^-)^{-i_k} 1_{\inn \mathcal{M}_{i_k}^{\mc C}([0, 1])}(m) \,.
        \end{align*}
        With Lemma~\ref{concave} the theorem follows by an application of Proposition~\ref{conv_normal} and the delta-method analogously to the proof of Theorem \ref{general_dist_clt}. Note that the minimizer of the map $m \mapsto -\log(m_{i_k + 1}^+ - m_{i_k + 1}^-)$ is unique and given by the moments $(m_1(\mu^{\mc C}), \ldots, m_{i_k}(\mu^{\mc C}))$ by the same arguments as in the proof of Theorem~\ref{uniform_lln}.
    The covariance matrix $\Sigma_l$ is given as
    \begin{gather}
        \Sigma_l:=T^{\mc C}\tilde{\Sigma}_l\,,\label{Sigma_l}\\
        \tilde{\Sigma}_l:= \mathcal{D}^t \operatorname{diag}\Big(\operatorname{Hess}^{\mc C}(W_{i_k})\big(\mathbf{m}_{i_k}(\mu^{\mc C})\big), \frac{1}{8}, \ldots, \frac{1}{8}\Big) \mathcal{D}\,,\notag \\
        \mathcal{D} := D\varphi_l^{\mc C}\Big(\mathbf{m}_{i_k}(\mu^{\mc C}), \frac{1}{2},\ldots, \frac{1}{2}\Big) \,, \notag
    \end{gather}
    where $T^{\mc C}:\R^{(l-{k})\times(l-k)}\to\R^{l\times l}$ denotes the transformation of an $(l-{k})\times(l-{k})$ matrix by inserting rows and columns of zeros at the positions $i_1,\dots,i_k$. The maps $\varphi_l^{\mc C}$ and $W_{i_k}$ are defined in~\eqref{coord_map_uniform} and~\eqref{moment_range_map} and $\operatorname{Hess}^{\mc C}$ denotes the Hessian with respect to the coordinates in $\mathbf{m}_{i_k}^{\mc C}$.
    \end{proof}
    \begin{proof}[Proof of Theorem~\ref{uniform_mdp}]
        The proof follows by an application of Proposition~\ref{mdp} analogous to the proof of Theorem~\ref{uniform_clt}. Note that the delta-method for large deviations must be applied similar to the proof of Theorem~\ref{general_dist_mdp}.
    \end{proof}
    
    \begin{proof}[Proof of Proposition~\ref{volume}]
        By Lemma~\ref{coordinates_uniform} and \eqref{density_first}, the volume of the restricted space is given by
        \begin{align*}
            &\operatorname{vol}_{n-k}\Big(\mathcal{M}_n^{\mc C}([0, 1])\Big) \\
            ={}& \int \limits_{\mathcal{M}_{i_k}^{\mc C}([0, 1])} \exp\Big(-(n - i_k) \log(m_{i_k + 1}^+ - m_{i_k + 1}^-)\Big) \, dm \prod \limits_{j = i_k + 1}^n B(n - j + 1, n - j + 1)\,,
        \end{align*}
        where $B(a, b) := \int_0^1 x^{a - 1} (1 - x)^{b - 1} \, dx$ is the beta function. By an application of Laplace's method and the minimization property of $\mu^C$, the first factor satisfies
        \begin{align*}
            &\int \limits_{\mathcal{M}_{i_k}^{\mc C}([0, 1])} \exp\Big((n - i_k) \log(m_{i_k + 1}^+ - m_{i_k + 1}^-)\Big) \, dm \\
            ={}& n^{-(i_k - k)/2} \Big((m_{i_k + 1}^+ - m_{i_k + 1}^-)(\mu^{\mc C})\Big)^{n - i_k} \left(\frac{(2\pi)^{(i_k - k)/2}}{\sqrt{d^{\mc C}}} + o(1)\right) \,,
        \end{align*}
        where the constant $d^{\mc C}$ is the determinant of the Hessian of the map $m \mapsto -\log(m_{i_k + 1}^+ - m_{i_k + 1}^-)$ with respect to the coordinates $\mathbf{m}_{i_k}^{\mc C}$, evaluated in $(m_1(\mu^{\mc C}), \ldots, m_{i_k}(\mu^{\mc C}))$.
        
        For the second factor, observe that by~\eqref{volume_momspace}
        \begin{align*}
            &\prod \limits_{j = i_k + 1}^n B(n - j + 1, n - j + 1) = \operatorname{vol}_n\big(\mathcal{M}_n([0, 1])\big) \prod \limits_{j = 1}^{i_k} B(n - j + 1, n - j + 1)^{-1}
        \end{align*}
        holds and another application of Laplace's method yields
        \begin{align*}
            B(n - j + 1, n - j + 1) = n^{-1/2}4^{-(n-j)} \Big(\frac{\sqrt{\pi}}{2} + o(1)\Big) \,.
        \end{align*}
        Combining these equations then gives
        \begin{align*}
            \frac{\operatorname{vol}_{n - k}\big(\mathcal{M}_n^{\mc C}([0, 1])\big)}{\operatorname{vol}_n\big(\mathcal{M}_n([0, 1])\big)} = n^{k/2} \Big(4^{i_k}(m_{i_k + 1}^+ - m_{i_k + 1}^-)(\mu^{\mc C})\Big)^{n - i_k}\left(\frac{2^{i_k^2 + (i_k - k)/2}}{\sqrt{d^{\mc C}} \pi^{k/2}} + o(1)\right).
        \end{align*}
    \end{proof}

    {\bf Acknowledgements.} 
    The authors would like to thank Fabrice Gamboa for useful discussions.
    The work of  H. Dette and D. Tomecki was supported by the Deutsche Forschungsgemeinschaft (DFG Research Unit 1735, DE 502/26-2, RTG
    2131:  High-dimensional Phenomena in Probability - Fluctuations and Discontinuity). 
    The work of M. Venker was supported by  the European Research Council under the European Unions Seventh
    Framework Programme (FP/2007/2013)/ ERC Grant Agreement n.  307074.
    \bigskip
    
        % Bibliography
    \bibliographystyle{apalike}
    \bibliography{quellen}

\begin{thebibliography}{}

\bibitem[Akemann et~al., 2011]{Handbook}
Akemann, G., Baik, J., and Di~Francesco, P., editors (2011).
\newblock {\em The {O}xford handbook of random matrix theory}.
\newblock Oxford University Press, Oxford.

\bibitem[Anderson et~al., 2010]{AGZ}
Anderson, G., Guionnet, A., and Zeitouni, O. (2010).
\newblock {\em An introduction to random matrices}, volume 118 of {\em
  Cambridge Studies in Advanced Mathematics}.
\newblock Cambridge University Press, Cambridge.

\bibitem[Chang et~al., 1993]{chakemstu1993}
Chang, F.~C., Kemperman, J. H.~B., and Studden, W.~J. (1993).
\newblock A normal limit theorem for moment sequences.
\newblock {\em Ann. Probab.}, 21(3):1295--1309.

\bibitem[Chihara, 1978]{chihara1978}
Chihara, T.~S. (1978).
\newblock {\em An introduction to orthogonal polynomials}.
\newblock Gordon and Breach Science Publishers, New York-London-Paris.
\newblock Mathematics and its Applications, Vol. 13.

\bibitem[Csisz\'ar, 1967]{Csiszar}
Csisz\'ar, I. (1967).
\newblock Information-type measures of difference of probability distributions
  and indirect observations.
\newblock {\em Studia Sci. Math. Hungar.}, 2:299--318.

\bibitem[Deift et~al., 1998]{Deiftetal99}
Deift, P., Kriecherbauer, T., and McLaughlin, K. T.-R. (1998).
\newblock New results on the equilibrium measure for logarithmic potentials in
  the presence of an external field.
\newblock {\em J. Approx. Theory}, 95(3):388--475.

\bibitem[Deift, 1999]{Deift}
Deift, P.~A. (1999).
\newblock {\em Orthogonal polynomials and random matrices: a
  {R}iemann-{H}ilbert approach}, volume~3 of {\em Courant Lecture Notes in
  Mathematics}.
\newblock New York University, Courant Institute of Mathematical Sciences, New
  York; American Mathematical Society, Providence, RI.

\bibitem[Dembo and Zeitouni, 1998]{demzeit1998}
Dembo, A. and Zeitouni, O. (1998).
\newblock {\em Large deviations techniques and applications}, volume~38 of {\em
  Applications of Mathematics (New York)}.
\newblock Springer-Verlag, New York, second edition.

\bibitem[Dette and Nagel, 2012]{detnag2012}
Dette, H. and Nagel, J. (2012).
\newblock Distributions on unbounded moment spaces and random moment sequences.
\newblock {\em Ann. Probab.}, 40(6):2690--2704.

\bibitem[Dette and Studden, 1997]{dettstud1997}
Dette, H. and Studden, W.~J. (1997).
\newblock {\em The theory of canonical moments with applications in statistics,
  probability, and analysis}.
\newblock Wiley Series in Probability and Statistics: Applied Probability and
  Statistics. John Wiley \& Sons, Inc., New York.
\newblock A Wiley-Interscience Publication.

\bibitem[Dette et~al., 2018]{DTV}
Dette, H., Tomecki, D., and Venker, M. (2018).
\newblock Universality in random moment problems.
\newblock {\em Electron. J. Probab.}, 23:23 pp.

\bibitem[Franklin, 1968]{frank1968}
Franklin, J.~N. (1968).
\newblock {\em Matrix theory}.
\newblock Prentice-Hall, Inc., Englewood Cliffs, N.J.

\bibitem[Gamboa and Lozada-Chang, 2004]{gamloz2004}
Gamboa, F. and Lozada-Chang, L.-V. (2004).
\newblock Large deviations for random power moment problem.
\newblock {\em Ann. Probab.}, 32(3B):2819--2837.

\bibitem[Gamboa et~al., 2016]{GNR1}
Gamboa, F., Nagel, J., and Rouault, A. (2016).
\newblock Sum rules via large deviations.
\newblock {\em J. Funct. Anal.}, 270(2):509--559.

\bibitem[Gamboa and Rouault, 2010]{GamboaRouault}
Gamboa, F. and Rouault, A. (2010).
\newblock Canonical moments and random spectral measures.
\newblock {\em J. Theoret. Probab.}, 23(4):1015--1038.

\bibitem[Gao and Zhao, 2011]{gaozhao2011}
Gao, F. and Zhao, X. (2011).
\newblock Delta method in large deviations and moderate deviations for
  estimators.
\newblock {\em Ann. Statist.}, 39(2):1211--1240.

\bibitem[Geronimo and Iliev, 2017]{GeronimoIliev}
Geronimo, J.~S. and Iliev, P. (2017).
\newblock Bernstein-{S}zeg{\H{o}} measures, {B}anach algebras, and scattering
  theory.
\newblock {\em Trans. Amer. Math. Soc.}, 369(8):5581--5600.

\bibitem[Hiriart-Urruty and Lemar\'echal, 2001]{hirbap2001}
Hiriart-Urruty, J.-B. and Lemar\'echal, C. (2001).
\newblock {\em Fundamentals of convex analysis}.
\newblock Grundlehren Text Editions. Springer-Verlag, Berlin.
\newblock Abridged version of {{\i}t Convex analysis and minimization
  algorithms. I} [Springer, Berlin, 1993; MR1261420 (95m:90001)] and {{\i}t II}
  [ibid.; MR1295240 (95m:90002)].

\bibitem[Johansson, 1998]{Johansson98}
Johansson, K. (1998).
\newblock On fluctuations of eigenvalues of random {H}ermitian matrices.
\newblock {\em Duke Math. J.}, 91(1):151--204.

\bibitem[Karlin and Shapley, 1953]{karsha1953}
Karlin, S. and Shapley, L.~S. (1953).
\newblock Geometry of moment spaces.
\newblock {\em Mem. Amer. Math. Soc.}, No. 12:93.

\bibitem[Karlin and Studden, 1966]{karlin1966}
Karlin, S. and Studden, W.~J. (1966).
\newblock {\em Tchebycheff systems: {W}ith applications in analysis and
  statistics}.
\newblock Pure and Applied Mathematics, Vol. XV. Interscience Publishers John
  Wiley \& Sons, New York-London-Sydney.

\bibitem[Killip and Nenciu, 2004]{KillipNenciu}
Killip, R. and Nenciu, I. (2004).
\newblock Matrix models for circular ensembles.
\newblock {\em Int. Math. Res. Not.}, (50):2665--2701.

\bibitem[Klartag, 2007]{Klartag}
Klartag, B. (2007).
\newblock A central limit theorem for convex sets.
\newblock {\em Invent. Math.}, 168(1):91--131.

\bibitem[Kre{\u i}n and Nudel'man, 1977]{krenud1977}
Kre{\u i}n, M.~G. and Nudel'man, A.~A. (1977).
\newblock {\em The {M}arkov moment problem and extremal problems}.
\newblock American Mathematical Society, Providence, R.I.
\newblock Ideas and problems of P. L. \v Ceby\v sev and A. A. Markov and their
  further development, Translated from the Russian by D. Louvish, Translations
  of Mathematical Monographs, Vol. 50.

\bibitem[Matsumoto, 2002]{Matsumoto}
Matsumoto, Y. (2002).
\newblock {\em An introduction to {M}orse theory}, volume 208 of {\em
  Translations of Mathematical Monographs}.
\newblock American Mathematical Society, Providence, RI.
\newblock Translated from the 1997 Japanese original by Kiki Hudson and
  Masahico Saito, Iwanami Series in Modern Mathematics.

\bibitem[Saff and Totik, 1997]{SaffTotik}
Saff, E. and Totik, V. (1997).
\newblock {\em Logarithmic potentials with external fields}, volume 316 of {\em
  Grundlehren der Mathematischen Wissenschaften [Fundamental Principles of
  Mathematical Sciences]}.
\newblock Springer-Verlag, Berlin.
\newblock Appendix B by Thomas Bloom.

\bibitem[Simon, 2011]{simon2011}
Simon, B. (2011).
\newblock {\em Szeg\H o's theorem and its descendants}.
\newblock M. B. Porter Lectures. Princeton University Press, Princeton, NJ.
\newblock Spectral theory for $L^2$ perturbations of orthogonal polynomials.

\bibitem[Skibinsky, 1967]{skibinsky1967}
Skibinsky, M. (1967).
\newblock The range of the {$(n+1)$}th moment for distributions on {$[0,\,1]$}.
\newblock {\em J. Appl. Probability}, 4:543--552.

\bibitem[Skibinsky, 1968]{skibinsky1968}
Skibinsky, M. (1968).
\newblock Extreme {$n$}th moments for distributions on {$[0,\,1]$} and the
  inverse of a moment space map.
\newblock {\em J. Appl. Probability}, 5:693--701.

\bibitem[Skibinsky, 1969]{skibinsky1969}
Skibinsky, M. (1969).
\newblock Some striking properties of binomial and beta moments.
\newblock {\em Ann. Math. Statist.}, 40:1753--1764.

\bibitem[Skibinsky, 1986]{Skibinsky86}
Skibinsky, M. (1986).
\newblock Principal representations and canonical moment sequences for
  distributions on an interval.
\newblock {\em J. Math. Anal. Appl.}, 120(1):95--118.

\bibitem[Szeg\H{o}, 1975]{Szego}
Szeg\H{o}, G. (1975).
\newblock {\em Orthogonal polynomials}.
\newblock American Mathematical Society, Providence, R.I., fourth edition.
\newblock American Mathematical Society, Colloquium Publications, Vol. XXIII.

\end{thebibliography}
\end{document}